\documentclass[]{amsart}
\usepackage{latexsym,fancyhdr,amssymb,color,amsmath,amsthm,graphicx,listings,comment}
\usepackage[section]{placeins}
\newtheorem{thm}{Theorem} \newtheorem{lemma}{Lemma} \newtheorem{coro}{Corollary}
\let\paragraph\subsection
\def\B#1#2{{#1\choose #2}}

\title{Graph complements of circular graphs}
\author{Oliver Knill}
\date{January 17, 2020}
\address{Department of Mathematics \\ Harvard University \\ Cambridge, MA, 02138 }


\begin{document}
\maketitle

\begin{abstract}
In this report, we study graph complements $G_n$ of cyclic graphs $C_n$ or graph complements 
$G_n^+$ of path graphs. $G_n$ are circulant, vertex-transitive, claw-free,
strongly regular, Hamiltonian graphs with a $Z_n$ symmetry and Shannon capacity $2$. 
Also the Wiener and Harary index are known. The explicitly known 
adjacency matrix spectrum leads to explicit spectral zeta function and tree or forest quantities.
The forest-tree ratio of $G_n$ converge to $e$ in the limit when $n$ goes to infinity.
The graphs $G_n$ are all Cayley graphs and so Platonic 
in the sense that they all have isomorphic unit spheres $G_{n-3}^+$.  
The graphs $G_{3d+3}$ are homotop to wedge sums of two $d$-spheres and $G_{3d+2},G_{3d+4}$ 
are homotop to $d$-spheres, $G_{3d+1}^+$ are contractible, $G_{3d+2}^+,G_{3d+3}^+$ are 
homotop to $d$-spheres. Since disjoint unions are dual to Zykov joins,
graph complements of all $1$-dimensional discrete manifolds $G$ are
homotop to either a point, a sphere
or a wedge sums of spheres. If the length of every connected component of a 1-manifold
is not divisible by $3$, the graph complement of $G$ must be a sphere.
In general, the graph complement of a forest is either contractible or a sphere.
It also follows that all induced strict subgraphs of $G_n$ are either contractible or 
homotop to spheres. The f-vectors $G_n$ or $G_n^+$ satisfy a hyper Pascal triangle relation, the 
total number of simplices are hyper Fibonacci numbers. The simplex generating functions are 
Jacobsthal polynomials, generating functions of $k$-king configurations on a circular chess board. 
While the Euler curvature of circle complements $G_n$ is constant by symmetry, 
the discrete Gauss-Bonnet curvature of path complements $G_n^+$ can be expressed explicitly
from the generating functions. There is now a non-trivial $6$-periodic Gauss-Bonnet 
curvature universality in the complement of Barycentric limits. 
The Brouwer-Lefschetz fixed point theorem produces a $12$-periodicity of the 
Lefschetz numbers of all graph automorphisms of $G_n$. There is also a $12$-periodicity of Wu characteristic. 
This corresponds to $4$-periodicity in dimension as $n \to n+3$ is homotop to a suspension. 
These are all manifestations of stable homotopy features, but purely combinatorial. 
\end{abstract}

\section{Introduction}

\begin{figure}[!htpb]
\scalebox{0.4}{\includegraphics{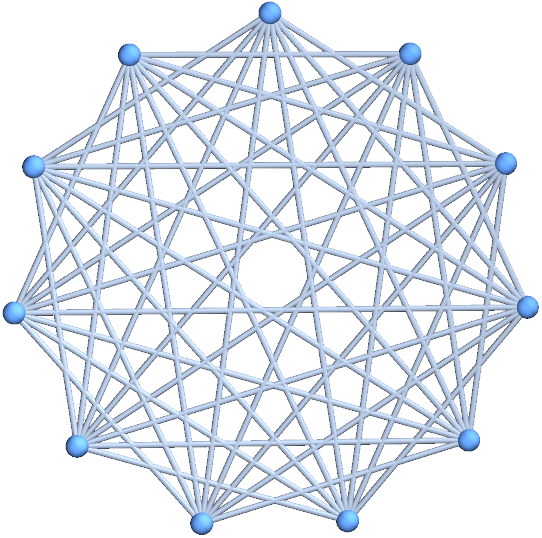}}
\scalebox{0.4}{\includegraphics{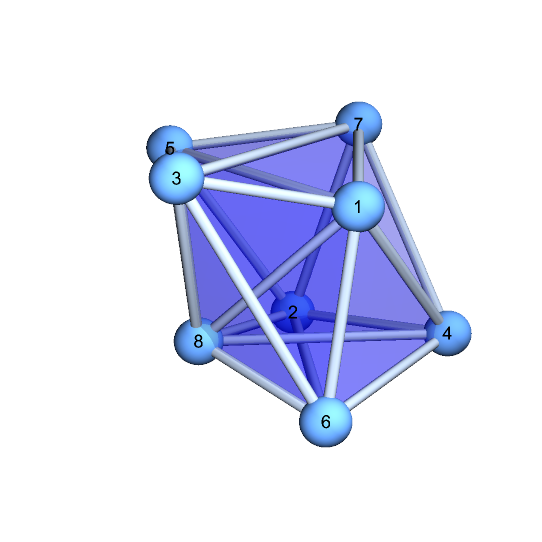}}
\label{Figure 1}
\caption{
The graph complement $G_{11}$ of $C_{11}$ is a
strongly regular graph of constant vertex degree $8$ and dihedral $\mathbb{D}_{11}$
symmetry and $f$-vector $(11, 44, 77, 55, 11)$, a row of a hyper-Pascal
triangle. The total number $198$ of complete subgraphs is a hyper Fibonacci number.
$G_{11}$ is actually homotop to a $3$-sphere $\mathbb{S}^3$ sharing the Betti vector 
$(1,0,0,1)$ and Euler characteristic $\chi(G_{11})=11-44+77-55+11=0$. The
curvature, the analog of the Gauss-Bonnet-Chern integrand in the continuum
is constant $0$. To the right is one of the $11$ isomorphic 
unit spheres $S(x)=G_{8}^+$ geometrically realized in $\mathbb{R}^3$. The graph 
$S(x)$ is homotopic to a 2-sphere $\mathbb{S}^2$, has the $f$-vector $(8,21,20,5)$ 
and $\chi(S(x))=8-21+2-5=2$. 
While $G_n$ has constant Euler curvature $K(x) = 1-8/2+21/3-20/4+5/5=0$, the spheres $S(x)$
have a curvature distribution that converges in the limit $n \to \infty$.
In this case $n=11$, all Lefschetz numbers, super traces of the induced maps on 
cohomology disappear. This happens always if $G_n$ is homotop 
to a $d=(4k-1)$-sphere, meaning $n=12k \pm 1$. 
}
\end{figure}

\paragraph{}
Graph complements of circular graphs are Cayley graphs of the circular group
that exhibit surprisingly rich combinatorial, geometric, topological and spectral 
features. They allow to illustrate results which in the continuum are more technical, 
like Lefschetz or Gauss-Bonnet theory.  
Spectral theoretical questions relate to zeta functions and harmonic analysis 
and eigenvalue problems for various circular matrices associated to the geometry.
Fourier theory makes there everything explicit, allowing also to understand graph limits.
Combinatorial problems come in when counting subgraphs, like the number of 
simplices the number of trees or forests or finding numerical quantities like the chromatic
number or independence number or when establishing
capacity quantities which are pivotal in information theory. 
The graph complements of cyclic graphs share the symmetry and regularity of the 
circle graphs but also always have the homotopy types of large dimensional spheres or a wedge sum 
of two large dimensional spheres. Unit spheres in these graphs are complements of path graphs and 
are always contractible or spheres. This is exactly what happens in realistic implementations
of manifolds like in a computer using floating point arithmetic: if a small accuracy threshold
is given, then the spheres of this radius are either homotopy spheres or then contractible.
At rational or typical algebraic numbers in an interval for example, the spheres of
accuracy are 0-spheres, while for most other numbers the spheres are contractible the
simple reason being that a computer can only have represent a finite amount of numbers.

\paragraph{}
Circle graph complements behave in many respects like manifolds or 
manifolds with boundary; but they are also models in which we allow for homotopy 
deviations. It is more subtle than homotopy alone as we look also at the homotopy of unit spheres.
We deal then with a class of metric spaces which share with Euclidean round 
spheres or round balls the property that arbitrary intersections
of geodesic spheres are either spheres or points. More generally, the union of path graphs or circle
graphs each of length not divisible by three all have as graph complement a single high
dimensional sphere, always homotopically speaking of course. 
A consequence is that graph complements of forests are always contractible or spheres. 
These spaces topologically behave like manifolds or manifolds with boundary.
An example is the Moebius strip which is the graph complement of the heptagon which 
however is homotopically a circle as removing a contractible part renders it contractible. 
Most amazingly, the Gauss-Bonnet curvature distribution of the graph 
complement of a large linear graphs show a non-trivial periodic universality. This is a stable
homotopy property in that this curvature attractor is invariant under the suspension operation.
It is completely unexpected to see an attracting period-6 cycle which consists of non-trivial 
smooth functions. It leads a period-2 attractor under the Barycentric refinement renormalization
map. In general, the curvature on a graph complement of a path graph has three 
different regions on each of which the curvature converges to a smooth limiting 
function in the limit. We have not yet analyzed the limiting functions carefully yet but they 
are given explicitly as hyperelliptic functions for each $n$. 

\paragraph{}
Our results are all located entirely within combinatorics. Geometric
realizations of the objects could produce classical topological spaces 
and would allow to see the results into the hart of stable homotopy questions in topology.
We would like to stress however that we never really need the continuum. Also the integrals
which appear are integrals of polynomials and so could be done without invoking the
real line. The linear algebra part entering when looking at cohomology groups or Lefschetz numbers
could be done over discrete fields like rational numbers, the eigenfunctions are in finite field extensions
of the rationals. Many questions are open. We do not understand 
yet for example for which trees the graph complement is contractible or what happens
for general with graph complements of triangle-free graphs. We know that for forests, the graph complement is 
either contractible or a sphere. The cyclic graphs are an example and one can look at 
more general Cayley graphs. If the cyclic graphs which are Cayley graphs of the cyclic group
is replaced by graphs which are Cayley graphs of dihedral groups, and which we only 
looked at briefly yet, we observe a similar story: we get again spheres and wedge sums of spheres. 
The combinatorial data are still given by more complicated recursions and there is a curvature 
renormalization limit. The dihedral groups are already non-Abelian. In general, 
one can look at the graph complements of any finite simple Cayley graph 
(in a frame work where generators always come also with their inverse so that we have
undirected graphs). A natural case not yet studied are the Cayley graphs of 
product groups $\mathbb{Z}_p \times \mathbb{Z}_q$ equipped with their
natural generators producing a grid graph. We then look at the graph complement of such 
grid graphs. The case $G=\mathbb{Z}_4 \times \mathbb{Z}_4$ is the {\bf tesseract}. Its graph dual is a 
wedge sum of six $3$-spheres as the Betti vector of $\overline{G}$
is $(1,0,0,7)$. But in general, it can be more complex: for $\mathbb{Z}_4 \times \mathbb{Z}_5$
already, the graph complement is no more a sphere, nor a bouquet of spheres,
as it has Betti vector $(1,0,0,2,1)$. 

\paragraph{}
This work started with the observation that 
graph complements of circle graphs are either wedge sums of spheres or 
spheres. We saw that by computing the cohomology groups for the first few dozen $n$
cases. These data can be obtained
as null spaces of the Hodge Laplacian $L=(d+d^*)^2$ for the graph, which is already 
a square of the Dirac operator $D=d+d^*$ \cite{KnillILAS,KnillBaltimore,AmazingWorld}. 
The Hodge method is swiftly done: order the maximal
simplices in the Whitney complex of the graph arbitrarily, write down the incidence matrices $d$
then build the matrix $L$ and look at the kernels of each block. The block corresponding
to $0$-functions is the Kirchhoff matrix $L_0= B-A$, where $B$ is the diagonal vertex
degree matrix and $A$ is the adjacency matrix of the graph. Discrete versions of calculus have been 
reinvented and refined many times, especially in the computer science literature. 
The basic idea however is already in the original approach of  pioneers like
Betti or Poincar\'e, using incidence matrices.
Hodge theory brings it down to linear algebra. 
There is no need for Clifford matrices in the discrete  because
the Dirac operator $D=d+d^*$ is the square root of the Laplacian. The functor giving
from an algebra its Clifford algebra is in the discrete implemented by Barycentric 
refinement, a process which transforms simplices into points on a new level giving rise then to 
new simplices. In the discrete the exterior algebra is already given in terms of functions on complete subgraphs.
The situation is trickier in the continuum, because lower-dimensional parts of a space like a manifold or
variety have to be accessed sheaf theoretically.

\paragraph{}
Given an automorphism $T$ of the graph $G$, one can look at the action $T_k$
which $T$ induces on the kernels on $k$-forms. The later are anti-symmetric
functions on $k$-simplices. We can then compute the trace of each block and super sum it up.
This gives the {\bf Lefschetz number} $\chi(G,T)$ of the graph. By applying the 
{\bf heat operators} $e^{-tL}$ on the 
{\bf Koopman operator} $U_T(f) = f(T)$ defined by the map, one can 
immediately see that $\chi(G,T)$ is the sum of the indices $i_T(x) = \omega(x) {\rm sign}(T|x)$
of the simplices (which we think of as points) fixed by $T$. The reason is that by McKean-Singer symmetry, 
\cite{McKeanSinger,Cycon,knillmckeansinger}, the super 
trace of ${\rm str}(L^k)=0$ for $k>0$ so that  $l(t)={\rm str}(\exp(-tL) U_T)$ is constant in $t$. 
We have $l(0) = {\rm str}(U_T) = \sum_{T(x)=x} i_T(x)$ and  
$\lim_{t \to \infty} l(t)=\chi(G,T)$ by Hodge and the fact that the eigenspaces of $L$
corresponding to non-zero eigenvalues are washed away, which is a result of having the 
non-zero eigenvalues being positive. 
This works without any assumptions whatsoever on the graph and is the reason why the Lefschetz
fixed point theory is so simple here. In \cite{brouwergraph} we gave a different proof closer
to Hopf. 

\paragraph{}
To digress a bit more, the story looked at first suitable for studying homotopy groups of spheres numerically: 
when adding two homotopies on a sphere, we realize the two on the two sides of the wedge sum, then lift it to the 
sphere which can be seen as a ramified cover of the wedge sum along the equator. 
A naive idea is to place {\bf sphere structures onto groups} and have so spheres 
act more conveniently on spheres.  We indeed see here sphere structures on 
Cayley graphs of cyclic groups $\mathbb{Z}_n$ and since $\mathbb{Z_m}$ can act 
on $\mathbb{Z}_n$, we can have spheres acting on spheres. 
It seems not yet to help to compute the homotopy groups of spheres but this is also not 
to be expected as homotopy groups of spheres are a difficult problem. 
However, the story is of independent interest in graph theory. 
We also want to use it to illustrate the use of graphs to study topology.
The already mentioned Lefschetz story. It implies for example
that if $G$ is a contractible graph, then any automorphism must have a fixed simplex. 
This is a discrete Brouwer fixed point theorem. 
Verifying Brouwer using discrete models has tradition \cite{Gale}.

\paragraph{}
The study of the {\bf topology of graph complements} in general appears to 
be a vastly unexplored area. 
That graph complements of cyclic graphs are already rich can surprise. 
In the future, we would like to know what happens in general with graph complements $\overline{\rho^m(G)}$ 
of $m$ times Barycentric refined graphs $G$. In one dimension, there is a limited class of 
topological spaces which appear: we get spheres and points (meaning contractible spaces). 
The topology is much richer in higher dimensions and still deserves to be explored further.
The Betti vectors of $\overline{\rho(K_3)}$ for example is $(2,2)$,
for $\overline{\rho(K_4)}$ already $(2,4,6)$ and for $\overline{\rho(K_5)}$ the Betti vector is
$(2, 5, 20, 38, 1, 5)$ and the Euler characteristic is $-25$. 
Also for general circulant graphs like the self-complementary {\bf Paley graphs} (quadratic reciprocity), 
rich topologies can appear. The Paley graph to the prime p=13 is a flat discrete torus is
a manifold with zero curvature and so a {\bf discrete Clifford torus}. 
We have not found larger Paley graphs yet that produce discrete manifolds.
Which graph complements of forests are points, which are spheres? 
Observed and not proven is also that  
the normalized spectrum of the Hodge operator of $G_n$ converges. It does so for the $C_n$
\cite{KnillBarycentric2}.
Also the structure of spectral zeta functions deserves more attention.
There are inverse problem of relating the spectrum to the graph. An example of a result is that
$\exp(\sum_{s=1}^{\infty} (-1)^{s+1} \zeta(s)/s)$ is the ratio of rooted forests over rooted trees,
a direct consequence of the matrix tree and matrix forests theorems. 
We plan to write about this more in the future but will mention the tree-forest ratio already here
as it diverges for $C_n$ in the limit $n \to \infty$ but converges to $e$ in the 
case of its complement $G_n$ for $n \to \infty$.

\paragraph{}
In order to illustrate the upcoming story, let us in this introduction take a concrete example and 
look at the graph complement $G_{11}$ of $C_{11}$. This graph is
obtained by drawing all the diagonals in a $11$-gon. The graph $G_{11}$ is a pretty inauspicious
geometric object, consisting of $11$ vertices and $44$ edges. It can be seen as the 
Cayley graph of $\mathbb{Z}_{11}$,
where we take as generators the complement of the usual generators $x \to x+1, x \to x-1$. 
There is more structure beyond this one-dimensional simplicial 
complex skeleton of $G$. There are $77$ triangles, $55$ tetrahedra 
(complete subgraphs with 4 vertices) and 11 hyper tetrahedra 
(complete subgraphs of 5 vertices). As we will see below, the {\bf f-vector}
$(11,44,77,55,11)$ is a row in a {\bf hyper Pascal triangle} and the sum $198$, counting all the simplices in 
the graph is a {\bf hyper Fibonacci number} $F_{11}$ satisfying the recursion $F_{n+1}=F_{n} + F_{n-1}+1$
with initial condition $F_0=1,F_1=0$. 
The {\bf Euler characteristic} of $G_{11}$ is $\chi(G_{11})=11-44+77-55+11=0$. In general, we will see
$\chi(G_n)=1-2 \cos(\pi n/3)$. 
Gauss-Bonnet tells that this is the sum of curvatures but because of the transitive automorphism
group the curvature $K(v)$ at every vertex $v$ of $G_{n}$ is constant 
$K(v) = \chi(G_n)/n$. By the general Gauss-Bonnet result,
the curvature is an anti-derivative of the simplex generating function of the 
unit sphere $S(x)$. In our case, the unit spheres $S(x)$ are all complements of 
graph paths $G_{8}^+$.

\paragraph{}
Let us dwell on this a bit more as it also brings us to the hart of the matter. First of all, the 
curvature of $G_n$ itself is not that interesting because it is constant: we have a transitive 
symmetry group $\mathbb{Z}_n$ on $G_n$ and so 
a constant Gauss-Bonnet curvature $\chi(G_n)/n$, because $n$ is the number of vertices.
In some sense, $G_n$ behaves like a 
{\bf round sphere} in the continuum, where in the even-dimensional case, 
we can write down immediately the Gauss-Bonnet-Chern integrand 
without computing the Riemann curvature tensor simply because by symmetry, 
the curvature is constant. For odd-dimensional manifolds, the Gauss-Bonnet-Chern
curvature is not even defined and usually assumed to be zero, as the Euler characteristic
is zero for manifolds for which Poincar\'e-Duality holds. We were 10 years ago first 
puzzled to see the discrete curvature always to be constant zero for odd-dimensional discrete manifolds 
\cite{cherngaussbonnet}.We proved it soon after with integral geometric methods
seeing curvature as an expectation of Poincar\'e-Hopf indices and quite recently also
saw it as a simple manifestation of Dehn-Sommerville \cite{dehnsommervillegaussbonnet}.
In our case, we have unit spheres $S(x)$ in $G_n$ which are graph complements
$G_{n-3}^+$ of path graphs which even in the odd-dimensional case have a non-trivial
curvature (which actually is universal in the limit $n \to \infty$). This is not
a contradiction, because we deal in the case of $G_n$ or $G_n^+$ with ``fuzzy spheres" 
which mathematically means that we have spheres modulo homotopy. But not only that
also the unit spheres of these spaces are either spheres or contractible. 

\paragraph{}
From a differential geometric point of view, the graphs $G_n^+$ are more interesting:
the curvature is not constant. We therefore started also to investigate more
the curvature of the path graph complements $\overline{G}^+$. Curvature is most elegantly 
described using the simplex generating functions. In full generality, the simplex generating 
function of a graph is the sum of the anti-derivatives of the simplex generating functions of
the unit spheres. The later are the curvature functions which when evaluated at the parameter $-1$
produce Euler characteristic and so the Gauss-Bonnet theorem. One might now have the impression
that this curvature is a discrete combinatorial oddity. But it actually is the real thing. Integral geometric
considerations show that this curvature is the analogue of Gauss-Bonnet-Chern in the case
when the graphs are triangulations of an even-dimensional compact Riemannian manifold. If we make
a finer Regge triangulation $G$ of a compact Riemannian manifold and realize both the graph $G$
as well as the manifold in a larger dimensional Euclidean space $E$ which is possible by Nash's embedding
theorem, then we can look at the curvature obtained by the index expectation of Morse functions 
given by linear functions in $E$. There is a natural Haar measure making this a nice rotationally
symmetric probability measure of functions. The induced curvature on $M$ is Gauss-Bonnet-Chern.
The induced curvature on $G$ is an index expectation curvature on the graph. 
High dimensional embedding considerations lead to the Levitt curvature we consider. 

\paragraph{}
In our case, for the graphs $G_n$ or $G_n^+$, the simplex generating functions satisfy a recursion
relation and will be identified as {\bf Jacobsthal polynomials}. The recursion for $G_n$ is 
$f_n(t)=f_{n-1}(t) + t f_{n-2}(t)$ with $f_0(t)=2,f_1(t)=1$ \cite{Jacobsthal1919, Swamy1999} 
For $G_n^+$, we have the same recursion with $f_0(t)=f_1(t)=1$.
We also know in the same way the simplex generating functions of the unit spheres of $G_{n}$. 
As the graphs $G_n$ have constant curvature and the Euler characteristic is known,
it follows from Gauss-Bonnet that 
$\int_{-1}^{0} f_{n-3}(t) dt = (1-2 \cos(\pi n/3))/n$.
For example, $f_{3}(t) = t^2+3 t+1$ (there are three vertices and one edge)
integrates to $-1/6$ on $[-1,0]$. But $G_3^+$ is the unit sphere of $G_6$, a graph 
that is homotopic to the figure 8 and so has Euler characteristic $-1$ with curvatures $-1/6$. 
It is curious that Gauss-Bonnet produces a relation for a sequence of recursively defined polynomials.

\begin{figure}[!htpb]
\scalebox{0.6}{\includegraphics{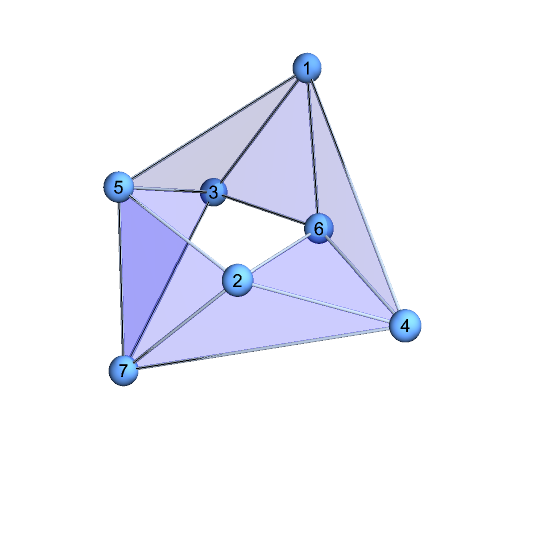}}
\scalebox{0.6}{\includegraphics{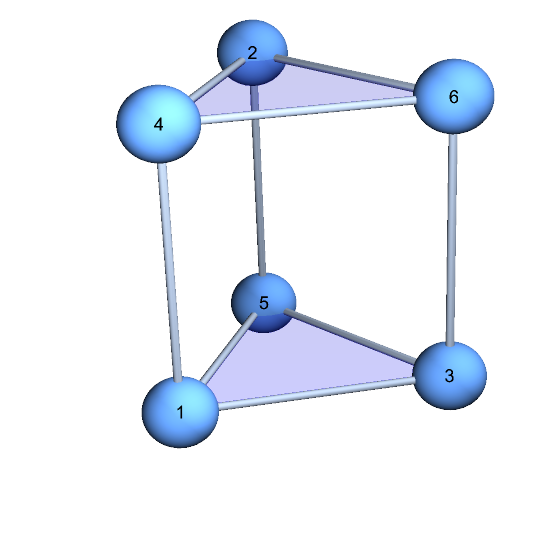}}
\label{Figure 2}
\caption{
To the left, we see the graph $G_7$. It is a triangulation of the {\bf Moebius strip}.
It shares with the self-complementary graph $G_5=\overline{C_5}=C_5$ the property that
it is homotopic to the circle $\mathbb{S}^1$. 
To the right, we see the graph $G_6$ which is a prism and homotopic to the 
figure $8$ graph $\mathbb{S}^1 \wedge \mathbb{S}^1$. All $G_{3d}$ are wedge sums of 
spheres. All $G_{3d-1}$ and $G_{3d+1}$ are spheres, homotopically speaking. 
}
\end{figure}

\paragraph{}
We have already mentioned that
the unit spheres $S(x)$ in $G_n$ are all graph complements $G_{n-3}^+$ of linear graphs $L_{n-3}$ 
with $n-3$ vertices. While $G_{3d+3}$ are wedge sums of two $d$-spheres and 
$G_{3d+2},G_{3d+4}$ are $d$-spheres, the unit spheres of $L_{3d+2},L_{3d+3}$ are 
$d$-spheres and $L_{3d+1}$ are contractible.
All unit sphere of $G_{11}$ for example are homotopic to $2$-spheres. 
Any intersection of two or more disjoint unit spheres in $G_{11}$ is always a $1$-sphere or contractible. 
If we look at unit spheres of unit spheres in $G_n$, then these are graph complements of
disjoint unions of linear graphs and therefore {\bf joins} of spheres or contractible graphs and 
so spheres or contractible graphs. In other words, intersections of an arbitrary collection of
unit spheres in $G_n$ always are either spheres or points. One can rephrase this 
and say that these are spaces of Lusternik-Schnirelmann category 1 or 2 \cite{josellisknill}.
Since we also know that the graph complements of star graphs are $0$-spheres, we can 
glue star graphs linear graphs and circular graphs together and get the topology 
of the total. For trees of forests, the graph complements
always are either contractible or spheres, graphs of category 1 or 2. 
We deal with a class of graphs for which the unit spheres are in the same class. 

\paragraph{}
Let us add a bit more on the use of graph theory in topology. 
Maybe because the origins of graph theory were ``humble, even frivolous" to quote \cite{BiggsLloydWilson}, 
it is tempting to dismiss higher dimensional topological features in graphs at first. 
Already Euler, who first considered graphs as a tool to study topological features of a 
two-dimensional map of St Petersburg and so looked beyond one dimension and seen graph 
theory as part of topology. Also early topologists like Hassler Whitney who worked both in 
graph theory as well as the foundations for modern differential topology would not
look at a graph as just one-dimensional simplicial complex but think about 
the {\bf clique complex} it defines. Indeed, the simplicial clique complex consisting of all subsets of 
the vertex set of a graph is today also called the {\bf Whitney complex} of the graph.
For $G_{11}$, it produces a $4$-dimensional topological space $X_{11}$ if realized in Euclidean space.
From the homotopy point of view, it is a three dimensional sphere. 

\paragraph{}
The $11$ hyper {\bf hyper-tetrahedra} $K_5$ in $G_{11}$ do not generate the topological realization
$X_{11}$ yet. The space $X_{11}$ is not ``manifold-like". 
But if we include the $11$ tetrahedra which are not contained in maximal hyper-tetrahedra, 
then this generates the full simplicial complex $X_{11}$, the geometric realization of the 
Whitney complex of $G_{11}$. 
\footnote{We usually do not bother with geometric realizations but stay within finite combinatorics. }
The smallest example of a graph $G_n$ with an impure simplicial complex is the graph $G_{6}$, 
It features $6$ vertices, $9$ edges and $2$ triangles. 
The two triangles as well as the 3 edges connecting them generate the complex. They form a {\bf prism}.
Traditionally, one would look at this as a model for a $2$-sphere but this is not what it is when looking
at the Whitney complex. The three squares are not present in the 
geometric realization $X_{6}$ of the Whitney complex. Including such faces comes natural in the context
of topological graph theory \cite{TuckerGross} but especially in higher dimensions has led to much 
confusion. We have $\chi(G_{6}) = 6-9+2=-1$ which reflects the fact that we have a 2-sphere in which 3 holes
are present. The book \cite{lakatos} is dedicated to the confusion 
and \cite{Richeson,Gruenbaum2003} explains this further. Things are crystal clear if one insists on
looking at simplices (complete subgraphs) as the faces of a polyhedron and refer to discrete CW complex
structures when building up polyhedra more effectively with as few cells as possible (i.e. appearing in
\cite{KnillEnergy2020}.

\begin{figure}[!htpb]
\scalebox{0.6}{\includegraphics{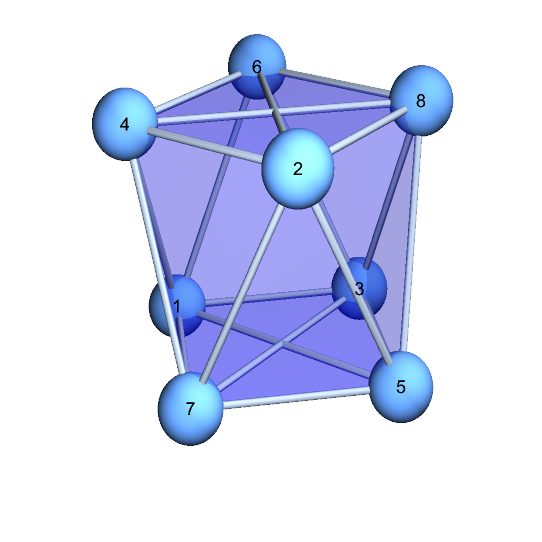}}
\scalebox{0.6}{\includegraphics{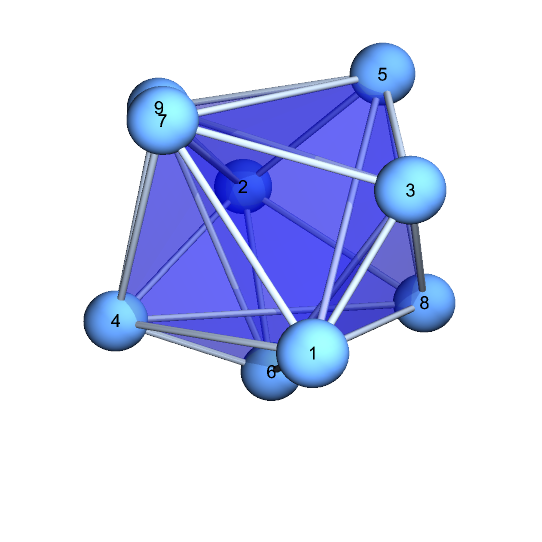}}
\label{Figure 3}
\caption{
The left graph $G_8$ is homotopic to a $2$-sphere $\mathbb{S}_d$ with $d=2$ and $n=3d+2=8$. 
It contains two sub-graphs $K_4$ which produces a 
$3$-dimensional complex, but the two $3$-dimensional caps can be thinned out by snapping
two edges. The unit spheres of $G_8$ are all house graphs homotopic to circles. 
To the right, we see $G_9$ which is homotopic to the wedge sum $\mathbb{S}^2 \wedge \mathbb{S}^2$. 
}
\end{figure}

\paragraph{}
If we shrink the triangles in $X_{6}$ to a point, we get a graph with two squares glued together. The shrinking
process is what one calls a {\bf homotopy}. Unlike homeomorphisms, homotopies can transcend dimension. The
maximal dimension of $G_6$ is $2$ but it is homotopic to a curve of maximal dimension $1$. 
You might have noticed that the topological space $G_{6}$ is homotopic to a figure $8$ complex which is a lemniscate. 
It is a curve which is a variety and not a manifold. Its Euler characteristic is $-1$ as should be for a
curve of genus $2$ as there are two holes. One calls the figure $8$ graph also the {\bf wedge sum} of two circles
or a ``bouquet of circles". The {\bf wedge sum} $S^d \wedge S^d$ of spheres is important as it is the crux to define
the {\bf addition} in the homotopy groups of spheres. If $f,g: S^t \to S^d$ are continuous maps from a pointed topological
space $S^t$ to the pointed topological space $S^d$, then one has a natural map $f \wedge g: S^t \to S^d \wedge S^d$.
As $S^d$ is a branched cover of $S^d \wedge S^d$ where the ramification happens on the equator, one can lift $f \wedge g$
to $S^d$ getting again a continuous map from $S^t \to S^d$. This addition is the operation in the {\bf homotopy groups} $\pi_t(S^d)$.
The structure of these groups is still not fully understood. Having {\bf small models for spheres and wedge sums of 
spheres} can be a major motivation for what we do here.

\paragraph{}
All the cohomology groups of $G_{11}$ can be computed quickly finite linear algebra. 
The $k$'th cohomology group is the kernel of the {\bf Hodge Laplacian matrix} $L=D^2=d d^* + d^* d$, 
where $D=d+d^*$ is the {\bf Dirac operator}
which is in the $G_{11}$ case a $198 \times 198$ matrix. The matrix $d$ is the {\bf exterior derivative}. It maps functions
on $k$-simplices to functions on $(k+1)$-simplices. We have to chose an {\bf order on each simplex} to define the map,
but changing such an order on a simplex is just a change of the coordinates and so a base change
which does not affect the kernel and so the cohomology. The Hodge approach to cohomology is rewarding because
we get more than just an abstract answer but {\bf concrete harmonic forms}, solutions to the Laplace equation $L \psi = 0$. 
In the case $G_{11}$, there is a harmonic form on $0$-forms, the constant function. All connected graphs just have one
such harmonic form. But then there is the harmonic $3$-form. Because $L$ is an integer matrix, it can be represented 
by an integer vector. In this case it takes values in $\{ \pm 1, \pm 3, \pm 4, \pm 7\}$. 

\paragraph{}
The exterior derivatives $d$ were called incidence matrices by Poincar\'e and 
have all the same properties than the exterior derivatives in the continuum. 
For a scalar function $f: V \to \mathbb{R}$ of the graph $(V,E)$, it produces the {\bf gradient} $df:E \to \mathbb{R}$
For a function $f: E \to \mathbb{R}$, the exterior derivative is the {\bf curl} $df: T \to \mathbb{R}$, where $T$
is the set of triangles. The matrix $L$ is block diagonal, where each block $L_k$ operators on $k$-forms. 
The operator $L_0$ is the Kirchhoff Laplacian, the analogue of the scalar ${\rm div} {\rm grad}$ in calculus. 
We have studied graph complements of cyclic graphs first in order to compute cohomology groups and 
spectra $\sigma(L(G_n))$ of $L(G_n)$. We notice for example that the {\bf Hodge spectrum} $\sigma(L(G_n))/n$ 
converges to a limit for $n \to \infty$ but we have not yet proven such a {\bf sphere central limit theorem}.

\paragraph{}
Triangulations of $d$-spheres can be realized as finite abstract simplicial complexes which are Whitney 
complexes of graphs defined by $2(d+1)$ points. The $2$-sphere for example is realized by the octahedron
graph, the $3$-sphere, also known as the $16$-cell is realized on a set with $8$ points.
These small simplicial complexes are the smallest triangulations of these manifolds and
so most economical. ($d$-simplices are not spheres in this frame work because simplices
are contractible and so points. Only the $(d-1)$-skeleton complex of the Whitney complex of a simplex is
a sphere.) The cross polytopes just mentioned are models of spheres that have high symmetry. They are
Platonic solids in arbitrary dimension and constant curvature 
$K(x) = \sum_{k=0}^d (-1)^k f_{k-1}(S(x))/(k+2)$, where $f_{-1}(A)=-1$
and $f_k(A)$ is the set of $k$ dimensional simplices in a graph $A$ and $S(x)$ is the unit sphere
of $x$. In the case of a $d$-sphere realized by the graph $G_n$  we have 
$K(x)=(1+(-1)^d) \frac{1}{n}$ on each vertex. The graphs $G_n$ share with {\bf cross polytopes} the property that they are very small
and also have a lot of symmetry. 
The cross polytopes are circulant graphs obtained by taking the complete graph 
$K_{2d+2}$ and deleting the $d+1$ large diagonals. Cross polytopes are graphs complements of $1$-dimensional
graphs. They are the complement graphs of $d+1$ disjoint $K_2$ graphs. But we have more with the graphs $G_n$
as we also have small implementations of wedge sums. 

\paragraph{}
As in classical topology, the wedge sum $S^d \wedge S^d$ of two spheres is no more a manifold. 
in classical frame works it can be realized as a {\bf variety} with one singular point, the point
where the spheres touch. For $d=1$, we get the {\bf lemniscate}. When discretizing this, we can glue two cross polytopes 
$S^d$ together at a point and so get a simplicial complex complex with $4(d+1)-1$ points. 
If we look at homotopic equivalents only, one can even do with $4(d+1)-2d=2d+4$ points.
For the circle $S^d=S^1$ for example one can glue together two circular graphs $C_4$ along an edge
and get the digital {\bf figure-8 curve}. But this implementation does not have constant
curvature. It is $-1/2$ on the two middle points and $0$ else, adding up to the Euler characteristic
$-1$. This persists in higher dimensions. Note that Gauss-Bonnet works for arbitrary graphs and 
that just in the case of even dimensional discrete manifolds, it goes in a limit to the 
{\bf Gauss-Bonnet-Chern theorem} known in differential geometry. 

\paragraph{}
It might surprise a bit that we realize a homotopy $d$-sphere or a wedge sum $S^d \wedge S^d$ in such 
a way that the curvature is constant and even keep the Platonic property of having isomorphic 
unit spheres everywhere with a transitive symmetry group. The graphs $G_n$ are {\bf constant curvature graphs}
and the ability to look at homotopic graphs allowed to get more symmetry. 
For every notion of sectional curvature which only depends on unit spheres 
we have the same curvature spectrum at every vertex. The graph $G_{12}$ for example which is 
homotopic to a wedge sum $\mathbb{S}^3 \wedge \mathbb{S}^3$, has constant Euler curvature 
$-1/12$ because the wedge sum of two spheres has Euler characteristic $\chi(G_{12})=-1$ and the 
curvatures must be constant on each of the 12 vertices.  There is therefore a
{\bf differential geometric angle} to the story. We will also see a {\bf differential topological}
aspect when building up the graphs $G_n$ or the dual graphs $G_n^+$ of linear graphs. Indeed, curvature
can be seen as the expectation of Poincar\'e-Hopf indices and the later can be seen as Euler characteristic
changes when adding cells during a build-up. In our case, we can see the sequence $G_n^ \to G_{n+1}^+$ as a
Morse build-up as the change of Euler characteristic is $1$ or $-1$ there. 

\paragraph{}
Euclid's geometry of the Euclidean space is primarily a story of circles and lines. 
If graph theory is used as a tool to study geometry in arbitrary dimensions, there is no natural notion 
of ``lines" or linear spaces or tangent spaces. Even geodesics are in general 
not unique already for points of distance two apart. There is however a {\bf theory of spheres}. 
Given a graph, there is a notion of {\bf unit spheres} $S(x)$ which are the subgraphs generated by 
the vertices adjacent to a vertex $x$. As a consequence of the fact that the graphs $G_n$ and $G_n^+$ 
have diameter $2$ for $n \geq 5$ the intersection of two unit spheres is always non-empty and of the form
$G_m^+$ which are homotopic to balls of spheres. The graphs $G_n$ share an important property which holds
in Euclidean spaces or round spheres: the intersection of an arbitrary number of geodesic unit spheres 
is either a point or a sphere. In the discrete, this is a property we know from discrete manifolds. 
If we define a $d$-manifold as a graph which has the property that every unit sphere $S(x)$ is a $d$-sphere
and a $d$-sphere is a $d$-manifold which becomes contractible when removing a vertex, then 
by induction, the intersection of an arbitrary number of unit spheres is either a point or a sphere. 

\section{Combinatorics}

\paragraph{}
The graph complements $G_n=\overline{C_n}$ of cyclic graphs $C_n$ are {\bf circulant graphs}
of diameter $2$. It is usually denoted as $Ci_n(2, \dots, [n/2])$ but we write $G_n$. We can
see every circulant graph as the undirected Cayley graph of a presentation of $\mathbb{Z}_n$ in which 
a set of generators $r_1,\dots,r_k$ together with its inverses $n-r_1,\dots,n-r_k$ are given. 
Then $C_n$ is $\mathbb{Z}_n$ with generators $1,-1$ while $G_n$ is $\mathbb{Z}_n$ with 
generators $2, \cdots n-2$. Whenever we look at the topology and homotopy of the finite simple graph
$G_n$, we refer to the topology of the {\bf Whitney complex} of this graph. This is the
finite abstract simplicial complex formed by the complete subgraphs of $G_n$. The {\bf maximal 
dimension} of $G_n$ is then the dimension of the largest simplex in $G_n$. It is $[n/2]-1$.
For $G_5$ for example, where $G_5=C_5$ the maximal dimension is $1$ while for $G_6$ the maximal
dimension is already $2$ because two triangles appear. A topological question for $G_n$ is 
equivalent to asking the same question for a geometric realization $X_n$ of $G_n$ 
in Euclidean space. Then $X_5$ is a circle and $X_6$ consists of two triangles with corresponding
vertices connected. Examples of quantitative numbers are the Betti numbers, the
dimensions of the cohomology groups, the homotopy groups. For $X_6$ for example which is homotopic
to a figure 8 and so a wedge sum of two circles, the Betti vector is $(1,2)$ and the fundamental
group is the free group with two generators.

\paragraph{}
Let us first look at the combinatorial question to 
determine the number $f_k$ of $k$-dimensional {\bf faces} (simplices, cliques) in $G$. 
We can build up the simplices inductively. If we know the $k$-simplices in $G_{n-1}$, they produce also
$k$-simplices in $G_n$. Additionally, any $k-1$-simplex in $G_{n-2}$ can be joined with a new vertex 
to get a $k$ simplex in $G_n$.  This immediately leads to the recursion
$$  f_k(G_n) = f_k(G_{n-1}) + f_{k-1}(G_{n-2}) \; . $$
They are known as {\bf hyper Pascal triangle relations}. It is a direct consequence from the set-up. We can build
the Whitney complex $G_n$ recursively by joining the Whitney complex of $G_{n-1}$ and adding an augmented
version of the complex of $G_{n-2}$ because there is then space for a new vertex. 
The number {\bf facets=maximal simplices=maximal cliques} is either $2$ or $n$, depending on whether $n$
is even or odd. The numbers $F_n=\sum_k f_k(G_n)$ giving the number of complete subgraphs of $G_n$
forms a sequence called the {\bf hyper Fibonacci numbers}. They satisfy the recursion
$$ F_{n} = F_{n-1} + F_{n-2} + 1 \; . $$
In our case, the initial condition are  $F_0=1,F_1=0$ which gives 
$F_2=1+0+1=2, F_3=0+2+1=3, F_4=2+3+1=6, F_5=3+6+1=10, F_6=6+10+1=17, F_7=10+17+1=28$ 
etc. Unlike the Euler characteristic formula which is $6$ periodic for $n \geq 2$ only, the 
numbers $F_n$ make sense as the total clique number for all $G_n$ with $n \geq 0$. 
For $n=1$ we have $K_1$ with $F_1=1$ and for $n=0$, we have the empty graph with $F_0=0$. 
For $n=2$ as $G_2=P_2 = \mathbb{S}^0$ is the 2 point graph which is the $0$-sphere and $n=3$ 
on as $G_3$ is the 3-point graph without edges which is $\mathbb{S}^0 \wedge \mathbb{S}^0$. 

\begin{center}
\begin{tabular}{l|lllll|ll}
     $n$    &    $f_0$ & $f_1$ & $f_2$ & $f_3$ & $f_4$ &     $F_n$ & $\chi(G_n)$ \\ \hline
     0      &       &          &     &       &       &             1  &   0         \\
     1      &     1 &          &     &       &       &             0  &   1         \\
     2      &     2 &          &     &       &       &             2  &   2         \\ \hline
     3      &     3 &          &     &       &       &             3  &   3         \\ 
     4      &     4 &       2  &     &       &       &             6  &   2         \\
     5      &     5 &       5  &     &       &       &            10  &   0         \\
     6      &     6 &       9  &  2  &       &       &            17  &  -1         \\
     7      &     7 &      14  &  7  &       &       &            28  &   0         \\
     8      &     8 &      20  & 16  &  2    &       &            46  &   2         \\
     9      &     9 &      27  & 30  &  9    &       &            75  &   3         \\
    10      &    10 &      35  & 50  & 25    & 2     &           122  &   2         \\ 
    11      &    11 &      44  & 77  & 55    & 11    &           198  &   0         \\ \hline
\end{tabular}
\end{center}

\paragraph{}
In the above table, $G_1$ is the only non-sphere or non-wedge sum of spheres. The 
formula $\chi(G_n) = 1-2 \cos(\pi n/3)$ only starts to apply for $n \geq 2$ and only really 
makes sense geometrically for $n \geq 3$ as for $n=3$ we have the complement of $C_3$ which 
is the graph with 3 vertices and no edges. Let us summarize this 

\begin{thm}[Hyper Pascal]
        The components $f_k(G_n)$ of the $f$ vector of $G_n$ satisfy the hyper Pascal relation.
        The total number of simplices in $G_n$ is the $n$'th hyper Fibonacci number.
\end{thm}

\begin{proof}
Use induction with respect to $n$. When going from $n$ to $n+1$, we pick a polar pair $(a,b)$, 
add the edge $(a,b)$, add a new vertex $x$ and connect it to all points except $(a,b)$. 
All the old $k$-simplices from $G_{n-1}$ remain. Additionally there is a $k$-simplex for
any $(k-1)$-simplex in $G_{n-2}$. 
\end{proof}

\paragraph{}
A consequence is that we know the {\bf Euler characteristic} 
$$  \chi(G_n) = \sum_{k=0}^{\infty} (-1)^k f_k(G_n) = \sum_{x \subset G_n} (-1)^{{\rm dim}(x)} $$
explicitly for $n \geq 2$: 

\begin{coro}[6-periodicity of Euler characteristic]
The Euler characteristic of $G_n$ is $1-2 \cos(\pi n/3)$ for $n \geq 2$.  \\
The Euler characteristic of $G_n^+$ is $1-\cos(\pi n/3)+\sin(\pi n/3)/\sqrt{3}$ for $n \geq 2$.
\end{coro}
\begin{proof}
Also this can be obtained by induction using the explicit recursion.
The formula only makes geometric sense for $n \geq 2$. For $n=2$, we think of $G_2$
as the 0-sphere, a 2-vertex graph without edges. 
An other proof can be done by using the explicit formula for the 
Jacobsthal polynomial $f_n(t)$ to which we come next. \\
The 6-periodicity comes from the fact that $G_n \to G_{n+3}$ is 
homotopic to a suspension and that the Euler characteristic of spheres is 
2-periodic switching between $0$ and $2$. The least common denominator of 
$2$ and $3$ is then $6$.
\end{proof}

\paragraph{}
For a general graph $G$, the {\bf simplex generating function} 
$$ f_G(t) = 1+\sum_{k=0}^d f_k(G) t^{k+1} $$
has the property that it multiplies $f_{A \oplus B}(t) = f_A(t) f_A(t)$, if $A \oplus B$
is the join of the graphs $A$ and $B$. The join in graph theory is often denoted
as the {\bf Zykov join} because it was first introduced by Zykov into graph theory
\cite{Zykov} but it does exactly have the same properties as in topology. The join with 
the zero sphere, a 2-point graph without vertices, is then a suspension.

\paragraph{}
The simplex generating function $f_n$ of $G_n$ and
the simplex generating function $f_n^+$ of $G_n^+$ are given by 
{\bf Jacobsthal polynomials}: 

\begin{lemma}[Jacobsthal]
The simplex generating functions $f_n(t)$ of $G_n$  and $f_n^+(t)$ of $G_n^+$ satisfy the
recursion
$$  f_n(t)   = f_{n-1}(t)   + t f_{n-2}(t),   f_0(t)  =2, f_1(t)  =1 \;, $$
$$  f_n^+(t) = f_{n-1}^+(t) + t f_{n-2}^+(t), f_{-1}^+(t)=1, f_0^+(t)=1 \;. $$
They are solved by the explicit formulas 
$$ f_n(t) =    \frac{\left( (\sqrt{4 t+1}+1)^n    -(1-\sqrt{4 t+1})^n    \right)}{2^n}                  \; . $$
$$ f_n^+(t) =  \frac{\left( (\sqrt{4 t+1}+1)^{n+2}-(1-\sqrt{4 t+1})^{n+2}\right)}{2^{n+2} \sqrt{4 t+1}} \; . $$
\end{lemma}
\begin{proof}
This follows from the Hyper-Pascal relations. While the values for $n=0$
and $n=1$ do not have a direct interpretation, they can back traced from the
relation. The explicit formulas are obtained using linear algebra similarly 
as the explicit {\bf Binet formula} for Fibonacci numbers are obtained: with an 
Ansatz $x^n$ one gets a basis solution $x=\sqrt{4 t+1}+ 1$ or $x=-\sqrt{4 t+1}+ 1$. 
One can then fix the constants to match the initial condition. 
\end{proof}

\begin{figure}[!htpb]
\scalebox{0.8}{\includegraphics{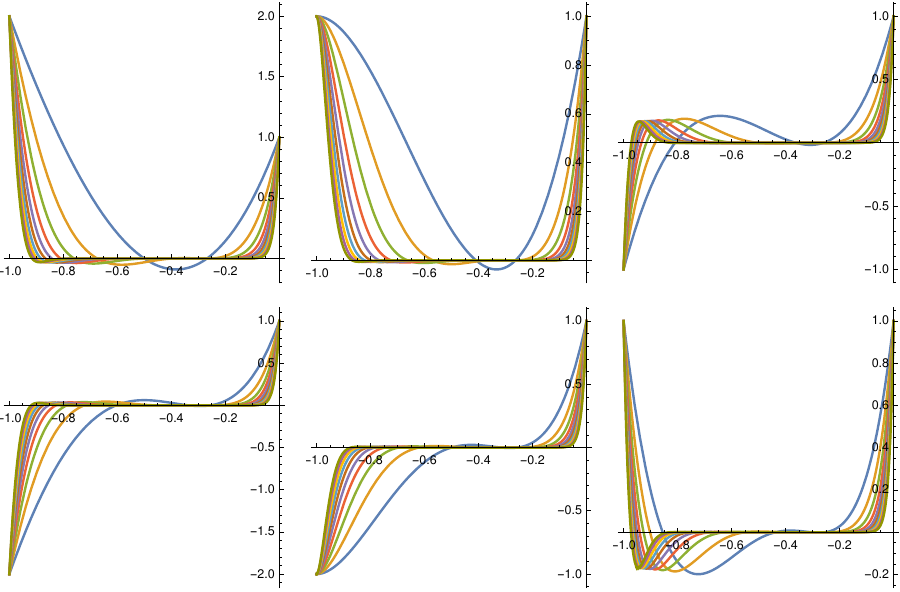}}
\label{Figure 4}
\caption{
Graphs of the first $60$ functions $f_n(t)$ on $[-1,0]$, grouped modulo 6. 
}
\end{figure}

\section{Homotopy spheres and wedge sums}

\paragraph{}
A topological question is to determine the {\bf homotopy
type} of the topological space obtained from $G_n$ when we realize it as
a Whitney complex. For more on discrete graph homotopy see
\cite{ComplexesGraphsHomotopy}.
The graphs $G_n$ are a test case to see how far we can go with computing all the
cohomology groups using Hodge theory. That is how we got to these graphs. 
When we looked at small $n$ cases and computed the cohomology up to $n=21$, 
we noticed that in all cases we have {\bf homology spheres} or 
{\bf homology wedge sums of two spheres}. We proved then the
following theorem: 

\begin{thm}[Sphere bouquet theorem]
$G_{3d+3}$ is homotopic to a wedge
sum of two $d$-spheres while $G_{3d+2}$ and $G_{3d+4}$ 
are both homotopic to $d$-spheres.
\end{thm}

\begin{proof}
Since in the dual picture $G_n \to G_{n+1}$ is an edge refinement
and the suspension $G_n \to G_n \oplus 2$ is $C_n \to C_n + 2$ (adding
two disjoint points), we see that the two operations commute. Indeed, modulo homotopies
the operation $G_n \to G_{n+1}$ is a cube root of the suspension
as the suspension of the wedge sum of two spheres is a wedge sum of two higher
dimensional spheres.
To prove this, look at the transitions $G_n \to G_{n+1}$ for $n=3$ to $n=5$.
Now pick a $G_n$ and look at its equator $E_n$ obtained by removing the vertices
$1,2,3$, then this is by induction either a lower dimensional
sphere or lower dimensional wedge sum of two spheres. If we make the
extension $G_n \to G_{n+3}$ the $E_n$ by induction transforms to a higher dimensional
space $E_{n+3}$ of the same type. Since $G_n$ is a suspension of $E_n$ and
$G_{n+3}$ is a suspension of $E_{n+3}$ we by induction know that also on level $n$,
the correct extension is done also on the next level.

In order to see what happens we split up the transition $G_n \to G_{n+3}$ 
and verify that it is a composition of suspension and homotopies. 
Since the transition happens for $C_n$ by taking a disjoint union of $C_n$ 
with the path graph $C_3^- = G_3^+=(u,v,w)$ 
(which has a dual $K_2 + K+1$ which is homotopic to a 0-sphere so that we have in the dual
a suspension), then snapping one of the edges $(a,b)$ in $C_n$ and connecting $(a,u)$, then $(b,w)$
(which are all homotopies when seen on the $G_n$ side), 
we see that $G_n \to G_{n+3}$ is a suspension modulo homotopies.
\end{proof}

\begin{figure}[!htpb]
\scalebox{1.5}{\includegraphics{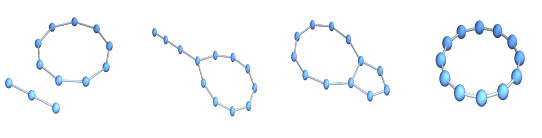}}
\label{Figure 5}
\caption{
Extending from $G_n$ to $G_{n+3}$ is a suspension followed up by 
homotopies. We look at this in the dual picture when going from 
$C_n$ to $C_{n+3}$. 
}
\end{figure}

\paragraph{}
Observe that if we take any unit sphere $S(x)$ in $G_{n} = \overline{C}_n$ 
and remove the edge $e=(a,b)$, if $\{ a,b \}$ was the sphere $S_{2,C_n}(x)$ 
of radius $2$ in $C_n$ we end up with a graph $S(x)-e$ which is isomorphic to 
$G_{n-3}$. It so happens that $S(x)$ and $S(x)-e$ are homotopic in the case $n=3d+2,n=3d+3$
and that $S(x)$ is contractible if $n=3d+1$. 
This fact shows with induction that all unit spheres are homotopic to $(d-1)$
dimensional spheres or a wedge sum of $(d-1)$-spheres. 

\paragraph{}
Let us look at the duals $G_n^+$ of path graphs $C_n^-$ now. They are 
the unit spheres of a vertex in $G_{n+3}$

\begin{thm}
The graphs $G_{3d+1}^+$ are contractible, the graph $G_{3d+2}^+$ and 
$G_{3d+3}^+$ are homotopic to $d$-spheres. 
\end{thm}
\begin{proof}
The proof is the same. We also have here that $G_n^+ \to G_{n+3}^+$ is a suspension
modulo homotopies. Since $G_3^+$ is a $0$ sphere and $G_4^+$ is contractible 
(it is the unit sphere of the Moebius strip $G_7$)
and $G_2^+$ is a sphere (the unit sphere of a point in the circle $G_5$), we
have now a dichotomy for the $G_n^+$ and not the trinity as in $G_n$. 
\end{proof}

\paragraph{}
Since disjoint unions of graphs become joins and joins of spheres are spheres and
all path graphs of length not divisible by $3$ and all cycle graphs of length not
divisible by $3$ are spheres, we have:

\begin{coro} 
A graph complement of an arbitrary disjoint union of linear graphs or cycle graphs which all have
lengths not divisible by $3$ are homotopy to some $d$-sphere, where $d$ is expressible through
the lengths of the parts. 
\end{coro}

\begin{proof} 
The disjoint union becomes joins. If $A_1,A_2$ are homotopic and $B_1,B_2$ are homotopic,
then the joins $A$ of $A_1$ and $A_2$ is homotopic to the join $B$ of $B_1$ and $B_2$. 
If the length of a circular graph is divisible by $3$, then we deal with a wedge sum of 
two spheres. If the length of a linear graph is divisible by 3, we deal with a contractible
space. In all other cases, we have spheres and joins of spheres are spheres. 
\end{proof}

\paragraph{}
There are more graphs which can be included and still are part of the sphere monoid
in the dual: any star graph with $n$ spikes has a complement which is a $0$-sphere as 
it is the disjoint union of a point $K_1$ and a graph $K_n$. 

\begin{figure}[!htpb]
\scalebox{0.8}{\includegraphics{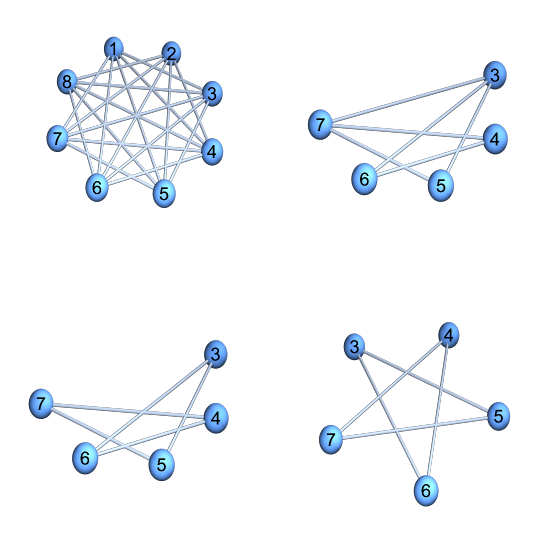}}
\label{Figure 6}
\caption{
We see the graph $G_8$, then a unit sphere $S(x)$ of $G_8$.
Remove an edge $e$ from this unit sphere to get the graph $G_5=S(x)-e$. 
In this case, $S(x)$ and $S(x)-e$ are homotopic.
}
\end{figure}

\begin{figure}[!htpb]
\scalebox{0.8}{\includegraphics{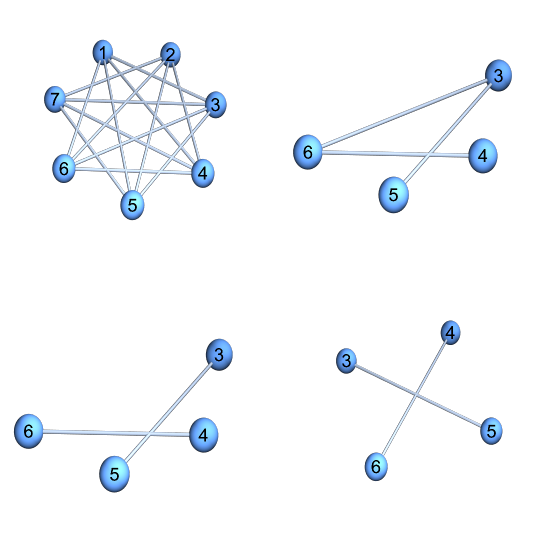}}
\label{Figure 7}
\caption{
For $G_7$, the unit spheres $S(x)$ are all contractible.
Remove an edge $S(x)-e$ to get the graph $G_4$ which is a union
of two complete graphs $K_2$ and so homotopic to a $0$-sphere
$\mathbb{S}^0$. 
}
\end{figure}

\begin{figure}[!htpb]
\scalebox{1.2}{\includegraphics{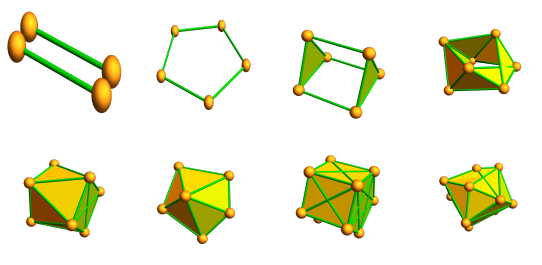}}
\label{Figure 8}
\caption{
The graphs $G_4 \sim S^0$,
$G_5 \sim \mathbb{S}^1$, $G_6 \sim \mathbb{S}^1 \sim \mathbb{S}^1$, 
$G_7 \sim \mathbb{S}^1$, $G_8 \sim \mathbb{S}^2$
$G_9 \sim \mathbb{S}^2 \wedge \mathbb{S}^2$, $G_{10} \sim \mathbb{S}^2$ 
and $G_{11} \sim \mathbb{S}^3$.
}
\end{figure}

\paragraph{}
Homotopy as a much rougher equivalence relation than homeomorphism. It does not honor
dimension for example. In order to capture also a discrete version of homeomorphism which works
in arbitrary dimensions, we explored a definition of homeomorphism based on homotopy
which also incorporates dimension, the inductive dimension of the nerve graph of the 
open covering defining the topology \cite{KnillTopology}. 
The topology of $G_n$ is in general more complicated because the simplicial 
complexes appearing for $G_n$ are non-pure for $n>7$ already. The self-dual case
$G_5=C_5$ is the only positive-dimensional discrete manifold without boundary. 
The graph $G_7$ is a {\bf discrete Moebius strip} (not to be confused with the discrete {\bf Moebius ladder} 
which is also given by circulant graphs but is not a discrete manifold with boundary). 
In our case, the graph $G_7$ is besides $G_4$ the only {\bf positive dimensional discrete manifold with boundary}
among the graphs $G_n$. 
The graph $G_6$ is a non-pure prism-graph homotopic to the figure $8$ graph. The graph $G_8$ which
is a homotopy $2$-sphere has as unit spheres house graphs which are homotopy $1$-spheres. 

\section{Subgraphs}

\paragraph{}
Given a graph $G=(V,E)$ and a subset $W$ of the vertex set, we get a graph 
$G(W)= (W,E(W))$, where $E(W)$ is the subset of edges $(a,b) \in E$ 
such that $\{ a,b \} \subset W$. It is called the {\bf induced subgraph} of $W$. 
What kind of subgraphs can occur? We know that as a consequence of the complement of $G_n$
or $G_n^+$ being triangle-free that $G_n$ or $G_n^+$ are {\bf claw-free}. This does not 
mean of course that there are no claw graphs as subgraphs, but it means that every claw
graph generates a larger graph. This in particular means that if $T$ is a
tree inside $G$ which is not a point (a seed) or path graph (a grass), 
then $T$ generates a larger graph. In other words, the only induced trees in
a claw free graph $G$ are path graphs or points. We can also look at the dual and look 
at the graph generated by the set $W$ in $C_n$. This is a finite collection 
of path graphs or points. Because the dual of such a disjoint union is the join
of the corresponding duals and the join of spheres is a sphere and the join of
anything with a contractible graph is contractible, we have:

\begin{thm}
Any strict induced subgraph of $G_n$ or $G_n^+$ is either a sphere or a 
contractible graph. 
\end{thm}

Special cases are the induced graphs of vertex sets with $n-1$ vertices in $G_n$ 
which produce $G_{n-1}^+$ or the unit spheres of a vertex which are generated by 
a set with three points missing and so is $G_{n-3}^+$. The induced graph of a 
set of $1$ point is always contractible. 

\paragraph{}
A graph without closed loops (including triangles) 
is also called a {\bf forest}. The connected components of a forest are {\bf trees}. 
This includes one point graphs $K_1$ which can be considered {\bf seeds} of a tree. 
Forests are triangle free graphs and so have maximal dimension $1$.
Forests which do not consist entirely of seeds define simplicial 
complexes for which the Euler characteristic is equal to the number of trees. 
This is a consequence of the Euler-Poincar\'e formula $\chi(G) = b_0-b_1$
and the fact that the contractibility of 
each component implies all Betti numbers $b_k$ to be zero for positive $k$ and
especially the genus $b_1=0$. One can also prove this by induction by seeing that 
adding branches to a tree does not change the 
Euler characteristic: any growth of the tree adds the same number of vertices 
and edges. Given a graph $G$, one can look at {\bf spanning trees}, trees within
$G$ which have the same number of vertices or {\bf spanning forests}. A {\bf rooted
tree} assigns to a tree also a base point, the root.

\begin{figure}[!htpb]
\scalebox{1.0}{\includegraphics{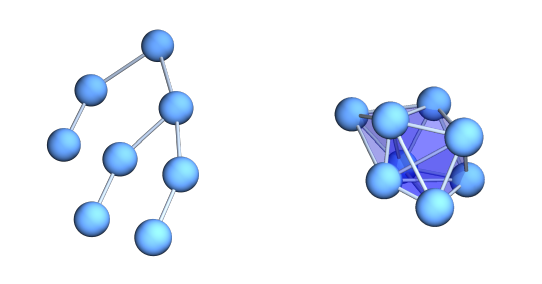}}
\scalebox{1.0}{\includegraphics{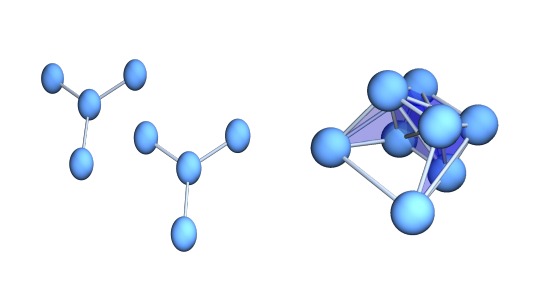}}
\label{Figure 8}
\caption{
A tree $G$ and its complement which is a $2$-sphere in this case. 
The second picture shows a forest with two star trees in which
case the complement is a circle. As each star produces a $0$-sphere
the complement is a join of two $0$-spheres. 
}
\end{figure}

\paragraph{}
Trees and forests in general are important structures in graph theory. 
The {\bf matrix tree theorem} tells that the number of rooted spanning trees in a graph 
is the {\bf pseudo determinant} ${\rm Det}(K)$ (the product of the non-zero eigenvalues
of $K$), of the {\bf Kirchhoff Laplacian} $K(G)$ of the graph. And ${\rm det}(1+K)$ 
is the number of {\bf rooted forests} in the graph. The number of {\bf trees} is then
${\rm Det}(K)/n$, where $n$ is the number of vertices. And the number of {\bf forests}
is ${\rm det}(1+K)$. This is the {\bf matrix forest theorem}
of Chebotarev and Shamis \cite{ChebotarevShamis2}.
Both the tree and forest results readily follow from a generalized 
{\bf Cauchy-Binet} result. (See \cite{cauchybinet,Knillforest} for
some results and references.)

\paragraph{}
The number of spanning trees of a graph is also called the {\bf tree complexity} while
the number of rooted spanning forests divided by $n$ 
is the {\bf forest complexity}. In the ratio of the complexities the $n$ disappears and
leads to the quantity 
$$  \frac{{\rm det}(K+1)}{{\rm Det}(K)} \;   $$
is interesting. The tree and forest numbers grow exponentially fast but the ration of their
complexities has a chance to go to a limit. 
Let ${\rm Tree}(G_n)$ denote the number of rooted trees and 
${\rm Forest}(G_n)$ the number of rooted forests in $G$. 

\begin{center}
\begin{tabular}{l|llll} 
n  & {\rm Tree} $G_n$  & {\rm Forest} $G_n$  &  {\rm Tree} $G_n^+)$ & {\rm Forest} $(G_n^+)$ \\ \hline
$4$  & $4$  & $9$  & $4$  & $21$  \\
$5$  & $25$  & $121$  & $55$  & $209$  \\
$6$  & $450$  & $1728$  & $780$  & $2640$  \\
$7$  & $8281$  & $28561$  & $12649$  & $40391$  \\
$8$  & $166464$  & $541205$  & $235416$  & $726103$  \\
$9$  & $3709476$  & $11621281$  & $4976784$  & $15003009$  \\
$10$  & $91494150$  & $279508327$  & $118118440$  & $350382231$  \\
\end{tabular}
\end{center}

\paragraph{}
The {\bf tree-forest complexity ratio} of rooted forest and rooted 
$$ r(G_n) = \lim_{n \to \infty} \frac{{\rm Forest}(G_n)}{n {\rm Tree}(G_n)}
          = \lim_{n \to \infty} \frac{{\rm Det}(K)}{{\rm det}(1+K)} $$
converges. This is the fraction of the {\bf pseudo determinant} over the {\bf Fredholm 
determinant} of the {\bf Kirchhoff matrix} of the graph. In terms of eigenvalues, it is
$$  r(G) = \lim_{n \to \infty} \prod_{\lambda_k \neq 0} (1+ \frac{1}{\lambda_k})  \; .   $$
Because the eigenvalues different from $n$ and $0$ are explicitly known 
$$  \lambda_{k,n} = \sum_{m=2}^{n-2} 2\sin^2(\pi m \frac{k}{n}) \;  $$
which is the same list than 
$$  \mu_{k} = n- 2 \sin^2(\pi k/n) \; $$
we have 
$$ r(G_n) = \prod_{k=2}^{n-1} (1+ \frac{1}{(n-1)-2 \sin^2(\pi k/n)})  \; . $$
The next formula also will allow to study tree-forest ratios for 
manifolds, where we do not have a finite amount of trees or manifolds to count. 
For $M=\mathbb{T}=\mathbb{R}/\mathbb{Z}$, where $\zeta$ is the {\bf Riemann zeta function},
we have $\rho(G) = 2/\pi$, more generally $2d/\pi$ if $d$ is the diameter of the circle.

\begin{lemma}
The forest-tree ratio $r(G)$ is equal to 
$$  e^{\sum_{s=1}^{\infty} (-1)^s \zeta_G(s)/s} \; , $$
where $\zeta_G(s)$ is the {\bf spectral zeta function} of $G$. 
\end{lemma}

\begin{proof}
The forest-tree ratio of $G_n$ is 
$$  r(G_n) = \prod_{k=2}^n (1+\frac{1}{\lambda_k(G_n)}) \; . $$
This product limit exists if the limits
$$ \zeta_n(1) = \sum_{k=2}^{n} \frac{1}{\lambda_k(G_n)} $$
$$ \zeta_n(2) = \sum_{k=2}^{n} \frac{1}{\lambda_k(G_n)^2} $$
etc exist and decrease because
$$ \log(r(G_n)) = \zeta_n(1) - \frac{\zeta_n(2)}{2} + \frac{\zeta_n(3)}{n} - \cdots . $$
\end{proof}

\begin{thm}
The forest-tree ratio of $G_n$ converges
$$  \lim_{n \to \infty} r(G_n)  = e \; ,  $$ 
where $e=2.71845$ is the Euler number.
\end{thm}

\begin{proof}
The limit $\zeta(s) = \lim_{n \to \infty} \zeta_{G_n}(s)$ exists for all $s \geq 1$. 
It is $1$ for $s=1$ and $0$ else. 
Because $G_n$ is the complement of $C_n$, the nonzero eigenvalues of the Kirchhoff matrix 
of $G_n$ can also be written as $n-\lambda_k$, where $\lambda_k$ are the eigenvalues of $C_n$. 
The spectrum therefore is in the interval $[n-4,n]$. 
So, $\zeta_n(1) \to 1$ and $\zeta_n(k) \to 0$ for all $k>0$. 
\end{proof} 

\paragraph{}
The limit exists actually surprisingly often for graph complements of sparse graphs, where
the complement has diameter $2$. The limit is usually $e$. This is very robust. Take for
example a fixed finite set $A$ of generators in $\mathbb{Z}_n$, then look at the circular graph 
$C_{n,A}$, which is the Cayley graph. Now, the graph complement $G_{n,A}=\overline{C}_{n,A}$ 
still satisfies
$$ \lim_{n \to \infty} r(G_{n,A})  = e  \; . $$
We can even increase the cardinality of the generators when taking the limit as long as $n/|A_n|$
goes to infinity. Of course, for cyclic graphs $C_n$, the result is no more true. In general, 
if the {\bf graph diameter} grows, also the tree-forest ratio grows simply because we have more
possibilities to plant forests than trees. We will explore this a bit more elsewhere as it 
has relations with topological invariants.

\begin{figure}[!htpb]
\scalebox{0.6}{\includegraphics{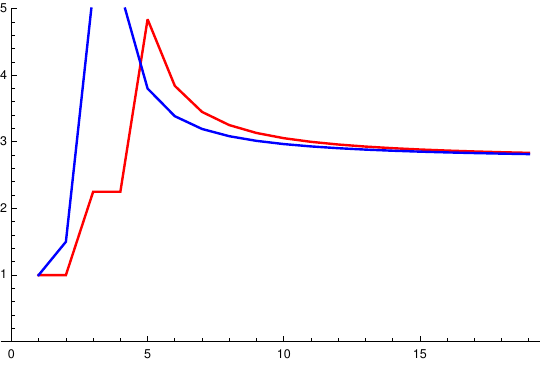}}
\label{Figure 9}
\caption{
The tree forest ratio for $G_n$ (red) and $G_n^+$ (blue) as
a function of $n$. 
}
\end{figure}

\paragraph{}
Given a forest $G$, we can look at the graph complement $\overline{G}$. 
As $G$ is a rather sparse graph when seen as an embedding in a complete graph, 
the graph $\overline{G}$ are quite messy in general. 
A bit surprising is that there is interesting topology coming in. A bit
more surprising is that the topology can not be too complicated.
In the following, we mean with {\bf ``is a point"} rephrasing that it is homotopic 
to a point and {\bf ``is a sphere``} with is homotophic to a sphere".

\begin{coro}
The graph complement of a tree is a point or a sphere,
(meaning as usual that it is either contractible or homotopic to a sphere). 
\end{coro}

\paragraph{}
The proof is has three ingredients:  \\

(i) The graph complement of a disjoint union is a join and the 
union of points and spheres produce a monoid under the join addition.
These statements are true for homotopy points and homotopy spheres.
The addition of a contractible graph $A$ and any other graph  $B$
is contractible (as one can see by induction in the number of points in $A$.
We still have to see that if $A$ is a homotopy sphere and $B$ is a homotopy sphere,
then the joint $A+B$ is a homotopy sphere. We can see this also by induction.
A homotopy sphere is characterized as a graph which allows a contractible part to
be taken away obtaining a contractible graph. In other words, a homotopy sphere
is a graph which is the union of two contractible graphs and which is not contractible
by itself. In other words, a homotopy sphere is a graph of Lusternik-Schnirelman
category 2 and this property is preserved under the Zykov join operation. \\

(ii) For star graphs or linear graphs, the graph complements are either
spheres or points. \\

(iii) If a linear graph $A$ with endpoints (leaves) $(a,b)$ is added to an other graph
$B$ connecting $a$ to a leaf $b$ of $B$, then the same properties like
for disjoint sums hold except that the graph can collapse a sphere
to a contractible one and a contractible one to a sphere. 
In other words, the wedge sum of two trees when 
done along leaves produces on the graph complement the same effect 
(modulo homotopies) than the disjoint sum.

\begin{coro}
The graph complement of a triangle free graph is either a point, a sphere
or a wedge sum of spheres. 
\end{coro}

\paragraph{}
We just need to see when we add disconnected graphs together. Combining two 
one-dimensional graphs is done by the join followed by a removal of an edge. 
Both preserve the union of the set of contractible and sphere graphs. 
We so far have not yet been figured out which tree complements 
give spheres and which tree complements give contractible graphs. 
This is equivalent of knowing the {\bf cohomological dimension} 
of a tree complement. 

\begin{figure}[!htpb]
\scalebox{1.1}{\includegraphics{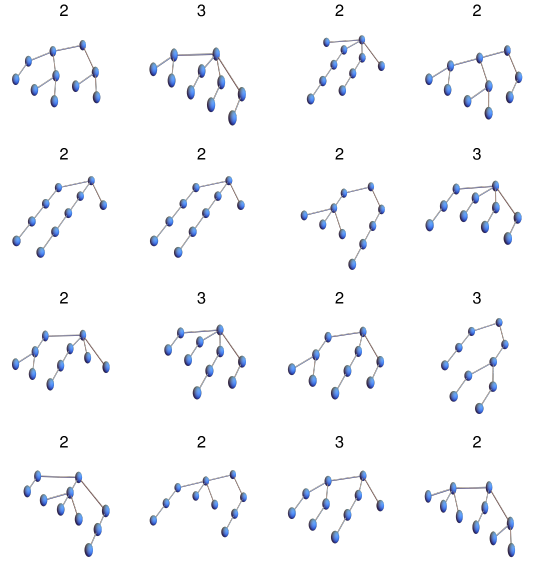}}
\label{Figure 10}
\caption{
Graph complements of trees are contractible or spheres. 
For most trees $G$, the graph complement is contractible. Here are
examples, where the graph complement is a positive dimensional sphere. 
This cohomological dimension is in the plot label. 
}
\end{figure}

\section{Morse build-up}

\paragraph{}
A sequence of graphs $G_0=0,G_1=K_1,G_2, \cdots, G_n=G$ is a {\bf Morse filtration} of $G$
if the subgraphs $G_n$ is obtained from $G_{n-1}$ by adding a vertex connected to either 
a contractible part of $G_n$ or to a subgraph which is homotopic to a $d$-sphere. 
If $f(x)$ is the numerical function on the vertex set of $G$ telling at which 
time the vertex has been added, then $S_{G_k}(k) = G_{k-1}$ is the unit sphere
and $S_f(x) = \{ y \in G, f(y)<f(x) \}= G_x$. 
We can then assign a {\bf Poincar\'e-Hopf index} $i_f(x) = 1-\chi(S_f(x))$.
Because $d$-spheres have Euler characteristic $(1+(-1)^d \in \{0,2\}$, the
indices on a graph with a Morse filtration are in $\{-1,1\}$. 

\begin{coro}
The graphs $G_n^+$ admit a Morse filtration.
\end{coro}
\begin{proof} 
This follows from the developments of the cohomologies and 
the fact that going from $G_n$ to $G_{n+3}$ produces a suspension. 
When we add a new vertex going from $G_{3d+2}^+,G_{3d+3}^+$ nothing
changes, while going from $G_{3d+1}^+,G_{3d+2}^+$ either has index
$1$ or $-1$.
\end{proof}

\paragraph{}
Discrete Morse theory emerged in the 1990ies \cite{Forman1999,Forman2002}. 
A different take came from digital topology \cite{I94a,Evako1994}. But all this
can be done conveniently also within pure graph theory. 
Topological features can now be observed in number theoretical contexts. 
In \cite{PrimesGraphsCohomology} we have looked at the graphs
$P(n)$ with vertex set $\{ k \; | \; 2 \leq k \leq n, k \; {\rm square} \; {\rm free} \;\}$
and edges consisting of unordered pairs $(a,b)$ in $V$, where either $a$ divides $b$ or $b$ divides $a$.
This produced a Morse filtration $P(n) = \{ f \leq n \}$ for $f(x)=x$. 
We had there $\chi(P(n))=1-M(n)$ with the Mertens function $M(n)$. 
The values $-\mu(k)$ of the M\"obius function are Poincar\'e-Hopf
indices $i_f(x)=1-\chi(S^-_f(x))$ of the counting function $f(x)=x$.
This cell complex has been introduced already in \cite{Bjoerner2011}.

\begin{figure}[!htpb]
\scalebox{0.5}{\includegraphics{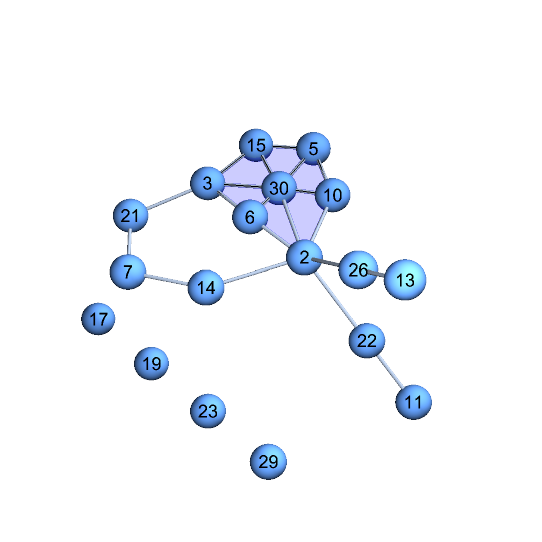}}
\label{Figure 11}
\caption{
We see here the prime graph $P(30)$ which is an other example of a 
Morse build-up. In this case, the Poincar\'e-Hopf indices are then
the Moebius function values $-\mu(n)$ known in number theory. 
}
\end{figure}

\paragraph{}
Every simplicial complex can be built up by starting adding vertices,
then edges, then triangles etc until the entire complex is there. 
Each addition of a $k$-simplex means to attach a $k$-dimensional cell
(handle) by attaching its boundary (a sphere) to the already given complex. 
More generally, with a notion of ``sphere" as a complex which becomes contractible
after removing a contractable part, we can define more general CW complexes 
in purely combinatorial manner. When attaching a new
cell to a contractable part, we have a homotopy step. 

\paragraph{}
In the case of graphs, a Morse build-up is not always possible. 
In the cube graph for example, we can not remove any point 
because each point is attached to three vertices. In order to decompose
the complex, we would have to treat it as a simplicial complex and 
remove edges first. The graphs $G_n$ are interesting
graphs because we can build them up by adding vertices and so have a Morse
build-up. What is remarkable that the Morse build-up is possible even with 
all stages to be spheres or wedge sums of spheres. Such deformations by adding
points is certainly not possible for discrete manifolds. 
But here it is possible $G_5$ is a pentagon. It is deformed to 
$G_6$ which is $\mathbb{S}^1 \wedge \mathbb{S}^1$ which then is deformed to $G_7$ 
which is the Moebius strip and again $\mathbb{S}^1$. 

\begin{figure}[!htpb]
\scalebox{1.0}{\includegraphics{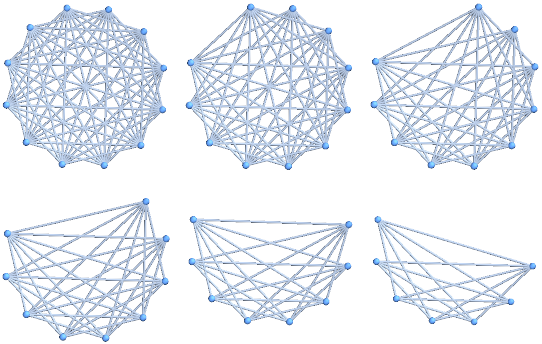}}
\label{Figure 12}
\caption{
When removing a vertex from $G_n$ we get $G_{n-1}^+$ which has an additional 
edge. Removing now vertices from that which belong the outer edge gives $G_{n-2}^+$
etc. The first move going from $G_n$ to $G_{n-1}^+$ as well as
the reductions $G_{n}^+ \to G_{n-1}^+$ are all Morse. The Morse buildup to $G_n$
produces Poincar\'e-Hopf indices adding up to Euler characteristic. 
}
\end{figure}

\paragraph{}
If $L_n = C_n^-$ is the linear graph with $n$ vertices obtained by snapping an edge from 
the cyclic graph $C_n$, then its graph complement $\overline{L}_n$ shall be denoted with
$G_n^+$ because it is obtained from $G_n$ by adding an edge. The automorphism group of 
$L_n$ and so $\overline{L}_n=G_n^+$ is $\mathbb{Z}_2$, in comparison to the
dihedral group $\mathbb{D}_n$ symmetry which is the automorphism group of $C_n$ or 
$G_n=\overline{C}_n$.

\begin{lemma}
The graphs $G_{3d+1}^+$ contractible and $G_{3d+2}^+,G_{3d+3}^+$ are 
homotopic to $d$-spheres. Two of the unit spheres of $G_{n}^+$ are $G_{n-2}^+$,
all other unit spheres are $G_{n-3}^+$. All of the unit spheres of 
$G_{n}$ are $G_{n-3}^+$. 
\end{lemma} 

\paragraph{}
This means that we can build up $G_{n}^+$ as a Morse complex. One third of the
time we have a homotopy extension, the other times, we alternating add even or
odd dimensional cells. It follows that also $G_n$ is a Morse complex. 

\paragraph{}
We have seen that the unit spheres of $G_n$ are $G_{n-3}^+$ and that two unit spheres
of $G_n^+$ are $G_{n-2}^+$. Indeed we have

\begin{coro} 
The intersection of two non-adjacent unit spheres in $G_n$ is $G_{n-4}^+$. 
The intersection of two adjacent unit spheres in $G_n$ is $G_{n-5}^+$. 
\end{coro}

\section{Dimension}

\paragraph{}
The {\bf maximal dimension} ${\rm max dim}(G)$ of a graph $G$ is one less than the {\bf clique number} 
of $G$ and is the dimension of the largest complete subgraph which can occur in $G$. 
The {\bf inductive dimension} \cite{elemente11} 
${\rm ind dim}(G)$ of $G$ is the average over all dimensions of unit spheres. One has
$$  {\rm ind dim}(G) \leq {\rm max exp}(G)  \;  $$
with equality for discrete manifolds or complete graphs. The {\bf dimension expectation}
finally is defined as ${\rm dim exp}(G) = \frac{f'_G(1)}{f_G(1)}$, the {\bf average simplex cardinality}
${\rm average}(G) = f'_G(1)/f_G(1)$.  We have proven the inequality
$$  {\rm ind dim}(G) \leq 2 {\rm dim exp}(G)-1  \;  $$
which is true for all graphs and sharp as we have equality for complete graphs $G=K_n$
\cite{AverageSimplexCardinality}.
Finally, we have the {\bf cohomology dimension} ${\rm coho dim}(G)$ which is defined
in general as the largest $n$ for which the Betti number $b_{n}$ is non-zero. We also have
$$  {\rm coho dim}(G) \leq {\rm max dim}(G) $$
with equality for $d$-sphere manifolds. 

\paragraph{}
Let us compute a few of these dimensions for the graph complements $G_n$ of the cyclic graphs $C_n$. 

\begin{center}
\begin{tabular}{cccccc}
$G=\overline{C_n}$ &  ${\rm ind dim}(G)$  &  $2 {\rm dim exp}(G) - 1$   & ${\rm coho dim}(G)$ & ${\rm max dim}(G)$ \\ \hline
$4$  & $1.$  & $1.28571$  & $0$  & $1$  \\
$5$  & $1.$  & $1.72727$  & $1$  & $1$  \\
$6$  & $1.66667$  & $2.33333$  & $1$  & $2$  \\
$7$  & $2.$  & $2.86207$  & $1$  & $2$  \\
$8$  & $2.46667$  & $3.42553$  & $2$  & $3$  \\
$9$  & $2.88889$  & $3.97368$  & $2$  & $3$  \\
$10$  & $3.32381$  & $4.52846$  & $2$  & $4$  \\
$11$  & $3.75556$  & $5.0804$  & $3$  & $4$  \\
$12$  & $4.18801$  & $5.63354$  & $3$  & $5$  \\
$13$  & $4.62032$  & $6.18618$  & $3$  & $5$  \\
$14$  & $5.05265$  & $6.73903$  & $4$  & $6$  \\
\end{tabular}
\end{center}

\paragraph{}
The average simplex cardinality  $f'(1)/f(1)$ is computable because $f_{G_n}(1)$ 
are explicitly known hyper Fibonacci numbers and $f'_{G_n}(1)$ are multiples of 
standard Fibonacci numbers: 

\begin{lemma}
While $f_G(1)$ is a hyper Fibonacci number, 
the derivative $f_G'(1)$ for $G=G_n$ is equal to $(n+1) F(n)$, where $F(n)$ is 
the $n$'th Fibonacci number. The average simplex cardinality therefore is
$$   {\rm dim exp}(G_n) = (n+1) F(n)/F(n)^+   \; . $$
\end{lemma} 

\paragraph{}
Lets look at the dimensions for the graph complements $G_n^+$ of path graphs with $n$ vertices:

\begin{center}
\begin{tabular}{cccccc}
$G=G_n^+$ &  ${\rm ind dim}(G)$  &  $2 {\rm dim exp}(G) - 1$   & ${\rm coho dim}(G)$ & ${\rm max dim}(G)$ \\ \hline
$4$  & $1.$  & $1.5$  & $0$  & $1$  \\
$5$  & $1.46667$  & $2.07692$  & $1$  & $2$  \\
$6$  & $1.88889$  & $2.61905$  & $1$  & $2$  \\
$7$  & $2.32381$  & $3.17647$  & $0$  & $3$  \\
$8$  & $2.75556$  & $3.72727$  & $2$  & $3$  \\
$9$  & $3.18801$  & $4.2809$  & $2$  & $4$  \\
$10$  & $3.62032$  & $4.83333$  & $0$  & $4$  \\
$11$  & $4.05265$  & $5.38627$  & $3$  & $5$  \\
$12$  & $4.48499$  & $5.93899$  & $3$  & $5$  \\
$13$  & $4.91732$  & $6.4918$  & $0$  & $6$  \\
$14$  & $5.34965$  & $7.04458$  & $4$  & $6$  \\
\end{tabular}
\end{center}

\section{Cohomology} 

\paragraph{}
One of our motivations to look at the graphs $G_n$ was to do use them for
computations, like computing statistical properties or algebraic properties like
cohomology groups and also higher order cohomology groups. This brings us often to
the limit what a machine can do with current technology. 

\paragraph{}
When looking at the cohomology of $G_n$, each Betti number $b_k(G_n)$ stabilizes and eventually is zero
except $b_0(G_n)$ which remains always $1$. Fr $n=5$ on, we have connected graphs $G_n$, from $n=9$ on, 
we have simply connected graphs. From $n=11$, on we have the second Betti number $b_2=0$.
When looking at the cohomology we noticed quickly that we have spaces homotopic to spheres or 
bouquets of two spheres. There is a ``stable homotopy feature" in that modulo 
suspension we see a $3$-periodicity: sphere-sphere-point-sphere-sphere-point etc. 

\begin{center}
\begin{tabular}{c|c|l}
$G=G_n=\overline{C_n}$ &   $\chi(G)$              & $\vec{b}(G)$  \\ \hline
$3$  & $3$    &  $(3)$ \\
$4$  & $2$    &  $(2, 0)$ \\
$5$  & $0$    &  $(1, 1)$ \\
$6$  & $-1$   &  $(1, 2, 0)$ \\
$7$  & $0$    &  $(1, 1, 0)$ \\
$8$  & $2$    &  $(1, 0, 1, )$ \\
$9$  & $3$    &  $(1, 0, 2, 0)$ \\
$10$  & $2$   &  $(1, 0, 1, 0, 0)$ \\
$11$  & $0$   &  $(1, 0, 0, 1, 0)$ \\
$12$  & $-1$  &  $(1, 0, 0, 2, 0, 0)$ \\
$13$  & $0$   &  $(1, 0, 0, 1, 0, 0)$ \\
$14$  & $2$   &  $(1, 0, 0, 0, 1, 0, 0)$ \\
$15$  & $3$   &  $(1, 0, 0, 0, 2, 0, 0)$ \\
$16$  & $2$   &  $(1, 0, 0, 0, 1, 0, 0, 0)$ \\
$17$  & $0$   &  $(1, 0, 0, 0, 0, 1, 0, 0)$ \\
$18$  & $-1$  &  $(1, 0, 0, 0, 0, 2, 0, 0, 0)$ \\
$19$  & $0$   &  $(1, 0, 0, 0, 0, 1, 0, 0, 0)$ \\
$20$  & $2$   &  $(1, 0, 0, 0, 0, 0, 1, 0, 0, 0)$ \\
\end{tabular}
\end{center}

\paragraph{}
Here are the Euler characteristic and Betti numbers for the dual path graphs $G_n^+$. There is again
a $6$-periodicity for the Euler characteristic and a $3$-periodic pattern in the Betti number shift: 

\begin{center}
\begin{tabular}{c|c|l}
$G=\overline{L_n}=G_n^+$ &   $\chi(G)$              & $\vec{b}(G)$  \\ \hline
$4$  & $1$   &  $(1, 0)$ \\
$5$  & $0$   &  $(1, 1, 0)$ \\
$6$  & $0$   &  $(1, 1, 0)$ \\
$7$  & $1$   &  $(1, 0, 0, 0)$ \\
$8$  & $2$   &  $(1, 0, 1, 0)$ \\
$9$  & $2$   &  $(1, 0, 1, 0, 0)$ \\ \hline

$10$  & $1$  &  $(1, 0, 0, 0, 0)$ \\
$11$  & $0$  &  $(1, 0, 0, 1, 0, 0)$ \\
$12$  & $0$  &  $(1, 0, 0, 1, 0, 0)$ \\
$13$  & $1$  &  $(1, 0, 0, 0, 0, 0, 0)$ \\
$14$  & $2$  &  $(1, 0, 0, 0, 1, 0, 0)$ \\
$15$  & $2$  &  $(1, 0, 0, 0, 1, 0, 0, 0)$ \\ \hline

$16$  & $1$  &  $(1, 0, 0, 0, 0, 0, 0, 0)$ \\
$17$  & $0$  &  $(1, 0, 0, 0, 0, 1, 0, 0, 0)$ \\
$18$  & $0$  &  $(1, 0, 0, 0, 0, 1, 0, 0, 0)$ \\
$19$  & $1$  &  $(1, 0, 0, 0, 0, 0, 0, 0, 0, 0)$ \\
$20$  & $2$  &  $(1, 0, 0, 0, 0, 0, 1, 0, 0, 0)$ \\
$21$  & $2$  &  $(1, 0, 0, 0, 0, 0, 1, 0, 0, 0)$ \\
\end{tabular}
\end{center}

\paragraph{}
On our off the shelve workstation, we were able to compute all the cohomology groups up to $n=21$, 
where the simplicial complex of $\overline{C_n}$ has
$F_{21}=24475$ simplices. We have then already to compute the kernel of matrices of this size.
It is a computational challenge to go higher. Due to the cyclic symmetry, we know the
nature of the harmonic forms in the sphere case. In the case of the wedge sum of spheres, the harmonic
forms are more interesting. Since the kernel of an integer matrix always can be given by integer matrices,
we could also wonder about the arithmetic of the individual entries of the harmonic forms. They are certainly 
of geometric interest

\section{Wu characteristic}

\paragraph{}
We have seen $\chi(G_n) = 1-2\cos(\pi x/3)$.  
which is a 6 periodicity for Euler characteristic 
$$   \chi(G) = \omega_1(G) = \sum_{x} \omega(x) $$
summing over all complete subgraphs $x$ in $G$, where with $\omega(x) = (-1)^{{\rm dim}(x)}$.
There is now a 12 periodicity of the Wu characteristic
$$ \omega(G) = \omega_2(G) = \sum_{x \sim y} \omega(x) \omega(y)  \; , $$
summing over all pairs of intersecting simplices $x,y$ in $G$.
Since $\omega(K_{n+1}) = (-1)^n = {\rm dim}(K_n)$, the notation
is consistent.  We see that the Wu characteristic is $12$-periodic in $n$.
In the following table, we also compute the cubic Wu characteristic
$$  \omega_3(G) = \sum_{x \cap y \cap z \neq \emptyset} \omega(x) \omega(y) \omega(z) \; . $$
which sums up all possible triple interactions. 
It is always equal to the Euler characteristic $\chi(G) = \omega_0(G)$. 
The fourth order Wu characteristic 
$$  \omega_4(G) = \sum_{x \cap y \cap z \cap w \neq \emptyset} \omega(x) \omega(y) \omega(z) \omega(w)\;  $$
again agrees with $\omega_2(G)$. 

\begin{center}
\begin{tabular}{ccccc}
$G=\overline{C_n}$ &   $\chi(G)$   &  $\omega(G)$  & $\omega_3(G)$   & $\omega_4(G)$  \\ \hline
$4$  & $2$   &  $-2$  &  $2$  &  $-2$ \\
$5$  & $0$   &  $0$   &  $0$  &  $0$ \\
$6$  & $-1$  &  $5$   &  $-1$ &  $5$ \\
$7$  & $0$   &  $0$   &  $0$  &  $0$\\
$8$  & $2$   &  $-2$  &  $2$  &  $-2$ \\
$9$  & $3$   &  $3$   &  $3$  &  $3$ \\
$10$  & $2$  &  $2$   &  $2$  &  $2$ \\
$11$  & $0$  &  $0$   &  $0$  &  $0$ \\
$12$  & $-1$  &  $1$  & $-1$  &  $1$  \\
$13$  & $0$   &  $0$  & $0$   &  $0$\\
$14$  & $2$   &  $2$  & $2$   &  $2$ \\
$15$  & $3$   &  $3$  &  $3$  &  $3$  \\ \hline

$16$  & $2$   &  $-2$  &  $2$  &  $-2$ \\
$17$  & $0$   &  $0$   &  $0$  &  $0$ \\
$18$  & $-1$  &  $5$   &  $-1$ &  $5$   \\
\end{tabular}
\end{center}

\paragraph{}
We believe that $\omega_{2k}$ all agree and that $\omega_{2k+1}$ all agree. This is a property
which holds for discrete manifolds.
In some sense, the graphs $G_n$ behave like manifolds or manifolds with boundary.
In the case when $G_n^+$ is contractible which happens for $n=3d+1$, we deal with a manifold
with boundary. In the other cases we have manifolds or a wedge product of manifolds $(n=3d)$.

\paragraph{}
While Wu characteristic and its cohomology appear less regular \cite{CohomologyWu}
we again see a ``stable homotopy" feature in that there is a $12$-periodicity in $n$
which corresponds to a $4$-periodicity in the cohomological dimension $d$ of $G_n$. 
Here is the list for $G_n$:

\begin{center}
\begin{tabular}{c|c|l}
$G=\overline{C_n}$ &   $\omega(G)$              & $\vec{b}(G)$  \\ \hline
$3$  & $3$   &  $(3)$ \\
$4$  & $-2$  &  $(0, 2, 0)$ \\
$5$  & $0$   &  $(0, 1, 1)$ \\
$6$  & $5$   &  $(0, 0, 5, 0, 0)$ \\
$7$  & $0$   &  $(0, 0, 0, 0, 0)$ \\
$8$  & $-2$  &  $(0, 0, 0, 3, 1, 0, 0)$ \\
$9$  & $3$   &  $(0, 0, 0, 1, 4, 0, 0)$ \\
$10$ & $2$   &  $(0, 0, 0, 0, 2, 0, 0, 0, 0)$ \\
$11$ & $0$   &  $(0, 0, 0, 0, 0, 1, 1, 0, 0, 0)$ \\ 
\end{tabular}
\end{center}

\paragraph{}
We have computational power to compute the Wu Betti numbers of $G_n$ beyond $n=10$ yet, as the matrices
get too large. The Wu Betti numbers behave already more erratic with respect to suspension. 
While the Betti vector $(1,b_1,\dots, b_d)$ just shifts
$(1,0,b_1,\dots,b_d)$, the Wu Betti numbers behave more interestingly. 
Examples of transitions under suspension are $(0, 0, 7, 0, 0) \to (0, 0, 0, 3, 4, 0, 0)$
or $(0, 0, 6, 1, 0, 0, 0) \to (0, 0, 0, 2, 5, 0, 0, 0, 0)$. 

\paragraph{}
The {\bf f-matrix} of a simplicial complex $G$ is $f_{jk}(G)$ counting the number of
non-empty intersections of $j$ and $k$ dimensional simplices. Then 
$$ \omega(G) = \sum_{j,k} (-1)^{j+k} f_{jk} \; . $$
Lets denote with $F_n$ the f-matrix of $G_n$. We have 
$$ F_4=\left[\begin{array}{cc} 4 & 4 \\ 4 & 2  \end{array} \right],
   F_5=\left[\begin{array}{cc} 5 & 10 \\ 10 & 15  \end{array} \right] \; , $$
$$ F_6=\left[\begin{array}{ccc} 6 & 18 & 6 \\ 18 & 45 & 12 \\ 6 & 12 & 2  \end{array} \right], 
   F_7=\left[\begin{array}{ccc} 7 & 28 & 21 \\ 28 & 98 & 63 \\ 21 & 63 & 35  \end{array} \right] \; . $$
$$ F_8=\left[\begin{array}{cccc} 8&40&48&8 \\ 40&180&192&28 \\ 48&192 & 184&24 \\ 8&28&24&2  \end{array} \right] 
 F_{9}=\left[\begin{array}{cccc} 9&54&90&36\\ 54&297&450&162\\ 90&450&624&207\\ 36&162&207&63 \end{array} \right] \; .$$
In order to understand the periodicity of the Wu characteristic $\omega(G)$, we need to understand
a recursion for the f-matrices $F_n$. This needs still to be done. We see that each entry grows polynomially
but we have still to understand the recursion. We have for example $F_n(1,j) = j f_n(j)$ but already
in the second row, we there are intersections possible which does not mean subsets.

\section{Graph theory}

\paragraph{} 
In this section, we collect now some graph theoretical notions of $G_n$. They all pretty
much follow from the definitions. 
The graphs $G_n$ have constant {\bf vertex degree} $n-1$ and so in particular are all
{\bf regular}. Because $C_n$ are {\bf strongly regular} (not only $k=|S(x)|$ but $\lambda =|S(x) \cap S(y)|$ 
for $(x,y) \in E(G_n)$ and $\mu=|S(x) \cap S(y)|$ for $(x,y) \notin E(G_n)$), also $G_n$ are strongly regular.
The graphs $G_n$ and $G_n^+$ are both {\bf claw free} because they are 
graph complements of triangle free graphs. Bot $G_n$ and $G_n^+$ are connected and
non-bipartite for $n \geq 5$. The graphs are also {\bf vertex transitive} because the 
complement is. They are not edge transitive for $n \geq 8$. 
The graphs $G_n^+$ are never Eulerian as the vertex degrees are both $n-2$ or $n-3$. 
The graphs $G_n$ {\bf Eulerian} if and only $n$ is odd,
by {\bf Euler-Hierholzer} theorem which tells that even vertex degree everywhere is 
equivalent to the graph being Eulerian. The vertex
degree of $G_n$ is $n-3$ which is the number of vertices in $G_{n-3}^+$. 
The graphs $G_n$ are all {\bf Hamiltonian} for all $n \geq 5$, and so
are $G_n^+$ for $n \geq 5$. 
For even $n$, the result for $G_n$ follows from the {\bf Nash-Williams theorem}. 
Much simpler is just to write down a path: pick an integer $1<a<n-1$ with no common
divisor with $n$, then use the translation $x \to x+a$ to get a closed 
Hamiltonian path $x_k = k a \; {\rm mod}(n)$. To compare, it is much harder to show that all 
combinatorial manifolds are Hamiltonian \cite{HamiltonianManifolds}, a result which 
generalizes Whitney's result for 2-spheres. A {\bf bridge-less} graph (isthmus free graph)
contains no graph bridges (edges which when removed make the graph disconnected). 
Here is a summary:

\begin{thm}[Graph properties]
The following properties hold:
\begin{itemize}
\item The graphs $G_n$ are strongly regular, and vertex transitive. 
\item The graphs $G_n$ are circulant graphs and for $n>4$ are biconnected and bridge-less.
\item The graphs $G_n,G_{n}^+$ are connected for $n>4$ and disconnected for $n \leq 4$.
\item The graphs $G_n,G_n^+$ are Hamiltonian for $n>4$.
\item The graphs $G_{2m+1}$ are Eulerian, the graphs $G_{2m},G_{n}^+$ are not.
\item The graphs $G_n$ and $G_n^+$ are always claw-free.
\item The graphs $G_n$ are Cayley graphs of $\mathbb{Z}_n$ with two generators $a,a^{-1}$
\item The graphs $G_{n-3}^+$ are the unit spheres of $G_n$. 
\end{itemize}
\end{thm}

\paragraph{}
The chromatic number of a graph in general is dual to the clique covering number. 
The independence number is dual to the clique number. 
The independence number of $G_n$ is $2$ because the clique number of $C_n$ is $2$. 
The chromatic number of $G_{n},G_{n}^+$ is $[n/2]+1$, where 
$[x]$ denotes the integer part. In other words,
the chromatic number of $G_{2m},G_{2m}+,G_{2m-1},G_{2m-1}^+$ is $m$. 
The independence number of $G_m$ is constant $2$ because the clique number of its
graph complement $C_n$ is $2$.  
The clique covering number is the chromatic number of $C_n$ which is equal to $2$ 
or $3$ depending on whether the graph is even or odd. 
Therefore, the graphs  $G_{2n}$ are {\bf perfect graphs} while $G_{2n+1}$ are not. 
The graphs $G_{2n}^+$ are always perfect. 
By the {\bf perfect graph theorem} $G$ is perfect if and only if $\overline{G}$
is perfect. We know that $C_{2m}$ are perfect while $C_{2m-1}$ are not and that 
the path graph $C_{n}^-$, the dual of $G_n^+$ is always perfect. 
The {\bf Shannon capacity} of a graph $G$ is defined as 
$\Theta(G) = \lim_{k \to \infty} \alpha(G^k)$, where $G^k$ is the $k$'th strong power
of $G$ and $\alpha$ the independence number. 
The {\bf Wiener index} of a graph is $\sum_{i,j} d_{ij}/2$ where $d_{ij}$ is the distance
matrix. The graph distance matrix of $G_n$ is a circular matrix which has $1$ everywhere except 
in the side diagonal, where the value is $2$ and the diagonal, where the value is $0$. 
The sum over all entries is therefore $n^2+2n-n=n^2+n$. The {\bf Harary index} of a graph is
$\sum_{i \neq j} 1/d_{ij}$.  The Wiener and Harary indices are relevant for
example in chemical graph theory \cite{ChemicalGraphTheory}. 
The Harary index is $(n-1)^2-1$ for $G_n$ for $(n-1)^2$ for $G_n^+$ as can be seen by induction.
Let us summarize these observations which all pretty much depend on the definitions: 

\begin{thm}[Graph quantities]
The following quantities are known:
\begin{itemize}
\item The graphs $G_n$ have maximal dimension $[n/2]-1$ meaning clique number $[n/2]$. 
\item The graphs $G_n,G_n^+$ have diameter $2$ for all $n>4$. 
\item The Wiener index of $G_n$ is $n(n+1)/2$, for $G_n^+$ is $n(n+1)/2-1$, all for $n>4$.
\item The Harary index of $G_n$ is $(n-1)^2-1$, for $G_n^+$ it is $(n-1)^2$, all for $n>4$. 
\item The graphs $G_n^+$ have maximal dimension $[(n+1)/2]-1$, clique number $[(n+1)/2]$.
\item The graphs $G_n,G_n^+$ all have the independence number $2$.
\item The $G_{2m},G_{2m}^+,G_{2m-1},G_{2m-1}^+$ have chromatic number $m$.
\item The graphs $G_{2m},G_{2m}^+$ have clique covering number $2$,
      the graphs $G_{2m-1},G_{2m-1}^+$ have clique covering number $3$. 
\item The graphs $G_{2m-1}$ are not perfect but all $G_{2m}$ and all $G_n^+$ are perfect.
\item The Shannon capacity of both $G_n$ and $G_n^+$ are always $2$. 
\end{itemize}
\end{thm}

\paragraph{}
When taking the {\bf strong product} of cyclic graphs (introduce by Shannon
in \cite{Shannon1956} in 1956 as the graph with Cartesian vertex set and edges 
which project on both factors to points or vertices), the clique
number multiplies. It is under the large product (which is dual to the strong product) 
that we do no know what happens with the independence number. 
See \cite{ArithmeticGraphs,StrongRing} for
a bit more on the arithmetic of graphs. 
The large product of $G_n$ corresponds to the strong product of $C_n$. 
While the Shannon capacity of $C_{2n}$ is always $2$, 
one only knows $C_{5}=\sqrt{5}$ \cite{Lovasz1979} among odd cyclic graphs 
and so far only has estimates for the Shannon capacity of $C_{7}$. 
It is kind of funny that the Shannon capacity for cyclic graphs is a difficult
vastly unsolved problem while the Shannon capacity for the dual graphs is easy. 
As complements of sparse graphs $G_n$ are pure communication error graphs. It 
is pretty much the worst case as independence number $1$ means that the graph
is the dual of a graph with clique number $1$ and so must be a complete graph. 
Shannon capacity $1$ also characterizes complete graphs. Graphs with larger
Shannon capacity have at least capacity $2$. Already the cyclic graphs illustrate
that the Shannon capacity computation for general circulant graphs is difficult
in general. 

\paragraph{}
Lets look a bit about the metric space $G_n$ obtained
by the geodesic shortest distance function as metric. 
Due to symmetry, the {\bf graph distance matrix} of $G_n$ is a {\bf circulant matrix}
with $2$ in the side diagonal, $0$ in the diagonal and $1$ everywhere else. The diameter
$2$ of $G_n$ is a consequence that two unit spheres always intersect for $n>4$ as there
is then for any pair $(a,b)$ a third point $c$ such that the $C_n$ 
distances $|a-c|_{C_n}$ and $|b-c|_{C_n}$ are larger 
than $1$ so that $a,b,c$ are all connected in $G_n$. 

\paragraph{}
If $G_n=G_{3d+4}$ (which means that it is a $d$-sphere), then the
intersection is homotopic to a by $(d-1)$-sphere or then is contractible. 
If $G_n=G_{3d+3}$ (it is then a wedge sum $S^d \wedge S^d$), then the 
intersection is either a $d-1$-sphere, a $(d-2)$-sphere or then contractible. 
If $G_n=G_{3n+5}$ (which means it is a $d$-sphere), then the intersection 
is either a $d-2$ sphere or contractible. What happens is that any intersection of
unit spheres is always a smaller dimensional sphere or contractible. 
The Euler characteristic of any intersection $S(x_1) \cap \cdots \cap S(x_k)$ 
of spheres in $\mathbb{R}^n$ is always in $\{0,-1,1\}$.  
This property not only holds in $\mathbb{R}^n$. It also holds for Riemannian 
manifolds that are round spheres in Euclidean spaces or round spheres in Euclidean 
rotationally symmetric spheres. It also holds in homogeneous constant negative curvature
manifolds. It fails in general for Riemannian manifolds, even flat ones like flat 
Clifford tori $\mathbb{T}^d$ embedded in $\mathbb{R}^{2d}$. So, the spaces $G_n$ 
behave very much like universal covers of Riemannian manifolds with constant sectional
curvature. By the Killing-Hopf theorem, these are {\bf space forms}, which are
either spheres, Euclidean spaces or hyperbolic space. This will bring us to notions of 
curvature in the next section. 

\begin{thm}[Space form property]
If $x_1, \dots, x_k$ are vertices in $G_n$, then $H=S(x_1) \cap \cdots \cap S(x_k)$ is the graph
complement of a subgraph of $C_n$ in which the vertices $x_1,\dots,x_k$ are deleted. They are 
joins of spheres or contractible graphs and so themselves all either a sphere or contractible.
The genus $1-\chi(H)$ of $H$ is in $\{-1,0,1\}$. 
\end{thm}
\begin{proof}
This is just obtained by seeing disjoint union into Zykov join in the graph complement.
A unit sphere $S(x_k)$ consists of all points in $G_n$ connected to $x_k$. This means that
the graph complement consists of all points in $C_n$ not connected to $x_k$, meaning that $x_k$
has been taken off. More generally, the graph $S(x_1) \cap \cdots \cap S(x_k)$ is the 
graph complement of $C_n$ in which the vertices $x_1,\dots,x_k$ are deleted. 
Under Zykov join the genus multiplies. Since spheres have genus in $\{-1,1\}$ and 
contractible graphs have genus $0$, we know the genus of all intersections of unit spheres.
\end{proof}

\section{Differential geometry} 

\paragraph{}
Differential geometric notions come in when looking at {\bf curvature} in graph theory. 
In classical differential geometry, curvature is a rather technical notion involving the 
Riemann curvature tensor leading to the 
Gauss-Bonnet-Chern curvature which integrates on even dimensional manifolds to Euler characteristic. 
We have worked on this more recently \cite{DiscreteHopf,DiscreteHopf2,ConstantExpectationCurvature}
in the context of the open Hopf conjecture, a topic which also can be explored in the discrete.
With a very strong sectional curvature assumption (all embedded wheel graphs have positive curvature),
one gets spheres \cite{SimpleSphereTheorem}. 

\paragraph{}
In the discrete case, when looking at graphs, we have a much simpler task when looking at notions of 
curvature which add up to Euler characteristic.  There are no technical 
assumption whatsoever necessary for an analogue of Gauss-Bonnet-Chern result and integral geometric
considerations writing curvature as index expectation (which is possible both in the continuum as 
well as in the discrete) shows that the discrete case is the right analogue. We do not even have to 
worry about triangulations, as things are so robust. Take for example an $\epsilon^3$ 
dense set of points on the manifold then look at the intersection graph obtained by $\epsilon^2$ 
balls at these points. Averaging the discrete curvature over $\epsilon$ balls then produces the
Gauss-Bonnet-Chern curvature in the limit. The graphs $G$ approximating $M$ are messy, have huge
dimension but they are homotopic to $M$ for small enough $\epsilon$ and their unit spheres have
diameter $2$ or $3$ and are either homotopy spheres or contractible. We pretty much have the frame-work
we consider here in the case of $G_n$. 

\paragraph{}
Lets look at the graphs $G_n$ now. The high symmetry given by a vertex transitive circle action
produces a constant {\bf Euler-Levitt curvature}
$$ K(x) = 1+\sum_{k=0} (-1)^k \frac{f_k(S(x))}{k+2}  $$
which is defined for all graphs $G$ and satisfies 
the {\bf Gauss-Bonnet} relation $\sum_{x \in V} K(x) = \chi(G)$. 
This curvature is always zero for odd-dimensional discrete manifolds because of Dehn-Sommerville.
It is a discretization of the Gauss-Bonnet-Chern integrand for even-dimensional discrete manifold.
One can see this by the fact that $K(x)$ is the expectation $E[i_f(x)]$ of Poincar\'e-Hopf indices 
\cite{poincarehopf,indexexpectation}
$i_f(x)$ when averaging all locally injective functions and that the Gauss-Bonnet-Chern integrand
of an even-dimensional manifold also is an average over Morse functions obtained by Nash embedding $M$
into an ambient Euclidean space and taking the Morse functions obtained by restricting linear functions
on $M$. This procedure can also be done on a graph with $n$ vertices. Embed it into $E=\mathbb{R}^{n+1}$
then a random linear function in $E$ induces a coloring on the graph and so a Poincar\'e-Hopf index.
The average is the Euler-Levitt curvature \cite{Levitt1992,cherngaussbonnet} 
similarly as the Gauss-Bonnet-Chern curvature is the average over Morse indices. 

\begin{thm}
While $G_n$ has constant curvature $\chi(G_n)/n$ at every point, the graphs $G_n^+$
have nontrivial curvature distribution. 
\end{thm}
\begin{proof}
The transitive automorphism group shows the constant curvature for $G_n$. 
The curvature of $G_n^+$ is more interesting. It still adds up to $\chi(G_n^+)$ of course.
It can not be constant zero because some unit spheres are homotopic to spheres with positive Euler 
characteristic or contractible (with Euler characteristic $1$).
\end{proof}

\paragraph{}
The curvatures of $G_n^+$ converge to something
interesting.  For example, for $G_{12}^+$, we have the curvatures
$$ \left\{-\frac{1}{10},-\frac{1}{10},-\frac{1}{12},-\frac{1}{12},0,0,\frac{1}{30},\frac{1}{30},\frac{
   1}{15},\frac{1}{15},\frac{1}{12},\frac{1}{12}\right\}  \; . $$
The structure becomes only apparent however if one looks at large $n$ cases. The curvatures then 
converge to an attracting $6$-period cycle.

\begin{figure}[!htpb]
\scalebox{0.4}{\includegraphics{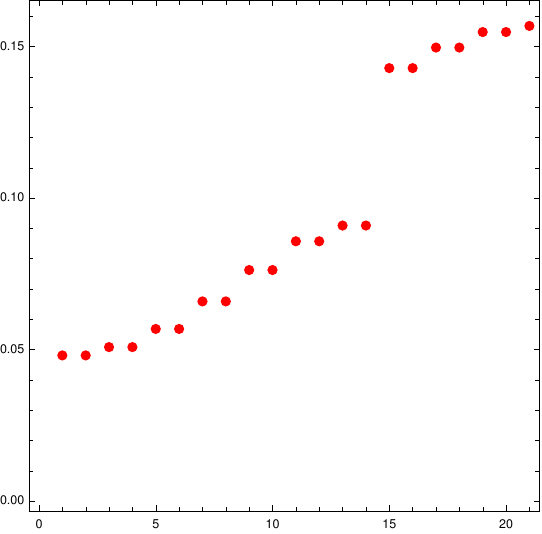}}
\scalebox{0.4}{\includegraphics{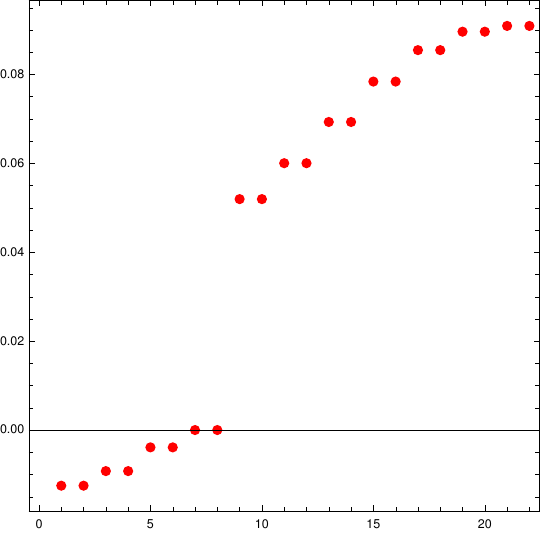}}
\scalebox{0.4}{\includegraphics{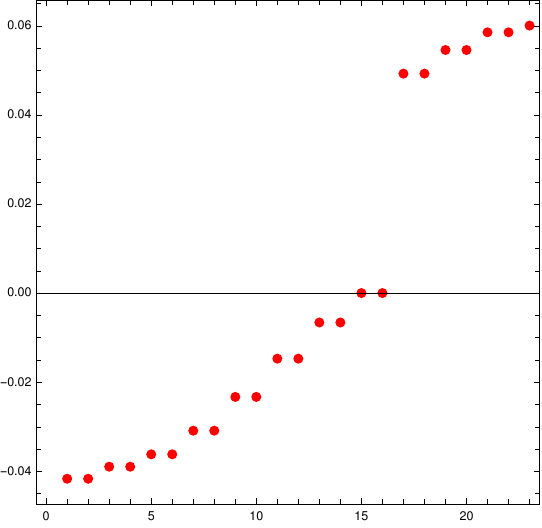}}
\scalebox{0.4}{\includegraphics{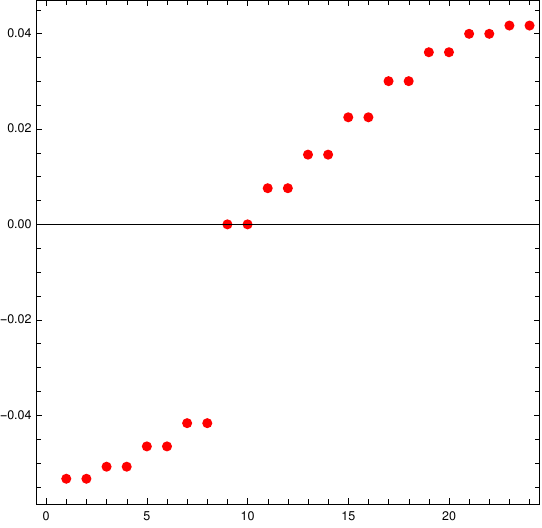}}
\scalebox{0.4}{\includegraphics{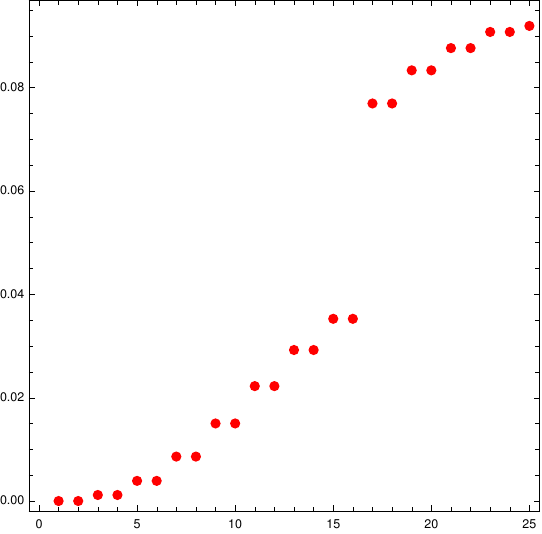}}
\scalebox{0.4}{\includegraphics{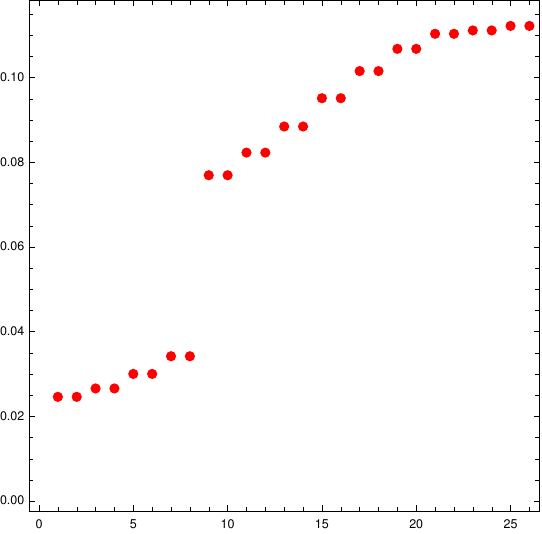}}
\label{Figure 13}
\caption{
The curvatures of $G_{n}^+$ appear first to be non-smooth as they converge to something developing a gap. 
The graph $G_{21}^+,G_{26}^+$ are homotopic to a $6$-sphere with curvatures adding up to $2$.
The graph $G_{22}^+$ is contractible with curvatures adding up to $1$.
The graphs $G_{23}^+,G_{24}^+$ are both homotopic to a $7$-sphere with curvatures 
adding up to $0$. The graph $G_{25}^+$ is again contractible.
}
\end{figure}

\paragraph{}
We see better what happens if we do not sort the curvature values. 
The next picture shows this:

\begin{figure}[!htpb]
\scalebox{0.4}{\includegraphics{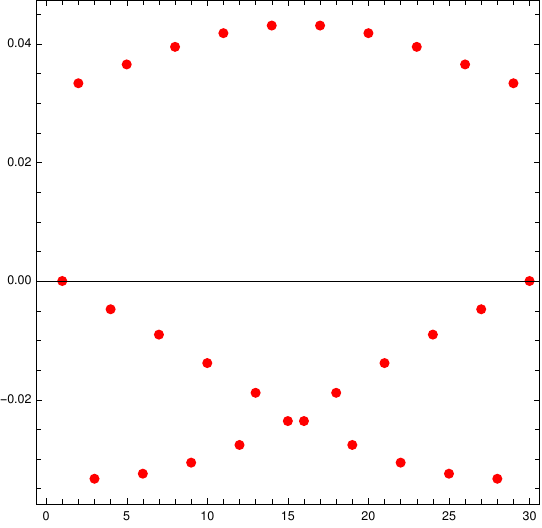}}
\scalebox{0.4}{\includegraphics{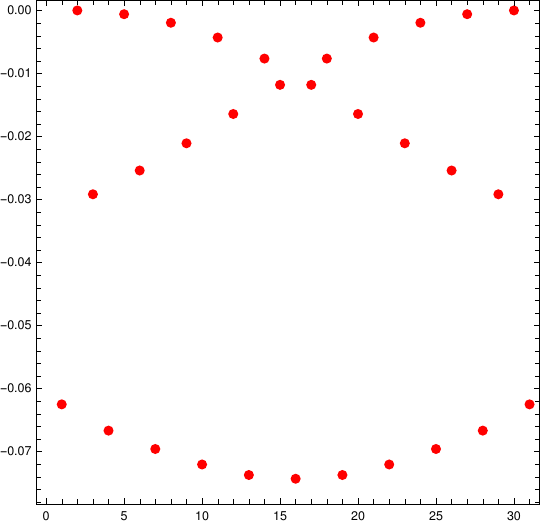}}
\scalebox{0.4}{\includegraphics{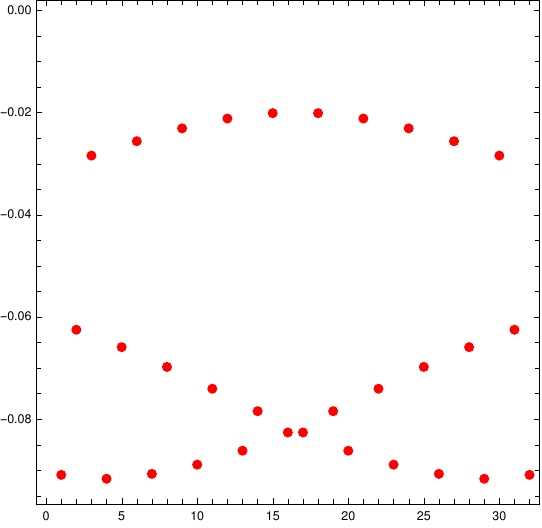}}
\scalebox{0.4}{\includegraphics{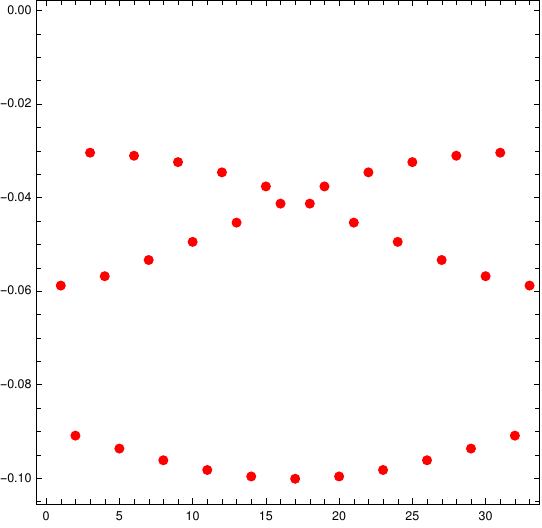}}
\scalebox{0.4}{\includegraphics{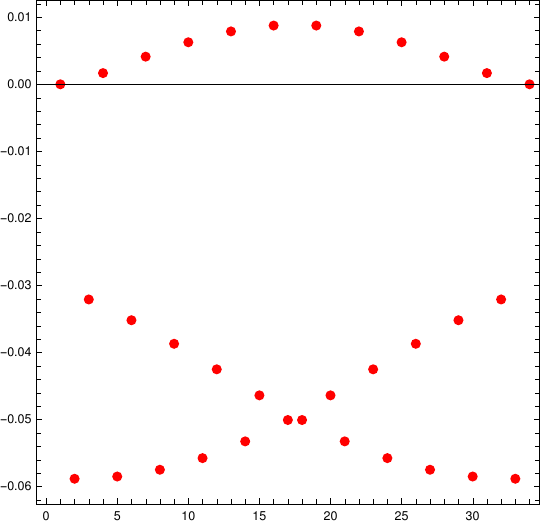}}
\scalebox{0.4}{\includegraphics{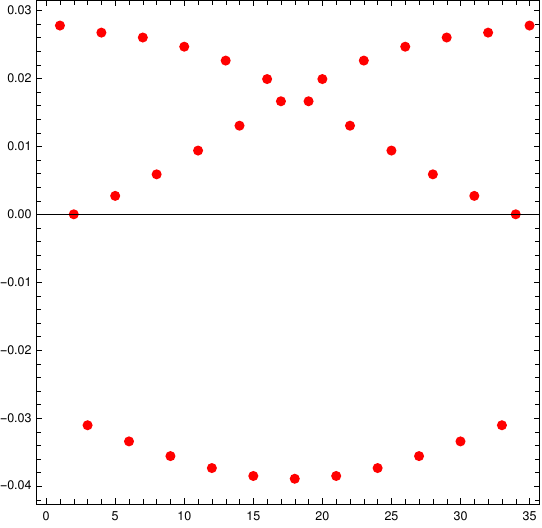}}
\scalebox{0.4}{\includegraphics{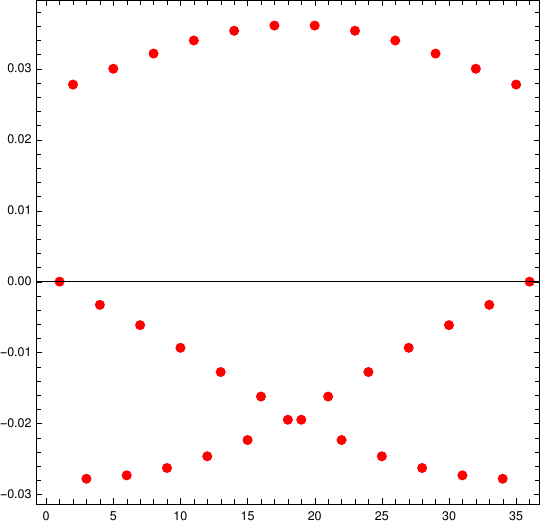}}
\scalebox{0.4}{\includegraphics{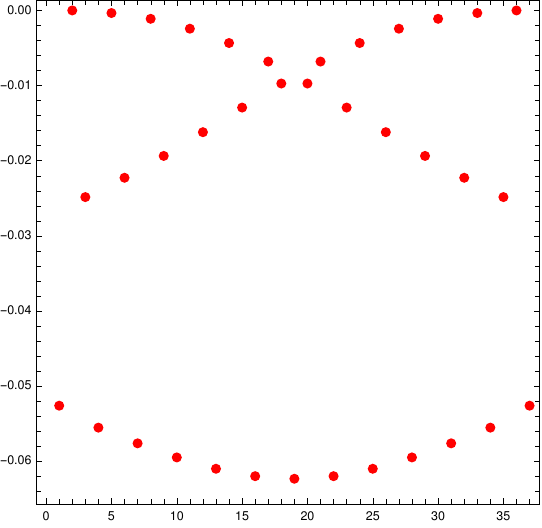}}
\scalebox{0.4}{\includegraphics{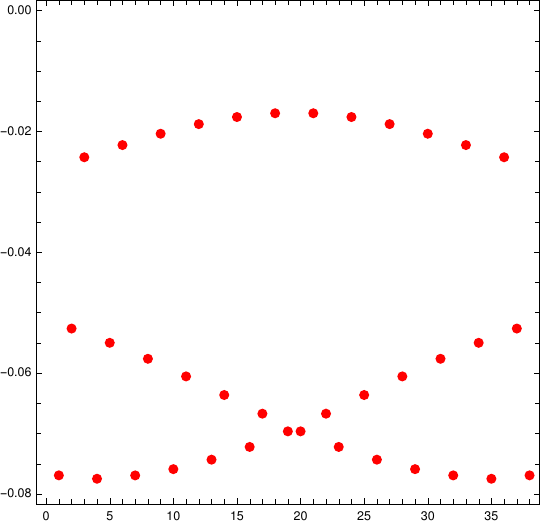}}
\label{Figure 14}
\caption{
The curvatures of $G_{n}^+$ for $n=30$ until $n=38$. 
}
\end{figure}

\paragraph{}
We have a convergence of curvature of $G_n^+$ in the limit $n \to \infty$
because every unit sphere in $G_n^+$ is the join of two unit spheres of the 
form $G_m^+$ and $G_l^k$ so that we can give explicit formulas for the 
curvature. We can best formulate this using generating functions and using the 
{\bf functional Gauss-Bonnet theorem}.

\begin{thm}[Curvature formula]
The curvature $K_n(k)$ of $G_n^+$ at the vertex $k$ is
$$K_n(k) = \int_{-1}^{0} f_{k-2}^+(t) f_{n-k-1}^+(t) \; dt  \;  $$
with $f_n^+(t) = f_{n-1}^+(t) + t f_{n-2}^+(t)$.
\end{thm}
\begin{proof}
To compute the Euler-Levitt curvature at $k$, we use the generating
function $f_{S(k)}(t)$ at the vertex $x$ and get 
$$ K(k) = \int_{-1}^0 f_{S_k}(t) \; dt  \; . $$
The unit sphere of $G_n^+$ at $k$ is the dual graph of the disjoint
union of two graphs $G_{k=2}^+$ and $G_{n-k=1}^+$. In the 
graph complement, the disjoint union becomes the Zykov join. 
And for the Zykov join, the simplex generating functions multiply.
\end{proof}

\paragraph{}
Now, for any fixed $k$, the sequences 
$$   n \to f_{k-2}(t) f_{n-k-1}(t) $$ 
satisfies the same recursion. 
We also know from Gauss-Bonnet for $G_n$ that 
$$ n \int_{-1}^0 f_{k-2}^+(t) f_{n-k-1}^+(t) \; dt = \chi(G_{n+3})  \; . $$

\paragraph{}
Using $u^2 = (1+4t)$ and $2u du = 4 dt$ and 
$g_{n-2}(u) =[(1-u)^n+(1+u)^n]/(u 2^n)$ 
we get the explicit formula 
$$  K_n(k) = \int_{-i\sqrt{3}}^1 g_{k-2}(u) g_{n-k-1}(u) u \; du \; . $$
We would still like to understand the limiting functions
$$  \kappa_{l}(x) = \lim_{n \to \infty} K_{6n+l}( [(6n+l) x])  \; .  $$

\paragraph{}
We see that we are able to compute the discrete analogue of the 
Gauss-Bonnet-Chern curvature for these high-dimensional spheres $G_n^+$
in a situation, where the curvature is not constant. Without
this functional knowledge, computing the curvature directly
is no more feasible already for numbers like $n=40$
where we deal with 12-dimensional spaces already. For $n=20$,
we have already $F(20)=15126$ simplices. 

\begin{figure}[!htpb]
\scalebox{0.5}{\includegraphics{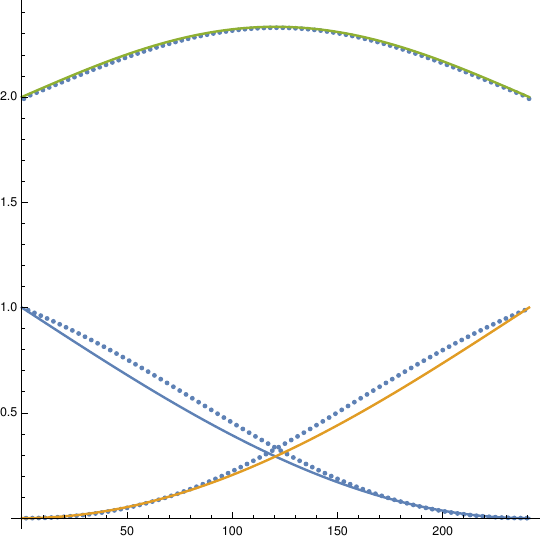}}
\scalebox{0.5}{\includegraphics{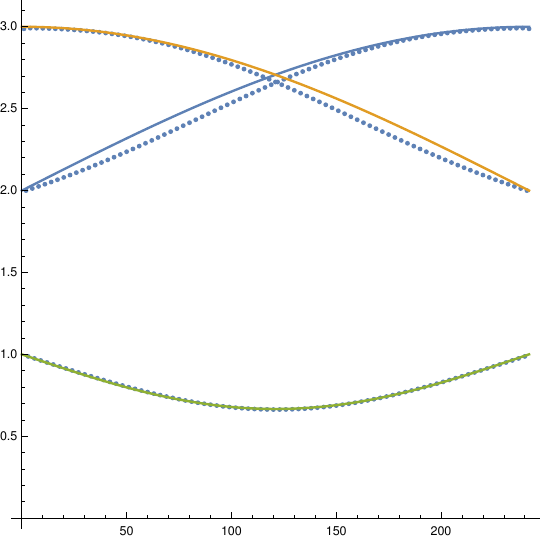}}
\scalebox{0.5}{\includegraphics{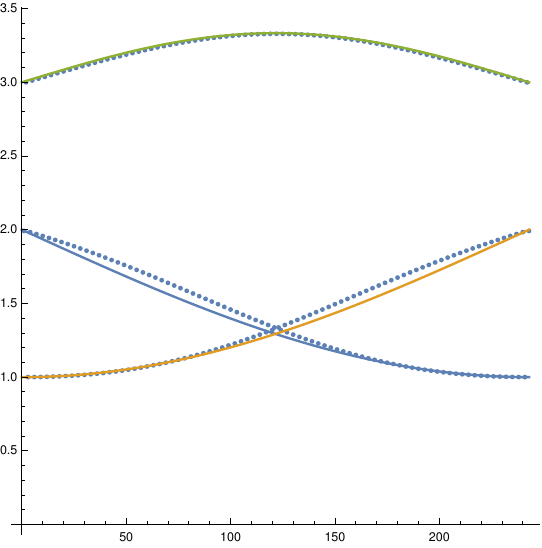}}
\scalebox{0.5}{\includegraphics{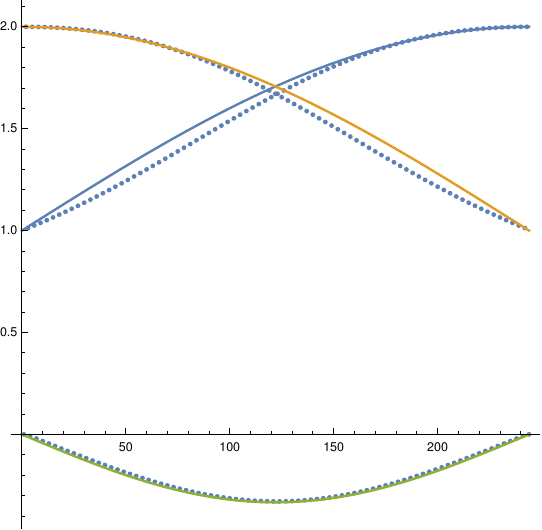}}
\scalebox{0.5}{\includegraphics{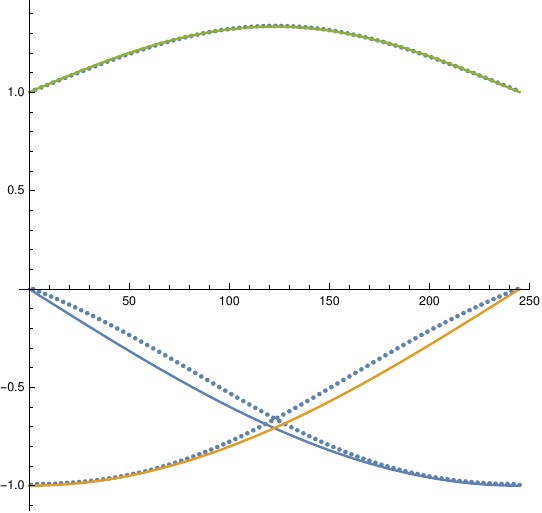}}
\scalebox{0.5}{\includegraphics{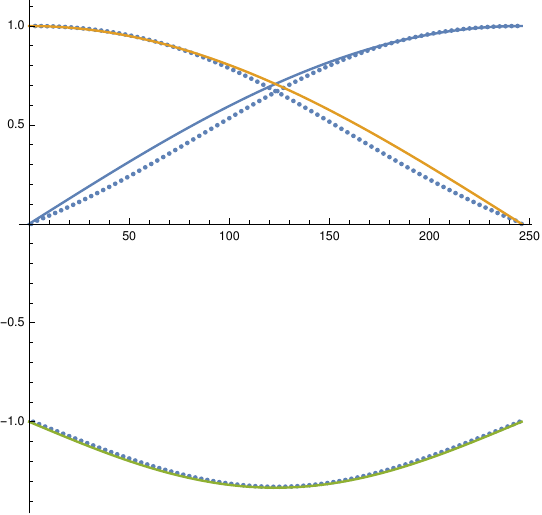}}
\label{Figure 15}
\caption{
The curvatures of $G_n^+$ for $n=240+k$ with $k=1,2,3,4,5,6$.
The picture is $6$-periodic. In each case we placed the best
first Fourier fit. The picture has been computed from
the generating function $f_n^+(t)$ using the formula
$K_n(k) = \int_{-1}^{0} f_{k-2}(t) f_{n-k-1}(t) \; dt$.
}
\end{figure}

\paragraph{}
While we see periodic attractor of a re-normalization 
map but we have not yet explicit expressions for the limiting
attracting curvature functions. Similarly as in 
\cite{KnillBarycentric,KnillBarycentric2} we get to a
result which can be seen as a {\bf central limit theorem} 
for Barycentric subdivision. Unlike in the cases studied,
we do not look at the spectral distribution (the result
there applies also here of course) but at the Euler-Levitt 
curvature distribution of graph complements 
$G_n^+$ of linear graphs $L_n$ in the Barycentric limit. 
The curvature distribution of the graph complements of 
$G_n$ of course is trivial because all unit spheres there
are the isomorphic graphs. 

\section{Renormalization} 

\paragraph{}
We have seen that for the one-dimensional interval $[0,1]$ when 
discretized as a linear graph $L_n = \{0/n,\dots,n/n\}$, the
graph complement $\overline{L}_n$ is either a point or sphere.
Let $\rho$ be the renormalization map, then $X_m=\rho^m(K_2)$ 
has length $n=2^m=3d+2$ and has $2^m+1$ vertices. Then 
$Y_n=\overline{X}_n$ is always a $d$-sphere, where $2^n+1=3d+3$ or
$2^n+1=3d+2$.  

\paragraph{}
Lets look at examples with small $n$. The graph complement $Y_1$ of $X_1=K_2$ 
and the graph complement $Y_2$ of $X_2$ is the house graph and so a $1$-sphere,
the path graph $Y_3$ of length $8$ is a $2$-sphere, 
the graph complement of the path graph $Y_4$ of length $16$
is a $5$-sphere and the graph complement of $Y_5$ of length $32$
is already a $10$ sphere. For even $n$ we get odd dimensional spheres $Y(n)$
for odd $n$ we get even dimensional spheres $Y(n)$. 

\paragraph{}
The discrete Euler curvatures $K_n(k)$ of
the graph complement $Y_n$ can be attached to the vertices of $X_n$ and
can be computed directly. This curvature is 
$$   K_n(k) = -\int_{-1}^0 f_{k+2}(t) f_{n-k-1}(t)  \; dt $$
with $f_k(t)$ satisfying the recursion 
$$  f_k(t) = f_{k-1}(t) + t f_{k-2}(t), f_{-1}(t)=f_0(t)=1 \; . $$

While 
$$  \kappa(x) = \lim_{n \to \infty} K_n([x n]) n $$ 
does converge weakly to the Lebesgue measure $2 dx$ along even subsequences $n=2k$
and to $0$ along odd sub-sequences. But if we split it up, we see more structure.

\paragraph{}
Define the functions on $[0,1]$ as 
$$  \kappa_l(x) = \lim_{n \to \infty} K_n(3 [x n]+l) n  \; . $$

\paragraph{}
If we do Barycentric refinements, then $n=2^k+1$ and we always have spheres. 
It is remarkable that we have now smooth non-trivial renormalization curvature limits.
For each $n$, there are three non-trivial functions which together add up to a constant function
$$ \kappa(x)=\kappa_0(x)+\kappa_1(x)+\kappa_2(x)  \; . $$ 
These curvatures build up Gauss-Bonnet curvatures onthese  large dimensional spheres. 

\begin{thm}[Renormalization limit]
The curvature function limits $\kappa_l(x)$ exist on $[0,1]$ for $l=0, \dots, 5$. 
\end{thm}
\begin{proof}
(Sketch) The explicit formula
$$  K_n(k) = \int_{-i\sqrt{3}}^1 g_{k-2}(u) g_{n-k-1}(u) u \; du \;  $$
with $g_{n-2}(u) =[(1-u)^n+(1+u)^n]/(u 2^n)$ are integrals of polynomials in $u$. 
These hyper-geometric functions can be written as sums of line integrals from 
$0$ to one of the $6$'the roots of unity. The fact that terms of the form 
$[(\pm 1 \pm i \sqrt{3})/2]^n$ and $(-1)^n$ appear show that the limit is 
$6$ periodic in $n$ for any choice of $x = [k n]$. 
\end{proof} 

\paragraph{}
We see that the limits are smooth functions and expect to prove this 
from explicit formulas for the limiting function. 
The curvature expressions are given by explicit hyper-geometric functions
which are integrals of polynomials in $u$.

\paragraph{}
When looking at the story from the renormalization perspective, where we do 
Barycentric refinement $L_{n+1} \to L_{2n+1}$, we can also generalize this to 
higher dimensions and make Barycentric refinements of higher dimensional graphs.
This is very difficult to investigate numerically because things explode very fast. 
To illustrate this, take the complete graph $G=K_3$ which is a triangle. 
The graph complement is the three point graph $P_3$ without edges
and Betti vector $(3,0,0)$. The graph complement of the 
Barycentric refinement of $K_3$ has the $f$-vector $(7,9,2,0,0)$ 
and Betti vector $(2,2,0)$. The second Barycentric refinement of $K_3$
has as a graph complement a graph with f-vector 
$(25, 240, 1154, 3022, 4485, 4026, 2438, 1065, 340, 78, 12, 1)$
which means a total of 16886 simplices. The Betti vector is 
$(1, 0, 0, 3, 26, 2)$. According to Euler-Poincar\'e, the super sum of
both the Betti vector and the f-vector is the same. It is 22. 
We were unable to compute both the $f$-vector and the Betti vector
of the graph complement of the third Barycentric refinement. 

\begin{figure}[!htpb]
\scalebox{1.0}{\includegraphics{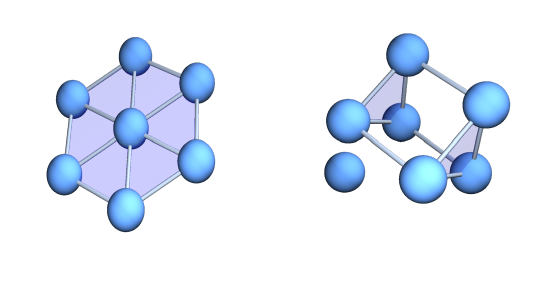}}
\scalebox{1.0}{\includegraphics{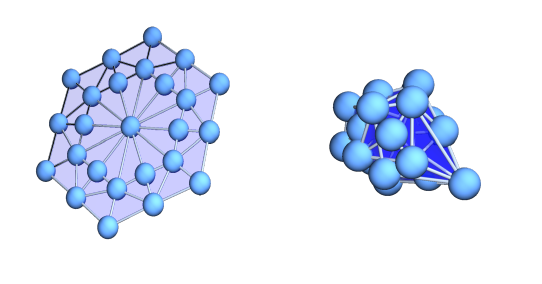}}
\label{Figure 16}
\caption{
The graph complements of the first and second Barycentric refinement
of a triangle. In the first case the dual is a disjoint union of a 
point and a wedge sum of two circles. In the second case, the graph
has 25 vertices but already 16886 simplices. 
}
\end{figure}

\section{Fixed point theory}

\paragraph{}
Since $G_n$ has a dihedral group symmetry $A_n = \mathbb{D}_n$ generated by a translations and reflections
$T \in A_n = {\rm Aut}(G_n)$. We can now look at the {\bf Lefschetz number} 
$\chi(G_n,T)$ of an automorphism $T \in A_n$. It is defined the super trace 
$$  \chi(G,T) = \sum_{k} (-1)^k {\rm tr}(T| H^k(G))  \; . $$
of the induced action of $T$ on the cohomology groups $H^k(G)$ which is by Hodge just
the null space of the Hodge operator $(d+d^*)^2$ restricted to the $f_k$-dimensional block 
of $k$-forms. 

\paragraph{}
A special case is when $T=Id=1$ is the identity map. The Lefschetz number is now 
$\chi(G_n,1) = \chi(G_n)$ which is the Euler characteristic. The Lefschetz fixed point theorem 
now becomes a special case, the discrete Euler-Poincar\'e theorem
\cite{Eckmann1,Edelsbrunner}. In the case of the circular graphs, we have the Lefschetz number $\chi(C_n,T)=0$ and consequently 
no fixed points. We can compute the Lefschetz number by looking at the fixed point and use 
the Lefschetz fixed point theorem \cite{brouwergraph}
$$ \chi(G,T) = \sum_{T(x)=x} i_T(x)   \;  $$
where $i_T(x) = (-1)^{{\rm dim}(x)} {\rm sign}(T:x \to x)$ is the Lefschetz index. 

\paragraph{}
The Lefschetz numbers are defined for every $T \in {\rm Aut}(G_n)$. The average
$$ \frac{1}{|{\rm Aut}(G_n)|} \sum_{T \in {\rm Aut}(G_n)}  \chi(G_n,T) \;   $$
can be interpreted as the Euler characteristic of the chain $G/{\rm Aut(G)}$. (The simplest
way to do that is to define the Euler characteristic of $G/{\rm Aut(G)}$ as such). 
The prototype example is $G=C_n$, which has the same automorphism
group like $G_n$ and where the Lefschetz number is zero for every translation but equal to $2$ for
every reflection. The average Lefschetz number is then always $1$. If we take only 
the subgroup $A_n=\mathbb{Z}_n$ of orientation preserving  maps, 
then the chain $C_n/\mathbb{Z}_n$ has Euler characteristic $0$ which is
indeed $(1/n) \sum_{T \in \mathbb{Z}_n} \chi(C_n,T)=0$. 
If we factor out the dihedral group $\mathbb{D}_n$ then 
$(1/2n) \sum_{T \in \mathbb{D}_n} \chi(C_n,T)=1$ and 
$C_n/\mathbb{D}_n$ can be seen as a point.  

\paragraph{}
In our case, the Lefschetz numbers can not be too complicated because the cohomology groups are not. 
In the case when we have $n$ not divisible by $3$, we deal with spheres and the Lefschetz number
can only be in $\{0,2\}$ depending on whether the map $T$ switches the sign of the Harmonic $d$-form
(in the sphere case, the harmonic $d$ forms form a one dimensional space only). In the case $n=3d$, 
we can have Lefeschetz numbers in $\{-1,0,1,2,3\}$ as we have $2$ harmonic $d$-forms. The maximum  $3$
is obtained if $T$ does not flip the sign of both forms. The minimum is obtained when $T$ switches
both signs. Here is a computation of the Lefschetz numbers for all $2n$ automorphisms of $G_n$ for small $n$. 
We see the structure. For all cases where $G_n$ is homotopic to $\mathbb{S}^{4d-1}$, all Lefschetz 
numbers zero.  For even $n=6k+2,6k-2$, the Lefschetz number of reflections are either $0$ or $2$
while for $n=6k$, the Lefschetz number of any reflection is $1$. The story clearly only depends on whether
$d$ is even or odd and what happens modulo $3$ for $n=3d+k$. 

\begin{thm}
The possible Lefschetz numbers show a $12$-periodic pattern. 
The average Lefschetz number is $1$ except for the cases $n=12k-1$ and $n=12k+1$
where the average Lefschetz number is $0$. 
\end{thm}

The program computing these numbers by  finding all fixed points and then adding 
up the Lefschetz indices is given below. We could push the computation to $n=37$
but needed to add graph specific code to find all the simplices in the graphs. 
Clique finding for $n=37$ graphs is already hard. Fortunately, there is a recursive
way to generate all complete subgraphs of $G_n$. 

\begin{center}
\begin{tabular}{lllll} 
 n & Rotations $T$ & Reflections $T$ & average  \\
4  & (0,2,0,2) & (2,0,2,0) & 1 \\
5  & (0,0,0,0,0) & (2,2,2,2,2) & 1 \\
6  & (0,2,3,2,0,-1) & (1,1,1,1,1,1) & 1 \\
7  & (0,0,0,0,0,0,0) & (2,2,2,2,2,2,2) & 1 \\
8  & (0,2,0,2,0,2,0,2) & (0,2,0,2,0,2,0,2) & 1 \\
9  & (0,0,3,0,0,3,0,0,3) & (1,1,1,1,1,1,1,1,1) & 1 \\
10 & (0,2,0,2,0,2,0,2,0,2) & (0,2,0,2,0,2,0,2,0,2) & 1 \\
11 & (0,0,0,0,0,0,0,0,0,0,0) & (0,0,0,0,0,0,0,0,0,0,0) & 0 \\
12 & (0,2,3,2,0,-1,0,2,3,2,0,-1) & (1,1,1,1,1,1,1,1,1,1,1,1) & 1 \\
13 & (0,0,0,0,0,0,0,0,0,0,0,0,0) & (0,0,0,0,0,0,0,0,0,0,0,0,0) & 0 \\
14 & (0,2,0,2,0,2,0,2,0,2,0,2,0,2) & (2,0,2,0,2,0,2,0,2,0,2,0,2,0) & 1 \\
15 & (0,0,3,0,0,3,0,0,3,0,0,3,0,0,3) & (1,1,1,1,1,1,1,1,1,1,1,1,1,1,1) & 1 \\
16 & (0,2,0,2,0,2,0,2,0,2,0,2,0,2,0,2) & (2,0,2,0,2,0,2,0,2,0,2,0,2,0,2,0) & 1 \\
17 & (0,0,0,0,0,0,0,0,0,0,0,0,0,0,0,0,0) & (2,2,2,2,2,2,2,2,2,2,2,2,2,2,2,2,2) & 1 \\
18 & (0,2,3,2,0,-1,0,2,3,2,0,-1,0,2,3,2,0,-1) & (1,1,1,1,1,1,1,1,1,1,1,1,1,1,1,1,1,1) & 1 \\
19 & (0,0,0,0,0,0,0,0,0,0,0,0,0,0,0,0,0,0,0) & (2,2,2,2,2,2,2,2,2,2,2,2,2,2,2,2,2,2,2) & 1 \\
20 & (0,2,0,2,0,2,0,2,0,2,0,2,0,2,0,2,0,2,0,2) & (0,2,0,2,0,2,0,2,0,2,0,2,0,2,0,2,0,2,0,2) & 1 \\
21 & (0,0,3,0,0,3,0,0,3,0,0,3,0,0,3,0,0,3,0,0,3) & (1,1,1,1,1,1,1,1,1,1,1,1,1,1,1,1,1,1,1,1,1) & 1 \\
22 & (0,2,0,2,0,2,0,2,0,2,0,2,0,2,0,2,0,2,0,2,0,2) & (0,2,0,2,0,2,0,2,0,2,0,2,0,2,0,2,0,2,0,2,0,2) & 1 \\
23 & (0,0,0,0,0,0,0,0,0,0,0,0,0,0,0,0,0,0,0,0,0,0,0) & (0,0,0,0,0,0,0,0,0,0,0,0,0,0,0,0,0,0,0,0,0,0,0) & 0 \\
24 & (0,2,3,2,0,-1,0,2,3,2,0,-1,0,2,3,2,0,-1,0,2,3,2,0,-1) & (1,1,1,1,1,1,1,1,1,1,1,1,1,1,1,1,1,1,1,1,1,1,1,1) & 1 \\
\end{tabular}
\end{center}

\paragraph{}
And here are the Lefschetz numbers of all $2$ automorphisms of $G_n^+$ for small $n$. 
The graph $G_n^+$ has the same automorphism group $\mathbb{Z}_2$ than the linear graph $L_n$ with $n$ vertices.
There is no translation any more. 

\begin{center}
\begin{tabular}{lll} 
n & Lefschetz of $T$ & average  \\ \hline
4 &  (1 ,  1) &  1  \\
5 &  (0 ,  2) &  1  \\
6 &  (0 ,  2) &  1  \\
7 &  (1 ,  1) &  1  \\
8 &  (2 ,  2) &  2  \\
9 &  (2 ,  0) &  1  \\
10 &  (1 ,  1) &  1  \\
11 &  (0 ,  0) &  0  \\
12 &  (0 ,  0) &  0  \\
13 &  (1 ,  1) &  1  \\
14 &  (2 ,  0) &  1  \\
15 &  (2 ,  2) &  2  \\
16 &  (1 ,  1) &  1  \\
17 &  (0 ,  2) &  1  \\
18 &  (0 ,  2) &  1  \\
\end{tabular}
\end{center}

\section{Code}

\paragraph{}
The following code generates the basis for all
the cohomology groups of a simplicial complex $G$ (a finite set of sets closed 
under the operation of taking finite non-empty subsets). 
The Wolfram language serves well also as pseudo code. Since no
libraries are used, it should be straightforward to rewrite the
code in any other programming language.  
The programs can be accessed from the LaTeX source of the ArXiv 
submission. We give it here for the situation at hand. We first
generate the complex $G$ for the graph $G_n$. Then we compute the 
Betti vector. The last line in the following block finally prints the spectral 
picture of the Hodge operator. We always start with a clean slate, clearing
all variables.

\begin{tiny}
\lstset{language=Mathematica} \lstset{frameround=fttt}
\begin{lstlisting}[frame=single]
ClearAll["Global`*"];
Generate[A_]:=Delete[Union[Sort[Flatten[Map[Subsets,A],1]]],1];
M=12; check[x_]:=Module[{t=True,m=Length[x]},
  Do[t=And[t,Less[1,Abs[x[[k]]-x[[l]]],M-1]],{k,m},{l,k+1,m}];t];
R=Generate[{Range[M]}]; G={};
Do[x=R[[k]];If[check[x],G=Append[G,x]],{k,Length[R]}]

n=Length[G]; Dim=Map[Length,G]-1;f=Delete[BinCounts[Dim],1];
Omega[x_]:=-(-1)^Length[x]; EulerChi=Total[Map[Omega,G]];
Orient[a_,b_]:=Module[{z,c,k=Length[a],l=Length[b]},
  If[SubsetQ[a,b] && (k==l+1),z=Complement[a,b][[1]];
  c=Prepend[b,z]; Signature[a]*Signature[c],0]];
d=Table[Orient[G[[i]],G[[j]]],{i,n},{j,n}];Dirac=d+Transpose[d];
H=Dirac.Dirac; f=Prepend[f,0]; m=Length[f]-1;
U=Table[v=f[[k+1]];Table[u=Sum[f[[l]],{l,k}];H[[u+i,u+j]],{i,v},{j,v}],{k,m}];
Cohomology = Map[NullSpace,U]; Betti=Map[Length,Cohomology]
EVPlot=ListPlot[Sort[Eigenvalues[1.0*H]]/M]
FormPlot=GraphicsGrid[Table[{ListPlot[Sort[Eigenvalues[1.0*U[[k]]]]/M, 
     Joined->True,PlotRange ->{0,1}]},{k,Length[U]}]]
\end{lstlisting}
\end{tiny}

\paragraph{}
And here is the code producing the connection cohomology groups which are 
more subtle and no more homotopy invariants. The following lines are independent
of the above for the sake of clarity. It computes the Wu cohomology of the 
{\bf Moebius strip} $G_7$ which is remarkably trivial. Having all cohomology 
groups trivial is not possible for simplicial cohomology. 

\begin{tiny}
\lstset{language=Mathematica} \lstset{frameround=fttt}
\begin{lstlisting}[frame=single]
ClearAll["Global`*"];
Generate[A_]:=Delete[Union[Sort[Flatten[Map[Subsets,A],1]]],1];
M=7; check[x_]:=Module[{t=True,m=Length[x]},
  Do[t=And[t,Less[1,Abs[x[[k]]-x[[l]]],M-1]],{k,m},{l,k+1,m}];t];
R=Generate[{Range[M]}]; G={};
Do[x=R[[k]];If[check[x],G=Append[G,x]],{k,Length[R]}]

Coho2[G_,H_]:=Module[{U={},n=L[G],m=L[H]},L=Length;
  c[x_]:=Total[Map[L,x]];
  Do[If[Greater[L[Intersection[G[[i]],H[[j]]]],0],
      U=Append[U,{G[[i]],H[[j]]}]],{i,n},{j,m}];
  U=Sort[U,Less[c[#1],c[#2]] &];u=L[U];l=Map[c,U];w=Union[l];
  b=Prepend[Table[Max[Flatten[Position[l,w[[k]]]]],{k,L[w]}],0];
  der1[{x_,y_}]:=Table[{Sort[Delete[x,k]],y},{k,L[x]}];
  der2[{x_,y_}]:=Table[{x,Sort[Delete[y,k]]},{k,L[y]}];
  d1=Table[0,{u},{u}];      d2=Table[0,{u},{u}];
  Do[v=der1[U[[m]]]; If[Greater[L[v],0],
    Do[r=Position[U,v[[k]]];
    If[r!={},d1[[m,r[[1,1]]]]=(-1)^k],{k,L[v]}]],{m,u}];
  Do[v=der2[U[[m]]]; If[Greater[L[v],0], 
    Do[r=Position[U,v[[k]]];
    If[r!={},d2[[m,r[[1,1]]]]=(-1)^(L[U[[m,1]]]+k)],{k,L[v]}]],
  {m,u}]; d=d1+d2; Dirac=d+Transpose[d];Hodge=Dirac.Dirac;
  Map[NullSpace,Table[Table[Hodge[[b[[k]]+i,b[[k]]+j]],
    {i,b[[k+1]]-b[[k]]},{j,b[[k+1]]-b[[k]]}],{k,L[b]-1}]]];
Betti2[G_,H_]:=Map[L,Coho2[G,H]];Coho2[G_]:=Coho2[G,G];
Betti2[G,G]
\end{lstlisting}
\end{tiny}

\paragraph{}
Here are the lines to compute the curvature of a complex. It works for
arbitrary simplicial complexes and is done in a functional way. 
For more details, see \cite{dehnsommervillegaussbonnet}. What is new
here is that we work directly with simplicial complexes and do 
need to generate Whitney complexes of unit sphere graphs. The unit sphere $S(x)$ 
in a complex is now a set of sets and not a simplicial complex but still has a 
simplex generating function $f_{S(x)}(t) = 1+\sum_{x \in A} t^{|x|}$, where
$|x|$ is the length of $x \in G$. {\bf Gauss-Bonnet} is then a theorem about
generating functions and tells for the simplex generating function of $G$: 
$$ f_G(t) = 1+\sum_{x \in V(G)} F_{S(x)}(t)  \; , $$
where $F(t)$ is the anti-derivative $F(t) = \int_0^t f(s) \; ds$ and $V(G)$ are
the zero-dimensional parts of $G$.  The curvature at a vertex $x$
is then a function too $F_{S(x)}(t)$. 
This functional generalization works also for Poincar\'e-Hopf 
\cite{MorePoincareHopf}. 

\begin{tiny}
\lstset{language=Mathematica} \lstset{frameround=fttt}
\begin{lstlisting}[frame=single]
ClearAll["Global`*"];
S[G_,x_]:=Module[{A={},n=Length[G]},SQ=SubsetQ; Do[y=G[[k]];         
  If[(SQ[x,y]||SQ[y,x])&&Not[x==y],A=Append[A,y]],{k,n}];A];
UnitSpheres[G_]:=Module[{A={}},Do[If[Length[G[[k]]]==1,
  A=Append[A,S[G,G[[k]]]]],{k,Length[G]}];A];
F[G_]:=If[G=={},{},Delete[BinCounts[Map[Length,G]],1]];
f[G_,t_]:=Module[{u=F[G]},1+u.Table[t^k,{k,Length[u]}]];
EulerChi[G_]:=1-f[G,-1];
Curvature[A_,t_]:=Integrate[f[A,u],{u,0,t}];
Curvatures[G_,t_]:=Module[{S=UnitSpheres[G]}, 
  Table[Curvature[S[[k]],u] /. u->t,{k,Length[S]}]]
Generate[A_]:=Sort[Delete[Union[Sort[Flatten[Map[Subsets,A],1]]],1]]; 
ComputeCurvatures[M_]:=Module[{},
  check[x_]:=Module[{t=True,m=Length[x]},
     Do[t=And[t,Less[1,Abs[x[[k]]-x[[l]]]]],{k,m},{l,k+1,m}];t];
  R=Generate[{Range[M]}]; G={}; L=Length;  Po=Position;
    Do[x=R[[k]];If[check[x],G=Append[G,x]],{k,Length[R]}];
  Curvatures[G,-1]];
Table[U=-ComputeCurvatures[k];{k,InputForm[U],Total[U]},{k,4,10}]

Print[f[G,t]==1+Total[Curvatures[G,t]]];
Print[EulerChi[G]== -Total[Curvatures[G,-1]]];
\end{lstlisting}
\end{tiny}

\paragraph{}
Because of the explicit formulas, we have a faster way to get the 
curvatures if we deal with $G_n^+$. This is the curvature formula.
The following code also is completely independent from any thing before. 
We then plot the list $n K_n(k)$ which produces the shape of the limiting
functions

\begin{tiny}
\lstset{language=Mathematica} \lstset{frameround=fttt}
\begin{lstlisting}[frame=single]
ClearAll["Global`*"];
g[n_]:=Expand[((1+u)^n-(1-u)^n)/(2^n*u)];
K[n_,k_]:=Integrate[g[k]*g[n-k+1]*u/2,{u,Sqrt[-3],1}]   
Curvatures[n_]:=Table[K[n,k],{k,n}]; Curvatures[7] 
ListPlot[100*Curvatures[100]] 
\end{lstlisting}
\end{tiny}

\paragraph{}
Here are the lines to compute the Lefschetz numbers of $G_n$. Also this code
is independent from anything before in order to make it easier to port it to 
other programming languages. The Wolfram language is perfect for pseudo code
but requires quite a bit of memory. We can not compute yet all the Lefschetz
numbers for $M=40$ for example. The task is simple: make a list of fixed point 
simplices and add up the indices to get the Lefschetz number using
the discrete Lefschetz fixed point theorem. 
In the following code, we have a custom fast computation of the Whitney 
copmlex of $G_n$. We know from the recursion that it can be obtained by 
using the Whitney complexes of $G_{n-1}$ and $G_{n-2}$ and add a point to 
$G_{n-2}$. What we essentially do when computing the Whitney complex of $G_n$
is compute all possible {\bf king configurations} on a one-dimensional chess board
of length $n$. On board of length $28$ for example there are 
$f_{28}(1)-1 = 710646$ possible king configurations. 

\begin{tiny}
\lstset{language=Mathematica} \lstset{frameround=fttt}
\begin{lstlisting}[frame=single]
ClearAll["Global`*"];
WhitneyDualCycle[M_]:=Module[{f,add,f1=Table[{k},{k,M}]},If[M<=3,f=f1];
    add[x_]:=If[Min[x]==1,Append[x,M-1],Append[x,M]];
    If[M>3,f=Union[Flatten[{WhitneyDualCycle[M-1],
                    Map[add,WhitneyDualCycle[M-2]]},1]]]; Union[f,f1]]
Do[ 
  L=Length; S=Signature; Po=Position; Ap=Append; W=WhitneyDualCycle[M];
  A=Table[RotateRight[Range[M],k],{k,M}]; B=Map[Reverse,A];
  T[x_,p_]:=Table[p[[x[[j]]]],{j,L[x]}];FixQ[x_,p_]:=Sort[T[x,p]]==x;   
  Fix[p_]:=Module[{r={}},Do[If[FixQ[W[[k]],p],r=Ap[r,W[[k]]]],{k,L[W]}];r];
  Q[p_]:=Total[f=Fix[p];Table[-(-1)^L[f[[k]]]*S[T[f[[k]],p]],{k,L[f]}]];
  U=Map[Q,A]; V=Map[Q,B]; Print[{M,U,V,Total[U+V]/(2L[A])}],{M,4,20}]
\end{lstlisting}
\end{tiny}

\paragraph{}
And here is the computation of the limiting Euler curvatures $\kappa_0(x), \kappa_1(x), \kappa_2(x)$
of the graph complements $G=G_{n+1}^+$ of length $n=600-2$ (point) ,$n=600-1$ (odd sphere),
$600$ (odd sphere),$600+1$ (point) ,$600+2$ (even sphere),$600+3$ (even sphere). 
In each case, the averaged curvature $\kappa(x) = \kappa_0(x)+\kappa_1(x)+\kappa_2(x)$ 
is constant whereas in reality, there are three separate parts of the graph in which 
the curvature function is different. This is a case, where we see much more structure
than in the continuum: now, the curvature at a node depends on the number theoretical
modulo 3 case and the nature of the geometry changes in a 6 periodic manner. 
We first compute the generating functions starting from $f_{-1},f_0$ so that $f_1$ 
is the third element. Then we compute the curvature at the node $k$ 
as $K_{n,k} = \int_{-1}^0 f_{k-2}(t) f_{n-k-1}(t) \; dt$ and 
then get the limiting curvature functions 
$$   \kappa_{b,a}(x) = \lim_{m \to \infty} (6m+a) K_{6m+a,[(6m+a) x]} \; $$
for $b=0,1,2$ in each of the cases $a=-2,-1,0,1,2,3$. 

\begin{tiny}
\lstset{language=Mathematica} \lstset{frameround=fttt}
\begin{lstlisting}[frame=single]
ClearAll["Global`*"];
R[a_,b_]:=Module[{},F={1,1};M=600+a;
 Do[n=Length[F];F=Append[F,Expand[F[[n]]+t*F[[n-1]]]],{M}];
 f[n_]:=F[[n+2]]; A=Table[Expand[f[k-2]*f[M-k-1]],{k,M}];
 Q[g_]:=M*Integrate[g,{t,-1,0}];
 S[x_]:=Table[x[[6k+b]],{k,Floor[(Length[x]-b)/6]}];
 U=Map[Q,A];V=S[U]; If[b==0,Print["EulerX=",Total[U]/M]];
 ListPlot[V,Frame->True,PlotRange->{-1.5,3.5}]];
GraphicsGrid[Table[Table[R[a,b],{b,0,2}],{a,-2,3}]]
\end{lstlisting}
\end{tiny}

\paragraph{}
This separation of the curvature parts explains the 
discontinuous curvature distribution function obtained
by just plotting an ordered list of curvatures of the
graph complements. These curvature distribution functions
converge in the limit $n \to \infty$. Who would think 
that a one dimensional path has such an interesting dual. 

\begin{tiny}
\lstset{language=Mathematica} \lstset{frameround=fttt}
\begin{lstlisting}[frame=single]
ClearAll["Global`*"];
R[a_]:=Module[{},F={1,1};M=600+a;
 Do[n=Length[F];F=Append[F,Expand[F[[n]]+t*F[[n-1]]]],{M}];
 f[n_]:=F[[n+2]]; A=Table[Expand[f[k-2]*f[M-k-1]],{k,M}];
 Q[g_]:=M*Integrate[g,{t,-1,0}]; U=Map[Q,A]; 
 ListPlot[Sort[U],Frame->True,PlotRange->{-1.5,3.5}]];
GraphicsGrid[Partition[Table[R[a],{a,-2,3}],3]]
\end{lstlisting}
\end{tiny}

\paragraph{}
Here is the code for computing the tree-forest ratio which the
ratio between the Fredholm and Pseudo determinant of the Kichhhoff matrix.
We also give the expression when taking the explicitly known eigenvalues. 
See \cite{cauchybinet} for more details on trees and forests and Cauchy-Binet. 

\begin{tiny}
\lstset{language=Mathematica} \lstset{frameround=fttt}
\begin{lstlisting}[frame=single]
ClearAll["Global`*"];
FirstNonZero[t_]:=-(-1)^ArrayRules[Chop[t]][[1,1,1]]*ArrayRules[t][[1,2]];
PDet[A_]:=FirstNonZero[CoefficientList[CharacteristicPolynomial[A,x],x]];
Fredholm[A_]:=A+IdentityMatrix[Length[A]];
TreeForestRatio[s_]:=Module[{K=KirchhoffMatrix[s]},Det[Fredholm[K]]/PDet[K]];
lambda[k_,n_]:=Sum[2 Sin[Pi*m*k/n]^2,{m,2,n-2}];
TreeForestRatioCycleDual[n_]:=Product[(1+lambda[k,n]^(-1)),{k,1,n-1}];

N[TreeForestRatio[GraphComplement[CycleGraph[100]]]]
N[TreeForestRatioCycleDual[100]] 
\end{lstlisting}
\end{tiny}

\section{Miscellanea}

\paragraph{}
In this quite inhomogeneous section, we collect a few mathematical parts and lose ends.
It is presented a bit differently and in parts slightly more general than in previous attempts of talks 
\cite{KnillILAS,KnillBaltimore,AmazingWorld} as some of the things have
evolved a bit. We especially review the various 
generalizations of Gauss-Bonnet, Poincar\'e-Hopf or Lefschetz by either looking
at valuations (energizations of the graph with some symmetry), or functionals
(like simplex generating functions). Curvature becomes elegant in the 
functional frame work. Of course, all this needs to be organized  once
in a single monograph; at the moment however we are much more interested in experiments, 
finding new patterns and relations, rather than working on consolidation. 

\paragraph{}
{\bf Graphs and complexes.} 
A finite simple graph $(V,E)$ is equipped with its {\bf Whitney complex},
the simplicial complex $G$ consisting of the vertex sets of all 
complete subgraphs. Not all finite abstract simplicial complexes are Whitney complexes
of finite simple graphs. An example is $G=C_3$ which is the $1$-skeleton of the 
$2$-dimensional complex $K_3$. The {\bf Barycentric refinement} of a complex is 
the order complex of $G$ consisting of all non-empty subsets of $2^G$ which are pairwise contained
in each other. The refinement is the Whitney complex of the graph in which 
the vertices are the elements of $G$ and two are connected if one is contained
in the other. We see that graph theory very well captures most interesting 
simplicial complexes. 

\paragraph{}
{\bf Homotopy}. The notion of homotopy is based on the notion of contractibility
which is inductively defined. A {\bf homotopy extension} of a finite simple graph 
chooses a contractible subgraph $A$ in $G$, adds a new vertex $x$ and adds all
connections from $x$ to all $a \in A$. The reverse operation is a {\bf homotopy reduction}.
It picks a vertex $x$ for which the unit sphere $S(x)$ is contractible and removes this vertex. 
Two graphs $G,H$ are {\bf homotopic} if there exists a finite set of homotopy extension
or homotopy reduction steps which brings $G$ to $H$. Being homotopic to a point is difficult
to check in general while contractibility, meaning that $G$ can be reduced to a point (without any 
extensions) can be done fast. The properties of graph homotopy are 
identical to the properties for classical homotopy deformations of nice topological spaces like CW
complexes. One could use the geometric realizations of $G$ with Whitney complex to prove these things
but all the mathematics can be done within combinatorics, meaning for example not having to use
the infinity axiom. The valuation property for 
Euler characteristic immediately shows that an extension $G \to G+_A x$ changes 
Euler characteristic by adding $1-\chi(A)$. This is zero for homotopy extensions as $\chi(A)=1$ then. 
Simplicial cohomology remains the same under homotopy deformation:
any cocycle on $A+x$ can modulo coboundaries on $A+x$
become a cocycle on $A$. The homotopy on graphs produces a chain homotopy for the 
differential complex. Natural operations like edge refinements or Barycentric refinements
are homotopies. The Wu cohomology however is not a homotopy invariant: for example,
$\omega(K_{n+1})=(-1)^n$. The graph $G_7$ has all cohomology groups zero while the 
discrete cylinder has also non-zero Wu cohomology groups. 

\paragraph{}
{\bf Differential geometry.} 
The $f$-vector $(f_0,f_1,\dots,f_d)$ of a finite set of sets $G$
encodes the number $f_k$ of $k$-dimensional sets in $G$. The {\bf simplex generating 
function} abbreviated as {\bf $f$-function} is defined as $f(t) = 1+\sum_{k=0}^d f_k t^{k+1}$. 
Its anti-derivative $K_G(t)=\int_t^0 f_G(s) \; ds$ is the {\bf curvature function}. 
Let $V(G)$ the set of $0$-dimensional sets in $G$, the sets of cardinality $1$. It is also
called the {\bf vertex set} of the complex $G$ and if $G$ is the Whitney complex of a graph $(V,E)$, 
then the vertex set is $V$.  When pushing the values $\omega(x)=(-1)^{\omega}$ 
from simplices $x$ equally to all vertices contained in $x$, we get the Levitt curvature.
This can now be formulated as a Gauss-Bonnet result \cite{dehnsommervillegaussbonnet}:

\begin{thm}[Gauss-Bonnet]
$f_G(t)-1 = \sum_{x \in V(G)} K_x(t)$, with curvature $K_x(t) = K_{S(x)}(t)$.
\end{thm}

\paragraph{}
For $t=-1$, this gives the {\bf Gauss-Bonnet formula}
$\chi(G) = \sum_{x \in G} K_x$ \cite{cherngaussbonnet}.
For graphs, this gives the {\bf Euler-Levitt curvature}
$K_x = 1+\sum_{k=0} (-1)^k \frac{f_k(S(x))}{k+2}$. In the special case when
the graph has circular unit spheres at every point it reduces to the famous
$K_x=1-{\rm deg}(x)/6$ which has first been considered
\cite{Heesch} in the context of the 4-color theorem. In the case of triangle free
graphs, $K_x = 1-{\rm deg}(x)/2$ where Gauss-Bonnet becomes the Handshake formula
$\sum_{x} {\rm deg}(x)/2 = |E|$, the number of edges as $|V|-|E|$ is then the 
Euler characteristic.  The general Gauss-Bonnet formula corresponds to the
Gauss-Bonnet-Chern formula for even dimensional compact Riemannian manifolds but
does not require any assumptions on the graph. We get the curvature distribution
of random networks for example if we know the distribution of the observables
$G \to f_k(G)$. We have in \cite{randomgraph} seen that for random $n$ Erd\"os-Renyi
graphs with edge probability $p$, the expectation ${\rm E}[f_k] = \B{n}{k+1} p^{\B{k+1}{2}}$
giving the formula
$$ {\rm E}_{n,p}[\chi] = \sum_{k=1}^n (-1)^{k+1} \B{n}{k} p^{\B{k}{2}} $$
which allows to see for example that in $n$ exponentially large positive or very 
negative Euler characteristic is possible.
If we assume that the unit spheres of a graph have unit spheres which have a known distribution,
we would also get a curvature expectation 
${\rm E}[K] = \sum_{k=0} (-1)^k \frac{ {\rm E}[f_k(S(x))]}{k+2}$. 

\paragraph{}
There are topological approaches to graph coloring \cite{knillgraphcoloring,knillgraphcoloring2}
builds on work of Fisk \cite{Fisk1977a,Fisk1977b}. 
For maximal planar 4-connected graphs, which are by a theorem of Whitney always
2-spheres (graphs for which every unit sphere is a circular graph and which 
when punctured becomes contractible), the total curvature is $2$ and the chromatic
number $3$ or $4$. In general, the chromatic number of a d-sphere $G$ is believed to be
$d+1$ or $d+2$ the reason being that one write $G$ as a boundary of a $d+1$ ball
then use edge refinements in the interior to make the interior Eulerian so that it
can be colored with the minimal amount of $d+2$ colors. This then colors the boundary. 
Now, in our case we have spheres $G_n$ which are not discrete manifolds and have very
small diameter $2$. The chromatic number is $[n/2]+1$ which grows faster than the dimension $d \sim n/3$. 
we need $[(3d+2)/2]+1 \sim 3d/2+2$ colors to color $G_n=G_{3d+2}$ while a discrete manifold
sphere needs $d+2$. A 3-sphere should have chromatic number 5, while our 3-sphere $G_{11}$ has
chromatic number $6$. The join $C_5 \oplus C_4$ for example which is a $3$-sphere with $f$-function
$f(x) = (1 + 9x + 29x^2 + 40x^3 + 20x^4 = (1 + 5x + 5x^2) (1 + 4x + 4x^2)$ and Euler characteristic
$1-f(-1)=0$ has chromatic number $5$. The Eulerian graph $C_4 \oplus C_4$  
(see \cite{KnillEulerian,knillgraphcoloring3} for the Eulerian theme) has minimal chromatic number $4$
(discrete 3-manifolds always have chromatic number at least 4 because of the existence of 3-simplices
with 4 vertices). 

\paragraph{}
Gauss-Bonnet can be generalized to {\bf valuations} 
$X(A)$ $\sum_k X_k f_k(A)$ for subgraphs $A$ of $G$ 
which in the case $X_k=(-1)^k$ is Euler characteristic. The theory of discrete
valuations \cite{KlainRota} is a combinatorial theory of valuations developed by
Hadwiger and others which is now part of integral geometry or geometric probability
theory. While in the continuum, one requires some structure for the theory to work
(like convex sets), the discrete works for all graphs. We have to do some Buffon 
type gymnastics in the continuum which allows to measure $k$ dimensional content
in a $n$ dimensional manifold for example (like the length of a surface for example,
which is approached by cutting it with random planes and measuring the length) 
while in the discrete we just count $f_k$. 

\paragraph{}
An even more general set-up, is a {\bf energized complex} $h:G \to \mathbb{K}$,
where $\mathbb{K}$ is any ring. This is a frame work which appears quite often also
in graph theory, like when looking at weighted graphs or when using orthogonal 
graph presentations like the Lovasz umbrella which allows to get upper bounds for the
Shannon capacity of a graph. The generality of assigning rather general data to any 
part of a simplical complex does not prevent to look at notions of c9urvature. 
Define for $v \in V(G)$ the curvature
$K_v = \sum_{x,v \subset x} h(x)/|x|$ which means distributing the energy of $x$
equally to all $|x|$ zero-dimensional parts in $x$. The Gauss-Bonnet formula is 
too obvious as moving the energy from a simplex equally to each of its zero dimensional
parts of course does not change the energy. Adding up all the energies of all the 
contributing simplices at a vertex is then the curvature. This can be done also 
by using other distributions, the extreme case being when distributing the energy to 
a minimum of some ordering, like for example given by a scalar function on the vertices. 
This is then Poincar\'e-Hopf. 

\paragraph{}
Even more general is to generalize this to {\bf multi-linear valuations} 
$X(A) = \sum_{k,l} X_{k,l} f_{k,l}(A)$ like
{\bf Wu characteristic} $\omega(A) = \sum_{k,l} (-1)^{k+l} f_{k,l}(A)$  
where $f_{k,l}(A)$ is the number of pairs of $k$ simplices and $l$ simplices
in $A$ which have non-empty intersection. One can then even look at even
more generality and see $\omega(A,A)$ as a self interaction generalizing 
a more general interaction part 
$$  \omega(A,B) =\sum_{k,l} (-1)^{k+l} f_{k,l}(A,B) $$ 
where $f_{k,l}(A,B)$ is the number of pairs $(x,y)$ of simplices $x$ in $A$ and 
$y$ in $B$ which intersect. It even generalizes further by not insisting on having 
the same ``energy" value for every $k$-simplex.  See \cite{EnergizedSimplicialComplexes,
EnergizedSimplicialComplexes2,EnergizedSimplicialComplexes3}. For all these
notions there is a Gauss-Bonnet relation. For example, for the Wu characteristic 
$$ \omega(G) = \sum_{x \sim y} \omega(x) \omega(y) \;  $$
one has the curvature $K(v) = \sum_{x, v \in x} \kappa(x)/|x|$, where
$\kappa(x)$ is the simplex curvature $\kappa(x) = \sum_{y \subset x} (-1)^k f_{k-1}(y)$.
We have seen earlier \cite{valuation} that the Wu curvature is the same than the 
Euler characteristic curvature if $G$ is a discrete manifold. The Wu curvature is 
$$ K(v) = \sum_{x \sim y, v \in x} \omega(x) \omega(y)/|x| \; . $$ 

\begin{thm}[Wu Gauss-Bonnet]
$\omega(G) = \sum_{v \in V(G)} K(v)$. 
\end{thm}

\paragraph{}
Since we have a $12$-periodicity in the Wu-characteristic, we expect a 12-periodic 
fixed cycle for the Wu curvature. Since we have no recursion yet for the f-matrix
of $G_n^+$, we can not compute the curvature yet for larger $n$.

\paragraph{}
{\bf Poincar\'e-Hopf angle}.
If $g$ is a function $G: \to \mathbb{R}$ which is locally injective
in that $g(x) \neq g(y)$ if $x \subset y$ or $y \subset x$, then 
it defines an total order the $0$-dimensional parts on every simplex 
with $v < w$ if $f(v) < f(w)$. This defines a map $T: G \to V$. 
Given a function $h:G \to \mathbb{R}$. 
we can send $(-1)^{{\rm dim}(x)}$ to the largest element $v \in V$ in $x$. All these
values add together to the Poincar\'e-Hopf index
$i_g(v) = \sum_{v \in x} h(x)$. We have now a Poincar\'e-Hopf
theorem for gradient vector fields again formulated for simplex 
generating functions $f_G(t)$. Let $S_-(v) = \{ w \in V(S(v)), g(w)<g(v) \}$
be the part of the unit sphere $S(v)$.

\begin{thm}[Poincar\'e=Hopf]
$f_G(t) = 1+t \sum_{v \in V} f_{S_-(v)}(t)$.
\end{thm}

A special case is if $G$ is the Whitney complex of a graph $(V,E)$ and $h:V \to \mathbb{R}$
is a locally injective function. Then $i_h(v) = 1-\chi(S_-(v))$
See \cite{poincarehopf,PoincareHopfVectorFields,MorePoincareHopf}. 
Also this idea can be generalized to multi-linear valuations \cite{valuation} and there
is a link to Gauss-Bonnet. 
Gauss-Bonnet can be seen as an expectation of Poincar\;e-Hopf \cite{indexexpectation}
when averaging over a probability space of functions $f$ like for example averaging
over all coloring functions with a minimal amount of colors \cite{colorcurvature}.

\paragraph{}
{\bf Brouwer-Lefschetz angle}.
We have illustrated above the Brouwer-Lefschetz theme and seen how it
relates with cohomology. If $G$ is a simplicial complex, and $T$ is an 
automorphism of $G$, then we have also an induced map $U_T$ on 
cohomology. The Lefschetz number $\chi(G,T)$ is defined as 
the super trace of $U_T$ on cohomology. Let $H_k(G)$ denote the 
linear space of harmonic $k$-forms n $G$ and $T_k$ the induced map 
on $H_k$, then 
$$ \chi(G,T) = \sum_{k} (-1)^k {\rm tr}(T_k)  \; . $$
By {\bf Euler-Poincar\'e}, we have $\chi(G)= \chi(G,Id)$. 

\paragraph{}
The fixed point set ${\rm Fix}(G,T)$ 
consists of all simplices $x \in G$ which are fixed by $G$. 
The index of a fixed point is $i_T(x) = (-1)^{\rm dim}(x) {\rm sign}(T|x)$
where the sign of $T$ on $x$ is the sign of the permutation induced
on the finite set $x$. 

\begin{thm}[Brouwer Lefschetz]
$$  \chi(G,T) = \sum_{x \in {\rm Fix}(G,T)} i_T(x) \; $$
\end{thm}

The proof is in \cite{brouwergraph}. As pointed out in the introduction
the simplest way to prove this is by applying a heat flow and use the McKean-Singer
super symmetry \cite{knillmckeansinger} which tells that the set of 
non-zero eigenvalues on even-dimensional forms is the set of non-zero
eigenvalues on odd-dimensional forms implying ${\rm str}(L^m)=0$ for 
all positive powers $m$. The McKean-Singer formula for the Hodge Laplacian 
$L=(d+d^*)^2$ tells then 

\begin{thm}[McKean Singer]
${\rm str}(e^{-t L})  = \chi(G)$. 
\end{thm}

One can see this by taking the Dirac operator $D=(d+d^*)$ serious and
use the analogous proof in the continuum \cite{Cycon}.
It maps even forms to odd forms and vice versa. 

\paragraph{}
There is also a Lefschetz fixed point story for Wu characteristic. But 
the computation is harder and the Lefschetz numbers are richer because
the cohomology is richer. What happens is that the super trace of an 
automorphism $T$ induced on Wu cohomology is the sum of the indices of 
the intersecting fixed point pairs with respect to $T$. Let us denote the 
Wu Lefschetz number of $T$ with $\omega(G,T)$. Then 

\begin{thm}[Wu Lefschetz]
$$  \omega(G,T) = \sum_{x \in {\rm Fix}(G,T)} i_T(x) \; $$
\end{thm}

For example, in the case $G=G_6$, the Wu Lefschetz numbers are
given by the 12 numbers 
$$  (5, -1, 1, 0, -1, 2, 1, 3, -1, 2, 0, 1) $$
belonging to the $12$ automorphisms of $G$. 
The Wu-Betti vector is here $(0,0,5,0,0)$.  For $G=G_7$, all Lefschetz numbers are zero
simply because the cohomology is trivial. 

\paragraph{}
We have pointed out in \cite{DiscreteAtiyahSingerBott} that whenever we have
a McKean-Singer symmetry given by a symmetry  Dirac operator $D=d+d^*$ then 
there is both an Atiyah-Singer type result (which is Gauss-Bonnet like) as
well as an Atiyah-Bott type result (which is Lefschetz-Brouwer) like). 
Wu characteristic and its cohomology are examples for such a general view.
Other examples are obtained by Lax deforming the matrix $D$ using an isospectral
flow \cite{IsospectralDirac}. In any of theses cases, the heat flow for the operator 
$L=D^2$ produces the proof. Euler-Poincar\'e for Euler characteristic or
Wu characteristic is the special case for the identity transformation
$T$. Despite the simplicity of the set-up, we still do not understand well
the Wu characteristic cohomology as well as the Lefschetz numbers there. 
We also do not know the periodicity yet for the  Lefschetz story for quadratic simplicial cohomology.
On can wonder whether there is a relation with other periodicity results known in 
stable homotopy theory. It is not so far fetched as the homotopy theory of orthogonal groups
is closely related to the homotopy theory of spheres and in our case, we have spheres realized 
with a group structure which in principle could be used to act on other spheres. 

\paragraph{}
{\bf Hodge spectrum.} 
When computing the spectrum $\sigma(L)/n$ of the {\bf Hodge Laplacian} $L = (d+d^*)^2$ of 
the graphs $G_n$ for large $n$, we see that the spectrum appears to converge to a nice 
smooth function for $n \to \infty$. We have investigated the same problem
in the case of $C_n$ and there convergence as a consequence of a general central limit theorem where
because of Fourier theory one has explicit eigenvalues both for $0$ 
forms and $1$-forms.  The spectrum of the $0$-forms converges so, but in the limit of larger and
larger dimensional spheres, this becomes less and less visible. 
The McKean super symmetry used in 1 dimensions shows that the $0$-dimensional non-zero spectrum
is the same than the $1$-dimensional non-zero spectrum. The number of zero 
eigenvalues of $L_0$ is then the number of components $b_0$ and 
the number of zero eigenvalues of $L-1$ is the genus $b_1$, the number of ``loops". 

\paragraph{}
Now, for the graphs $G_n$, we see again convergence of the spectrum when rescaled properly,
but the method used in the case of Barycentric refinements does not apply because now.
The dimensions of the simplicial complexes increases in $G_n$. In the Barycentric refinement, 
we had convergence of the spectral density of states on every sector of $k$-forms. 
\cite{KnillBarycentric,KnillBarycentric2}.
There are more open questions which are spectral.
Unrelated one can also ask what happens with the spectrum of the connection Laplacians of $G_n$.
We measured until $n=7$ that the connection graphs of $G_n$ have the same cohomology than $G_n$

\paragraph{}
Let us come back to the spectrum of the Hodge Laplacian $L(G_n)$ of the graphs $G_n$. 
Two blocks of the Hodge Laplacian are well known: the Kirchhoff Laplacian as well as the
highest block are {\bf circulant matrices} allowing 
Fourier theory to compute the eigenvalues. This works for $0$-forms: in a Fourier basis, the eigenvalues of the
Kirchhoff matrix of $G_n$ are explicitly given as
$$  \lambda_{k,n} = \sum_{m=2}^{n-2} 1-\cos(2\pi m \frac{k}{n}) 
                  = \sum_{m=2}^{n-2} 2\sin^2(\pi m \frac{k}{n}) \; . $$
We have $\lambda_{k,n} \leq n$.

\begin{figure}[!htpb]
\scalebox{0.8}{\includegraphics{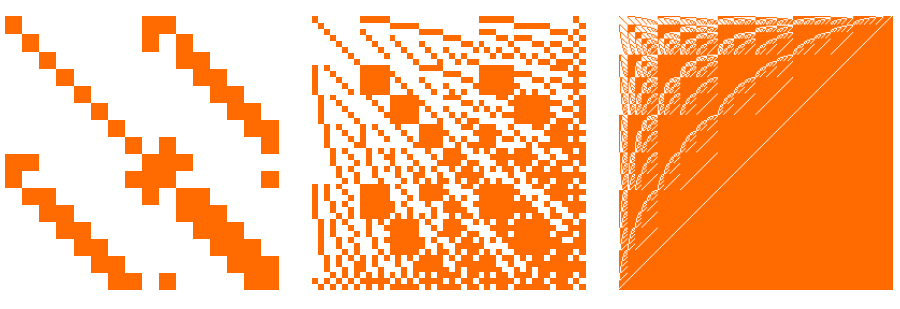}}
\label{Figure 17}
\caption{
The connection Laplacians of the $C_8,G_8$ and $K_8$.
The complexes $G_n$ are highly connected.
}
\end{figure}

\paragraph{}
{\bf Zeta function aspect}.
Th zeta function theme related to spectral questions. It relates to harmonic analysis and 
complex analysis because we deal with analytic functions. 
Given a matrix $A$ with non-negative spectrum, one can 
associate a Zeta function with the non-zero spectrum. 
In the case of the Hodge Laplacian $H=D^2$ we have to discard
the zero spectrum which is associated to cohomology. The zeta function
is 
$$ \sum_{k} \lambda_k^{-s} $$ 
Already in the case of the circle $\mathbb{T}$, where the Dirac operator
$i d/dx$ has spectrum $n$ with $n \in \mathbb{Z}$ and the Laplacian
has eigenvalues $n^2$, we have to discard the eigenvalue $0$ and take
the square $D^2$ to get non-negative eigenvalues. It is custom to replace
$s$ with $s/2$ to get then the Riemann zeta function $\sum_n n^{-s}$.  
For counting Laplacians or for connection Laplacians of one-dimensional 
complexes, we know that the zeta function satisfies a functional equation.
We would like to understand of course the Zeta functions in the Barycentric
limits. This has been explored so far only for complexes coming from cycle 
graphs. 

\begin{figure}[!htpb]
\scalebox{0.2}{\includegraphics{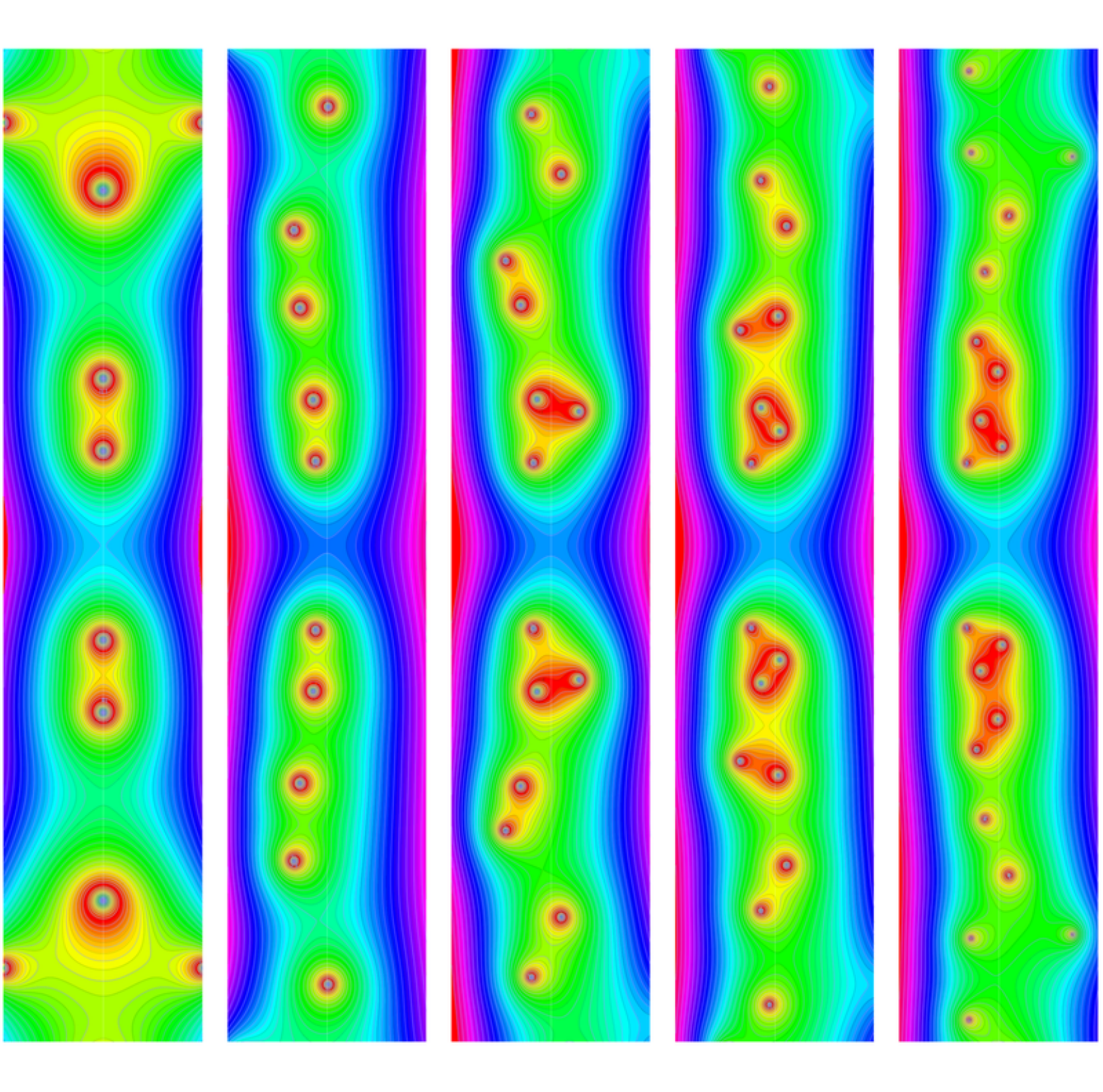}}
\label{Figure 18}
\caption{
The connection zeta functions of $G_5$, $G_6$, $G_7$, $G_8$ and $G_9$.
For one-dimensional complexes, we have a zeta functional equation $\zeta_G(s)=\zeta_G(-s)$.
This happens for $G_5$. The other complexes show typical random locations of the roots.
}
\end{figure}

\begin{figure}[!htpb]
\scalebox{0.2}{\includegraphics{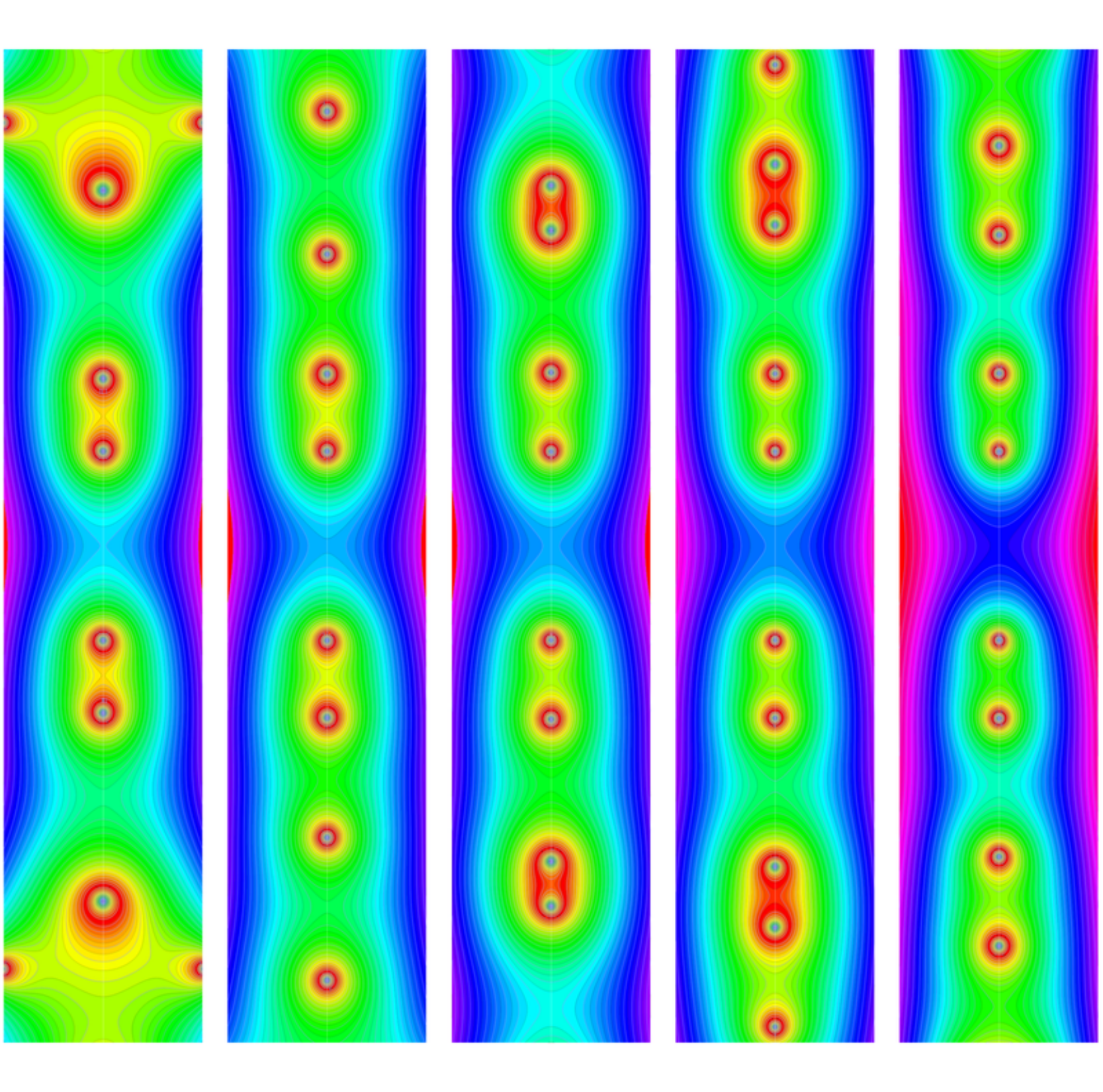}}
\label{Figure 19}
\caption{
The connection zeta functions of $C_5$, $C_6$, $C_7$, $C_8$ and $C_9$.
}
\end{figure}

\paragraph{}
We should point out that the Hodge Zeta function as well as connection Zeta functions for
circular graphs is already interesting. The circular case is harder \cite{KnillZeta}.
In the connection case we have a functional equation things and the limit is
more accessible \cite{DyadicRiemann}.

\paragraph{}
Here are again two main questions related to Wu characteristic and the spectrum.
What happens with the Wu characteristic and Wu cohomology of $G_n$?
We also do not know why the normalized spectrum of the Hodge Laplacian $\sigma(H(G_n))/n$
converges and why the curvature function of $G_n^+$ has a normalized limit.
The Hodge Laplacian in full generality satisfies the McKean-Singer super symmetry
\cite{McKeanSinger} also in the discrete as we had pointed out in \cite{knillmckeansinger}
and this also applies here. This symmetry implies that if there is a limiting
density of states  of the spectrum for the Laplacian, then it is the same when restricted
to even or odd dimensional discrete differential forms.

\begin{figure}[!htpb]
\scalebox{0.7}{\includegraphics{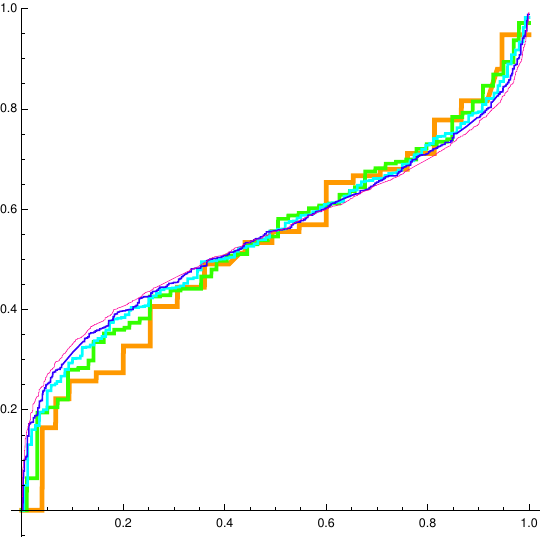}}
\label{Figure 20}
\caption{
The spectrum $\sigma(H(G_{n}))/n$ for the {\bf Hodge Laplacian} of
$n=8,10,\dots,18$ plotted in an order way so that the smallest eigenvalue $0$ is to the left
and the largest eigenvalue is to the right.
It appears that $\sigma(H(G_{n}))/n$ converges to a smooth or piecewise
smooth limiting function. We have not yet been able to prove this.
}
\end{figure}

\begin{figure}[!htpb]
\scalebox{1.0}{\includegraphics{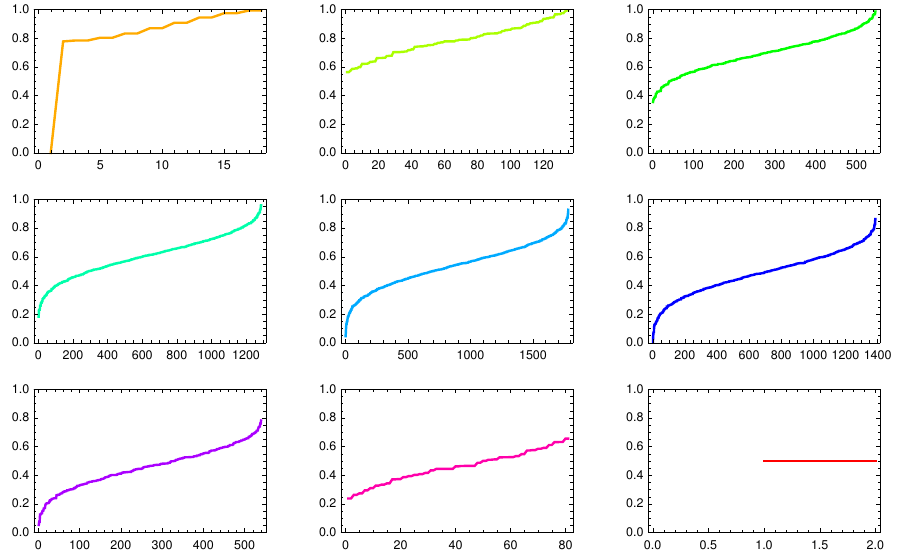}}
\label{Figure 21}
\caption{
Here is the spectrum of the Hodge Laplacian $H(G_n)$ for $n=18$ split up for the
different forms $H_0,\dots,H_8$. The Betti vector is $(1, 0, 0, 0, 0, 2)$. We deal with a
wedge sum of two 5-spheres. The super symmetry tells that the union of the non-zero
eigenvalues of the even forms $H_0,H_2,\dots,H_8$ is the same than the union of the
non-zero spectrum of the odd forms $H_1,H_3,\dots,H_7$. The $f$-vector is
$(18, 135, 546, 1287, 1782, 1386, 540, 81, 2)$, meaning for example that there are
only $17=18-1$ non-zero eigenvalues for $0$-forms. In the large $n$ limit, the
spectra from the forms in the middle will dominate. 
}
\end{figure}

\paragraph{}
{\bf Connection aspect}.
When looking not at incidences but intersections, some  connection 
linear algebra comes in. Since connection Laplacians are always invertible
we do not have to throw away the zero eigenvalue. See \cite{Unimodularity,CountingMatrix,
EnergizedSimplicialComplexes,EnergizedSimplicialComplexes2,EnergizedSimplicialComplexes3}.
Connection Laplacians of complexes encode in their spectrum
a lot of topological information even so we do not know all yet. 
The number of negative eigenvalues is the number of odd dimensional 
simplices in the complex. \\

\paragraph{}
The connection matrix $L$ of a simplicial complex $G$ is the matrix $L(x,y)=1$ if 
$x,y$ intersect and $L(x,y)=0$ else. The determinant depends on the number of odd
dimensional simplices in $G$. We have ${\rm det}(L(G_n)) = -1$ if $n = 6k,6k+1$
and equal to $1$ else.  We also have ${\rm det}(L(G_{n}^+)=-1$ for $n=6k+3,6k+4$ and
$1$ else. The number of positive eigenvalues minus the number of negative eigenvalues
of $L$ is the Euler characteristic \cite{HearingEulerCharacteristic}. As for the 
counting matrix \cite{CountingMatrix}, where $L(x,y)$ counts the number of simplices in 
$x \cap y$ we have positive definite quadratic forms $L$ which are isospectral to
the inverse matrix $L^{-1}$ which is also a positive definite quadratic form.  \\

\begin{figure}[!htpb]
\scalebox{0.7}{\includegraphics{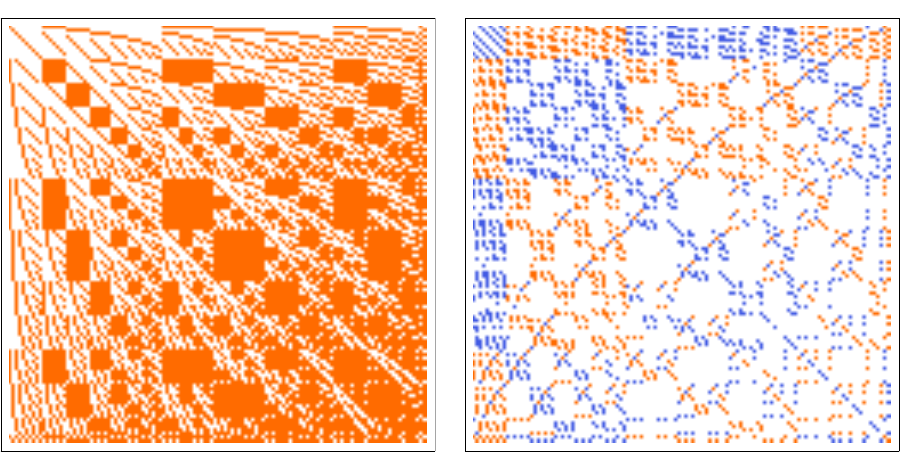}}
\label{Figure 22}
\caption{
The connection matrix $L$ for $G_{10}$ as well as its inverse (Green function matrix) $g=L^{-1}$, which is an integer
matrix too. The sum over all matrix entries of $g$ is $2$, the Euler characteristic of $G_{10}$ (which is a 2-sphere). 
}
\end{figure}

\begin{figure}[!htpb]
\scalebox{0.7}{\includegraphics{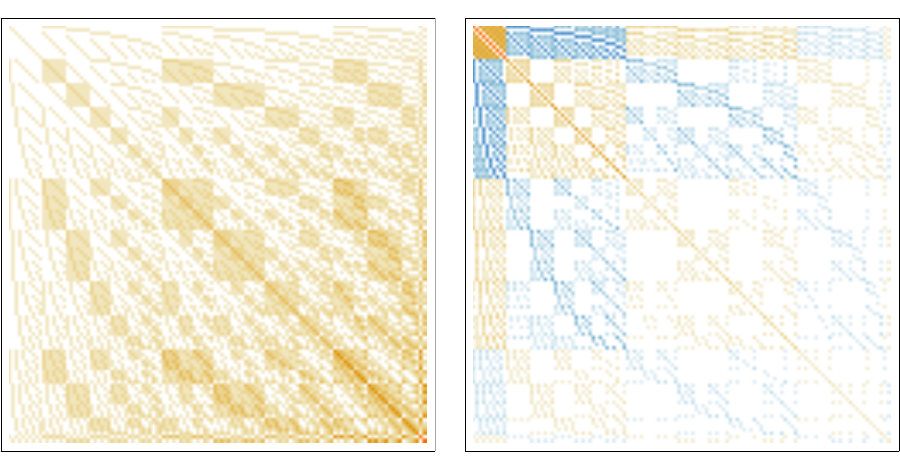}}
\label{Figure 23}
\caption{
The counting matrix $L$ for $G_{10}$ as well as its inverse $L^{-1}$, which is an integer
matrix too. The two matrices are isospectral positive definite integer quadratic forms. 
This holds for any simplicial complex $G$. The Whitney complex of $G_{10}$ is just an 
example. 
}
\end{figure}

\begin{figure}[!htpb]
\scalebox{0.7}{\includegraphics{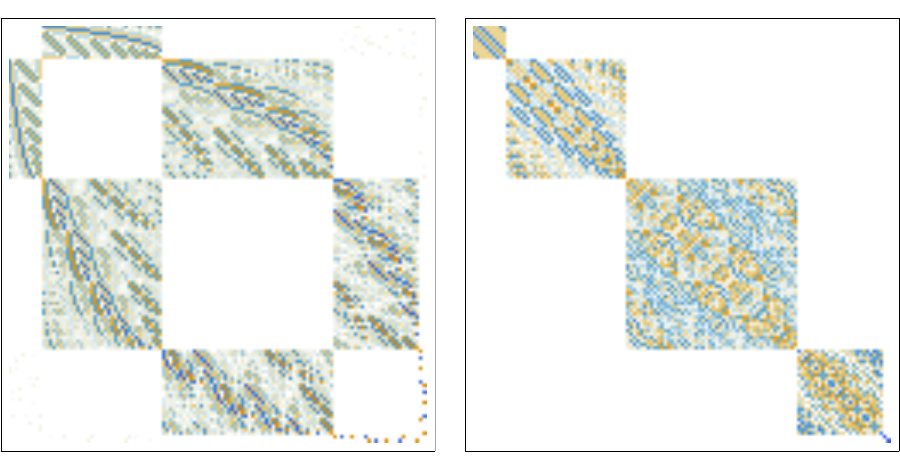}}
\label{Figure 24}
\caption{
The $\sin(D)$ of the Dirac matrix $D=d+d^*$ and the $\sin(D)$ of the 
Hodge matrix $H=D^2$ for $G_{10}$. They have the same size than the connection
matrices. The kernels of the blocks of $H$ have a cohomological interpretation. 
The block structure of these matrices is better seen when looking at the 
{\bf Schr\"odinger evolution} $e^{i D t}$ or $e^{i H t}$. The {\bf Schr\"odinger evolution }
for $D$ (including both positive and negative time) 
is d'Alembert equivalent to the {\bf wave evolution} $u_{tt} = - H u = -D^2 u$. 
for the Hodge Laplacian because of the factorization 
$(\partial_t-i D)(\partial_t + i D) = \partial_t^2+D^2$. 
}
\end{figure}

The Hodge Laplacian and the connection Laplacian are not always unrelated.
in \cite{Hydrogen} we pointed out that for 1-dimensional simplicial complexes,
the hydrogen identity $|H|=L-L^{-1}$ holds,
where $|H|=(|d|+|d|^*)^2$ is the sign-less Hodge Laplacian defined by the sign-less
incidence matrix $|d|$ and where $L$ is the connection Laplacian. This is useful
for estimating the Laplacian spectral radius $\rho$.


\paragraph{}
Many classical theorems in mathematics have discrete analogues which are
technically much simpler. There are already versions of Riemann-Roch available. 
We tried once to see how a discrete Atiyah-Singer theorem would look like
\cite{DiscreteAtiyahSingerBott} by using the McKean-Singer spectral symmetry 
as an analogue replacing technical assumptions like elliptic regularity. The
reason is that McKean-Singer spectral symmetry is really what one needs for
having a notion of analytic index. Atiyah-Singer generalizes Gauss-Bonnet, 
Athiah-Bott generalizes Brouwer-Lefschetz.
An example of such an ``elliptic  complex" is the Wu cohomology for 
Wu characteristic 
$$ \omega(G) = \sum_{x \sim y} \omega(x) \omega(y) \; $$
for a finite abstract simplicial complex, where $\omega(x) = (-1)^{{\rm dim}(x)}$ 
and $x \sim y$ means that the intersection $x \cap y$ is not empty. 
Wu characteristic is one of many multi-linear valuations \cite{valuation}.
Multi-linear valuations, like valuations have a homogenity condition in that 
the values for isomorphic graphs is the same. This allows to define the valuation
by giving the value $X_k$ on $k$-dimensional simplex tuples. In the case of 
generalized Wu characteristic, one has the value $\omega(x_1) \cdots \omega(x_k)$ 
for intersecting tuples $x_1, \dots, x_k$ of simplices. 
This can be vastly generalized by looking more generally at a function $h: G \to \mathbb{A}$
\cite{EnergizedSimplicialComplexes2,EnergizedSimplicialComplexes},
where $A$ is a ring (this is a frame work which sometimes is also considered in 
frame works like interaction models \cite{Biggs1977}).

\paragraph{}
{\bf Statistical aspect}.
Statistical questions come in when asking what happens typically? 
In our case, we can look what typical Euler characteristic, dimension, or simplex
cardinality we get for the graph complement of
a one dimensional graph with $n$ nodes. We would like to know what Betti numbers
can occur or what Betti numbers occur most likely. What Euler characteristic
what dimension appears typically \cite{randomgraph}? What size of the simplices 
\cite{AverageSimplexCardinality}. \\

\paragraph{}
Then there is the geometric probability or integral geometric 
aspect, where one averages over probability spaces of external quantities. 
One can see curvature as an expectation of indices \cite{indexexpectation,
colorcurvature,indexformula}.
We see for example that the graph complements of most 
trees appear to be contractible (their cohomological dimension is $0$). 
The cases where we have spheres is less frequent and one could suspect that in 
the limit $n \to \infty$ the probability of getting 
contractible graph complements is one.  

\paragraph{}
{\bf Dynamical aspect}.
There is a dynamical theme. What happens with 
quantum mechanical evolutions or isospectral 
deformations on a graph.  
See \cite{IsospectralDirac,IsospectralDirac2}.
While these features work for any graph, the dynamics
could become more interesting in special cases
like for example when we deal with graph complements
of cyclic graphs. 

\paragraph{}
Dynamics also comes in when looking at automorphisms.
This is especially interesting if $G$ is a Cayley graph 
of a group $A$.  We have seen the case of a Cayley graph
with the group $\mathbb{Z}_n$, where we take the complement
of the natural generators $T(x)=x+1$ and $T^{-1}(x)=x-1$. 
An other example is the dihedral group $\mathbb{D}_n$ 
the symmetry group of the $n$-gon which is non-abelian. 
We looked at the topology of the complement graphs of the 
natural Cayley graphs. For $n=4$, we look at the complement
graph $H_4$ of the cube graph. Let us call $H_n$ these graphs.
There seems to be a parallel story in that now these graphs
are homotopic to either spheres or wedge products of 3 spheres.
The later happens for $n=4d$.

\begin{figure}[!htpb]
\scalebox{0.4}{\includegraphics{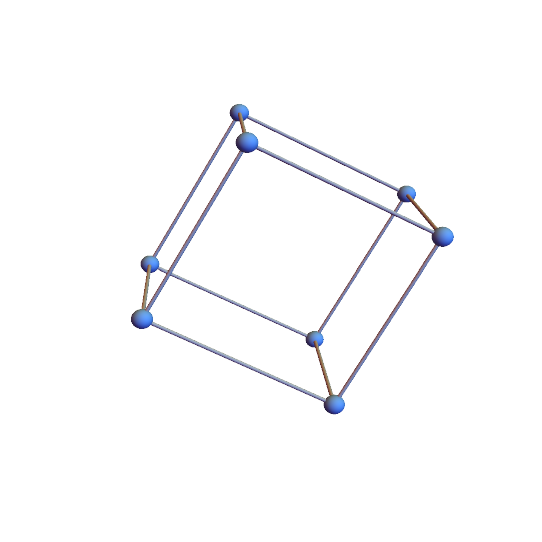}}
\scalebox{0.4}{\includegraphics{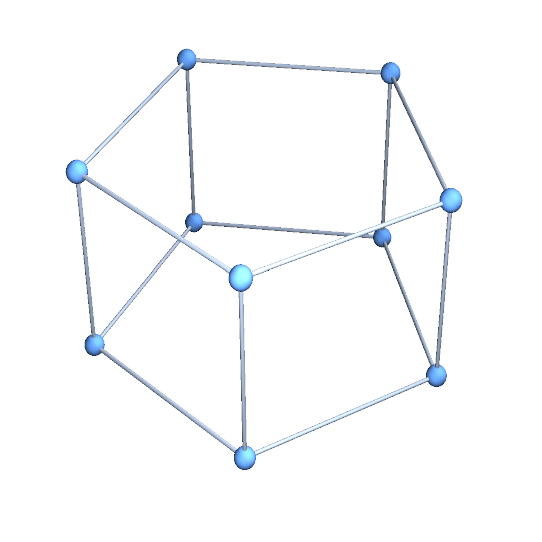}}
\scalebox{0.4}{\includegraphics{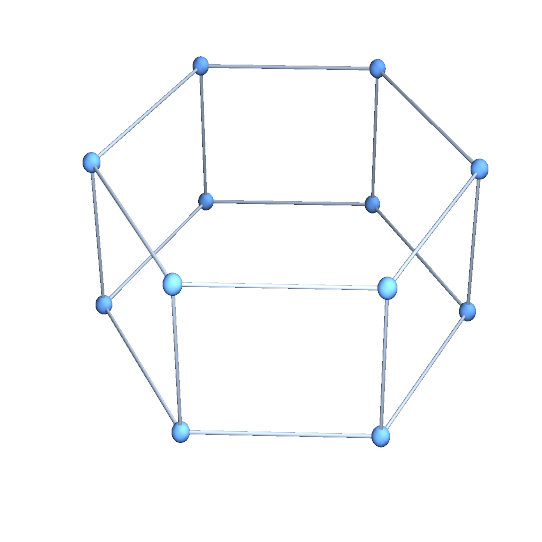}}
\scalebox{0.4}{\includegraphics{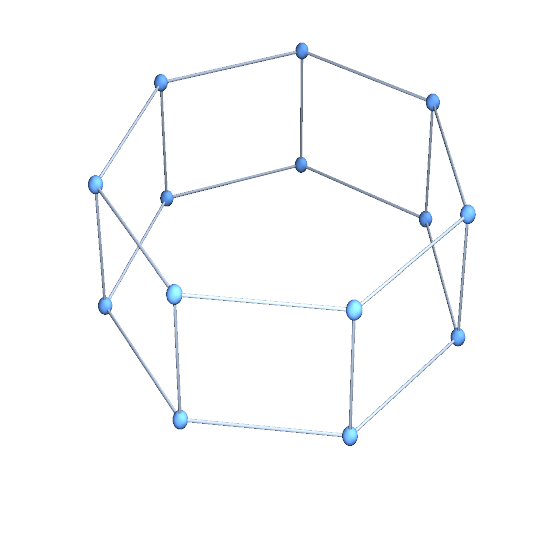}} 
\label{Figure 25}
\caption{
Cayley graphs of dihedral groups. Graph complements to 
such graphs are either spheres or a wedge sum of 3 spheres. 
}
\end{figure}

\begin{center}
\begin{tabular}{c|c|l}
$H_n$ for n= &   $\chi(G)$   & $\vec{b}(G)$  \\ \hline
$4$  & $-2$    &  $(1,3,0,0)$       \\
$5$  & $ 0$    &  $(1,1,0,0)$       \\
$6$  & $ 2$    &  $(1,0,1,0,0)$      \\
$7$  & $ 0$    &  $(1,0,0,1,0,0)$    \\
$8$  & $-2$    &  $(1,0,0,3,0,0,0,0)$  \\
$9$  & $ 0$    &  $(1,0,0,1,0,0,0,0)$  \\
$10$ & $2$    &   $(1,0,0,0,1,0,0,0,0,0)$ \\
$11$ & $0$     &  $(1,0,0,0,0,1,0,0,0,0)$ \\
$12$ & $-2$    &  $(1,0,0,0,0,3,0,0,0,0)$ \\
\end{tabular}
\end{center}

In the {\bf dihedral case}, where we look at the graph complements of the Cayley 
graph of dihedral groups, the computation of the f-function is 
also recursive but a bit more complicated:
$$   f_n(t) = f_{n-1}(t) + t f_{n-1}(t) + x f_{n-2}(t) + (-1)^n t^n \; . $$

\paragraph{}
{\bf Positive curvature aspect}
A complete classification of positive curvature manifolds is still missing. 
All known cases admit a effective continuous group of effective 
symmetries. Such a symmetry must be a Lie group \cite{Kobayashi1972}. A reduction
theory sees the {\bf fixed point manifold} $N$ of this group action 
is a union of smooth positive curvature manifolds, where each has 
even co-dimension. All known positive curvature manifolds admit a 
circle action. Circulant graphs come already with a symmetry and it is natural
to try to use them for constructing positive curvature manifold. There are some
interesting consequences coming directly from the reduction. See 
\cite{GroveSearle2020} on a theorem of Grove and Searle \cite{GroveSearle} 
and references.

\paragraph{}
What other manifolds can be obtained like this?
For the graphs $G_n$ one can not identify points to get from $\mathbb{S}^d$ to 
$\mathbb{P}^d$. However, one can look at Barycentric refinements
and then take the quotient. One gets so discrete geometries for which 
one has a circular symmetry but the Euler curvature must be averaged
to become positive. Barycentric refinements in general already introduce
negative curvature. The Barycentric refinement of an icosahedron graph
for example has curvature $1-10/6$ at some points of vertex degree
$10$. The curvatures still add up to 2. After an identification of a 
polar map one can get models of the projective plane and now the curvatures
add up to $1$. In the spirit of what we have done here, it is natural to 
look for larger classes of models and circulant graphs are natural. 

\paragraph{}
The quest to push computation of cohomology groups further is
also motivated by the story of positive curvature. We would like to understand
the cohomology of positive curvature manifolds better because we have a 
strange (possibly only mathematical) affinity between positive curvature
manifolds and Bosonic matter in physics \cite{PosCurvBosons}. The heavy 
bosons are related to positive curvature manifolds which have more 
elaborate cohomologies. 

\begin{figure}[!htpb]
\scalebox{0.6}{\includegraphics{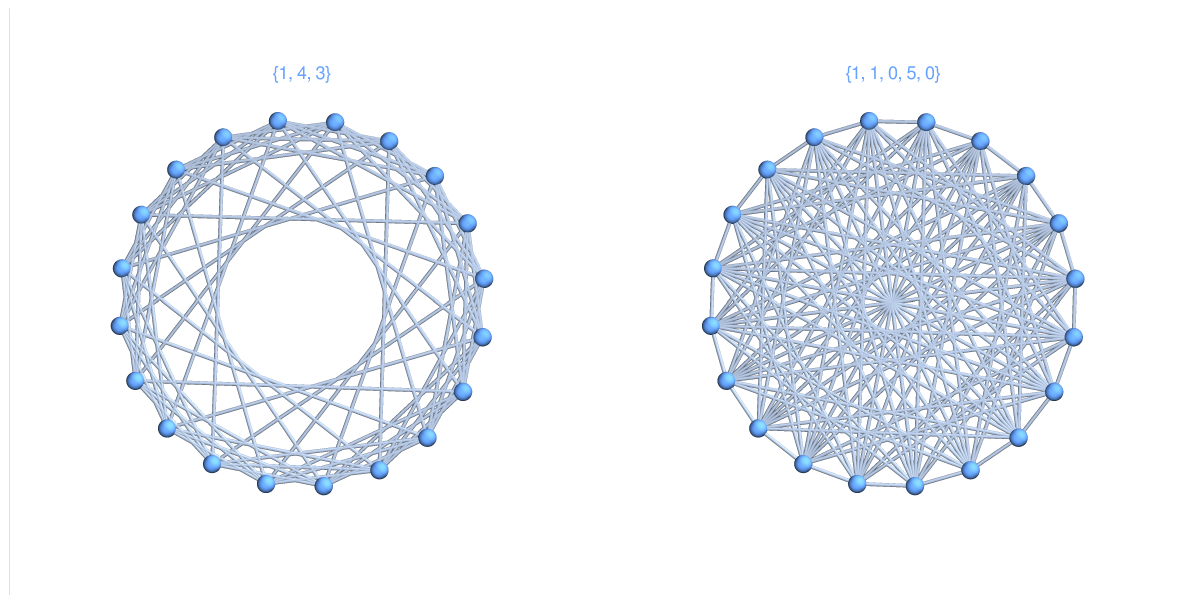}}
\label{Figure 26}
\caption{
A circulant graph $G=C_{20}^{\{2,3,4,7,13\}}$ and its graph dual $\overline{G}$.
The Betti vector of $\overline{G}$ is $b=(1, 1, 0, 5, 0)$. 
A natural question is which cohomologies can occur for circulant graphs. 
}
\end{figure}

\paragraph{}
{\bf Number theoretical aspects}.
The are many questions open like how many fixed points the automorphisms of $G_n$ have
of each dimension, then about the factorization of the polynomials $f_n$. We see that
$f_{4n+2}$ always have a factor $1+2x$ which indicates to be a suspension. But it is not
necessarily. For $G_5$ where $f_6(t) = 2t^3+9 t^2+6t+1 = (2 x+1) (x^2+4 x+1)$ we
get an f-function which is the same than of the join of $K_2+P_2$. 
We also see experimentally that we can always modify the largest power of $f_n(x)$
by adding $a x^m$ with $|a| \leq 2$ so that we have a factorization and that this $a=a(n)$
is $8$-periodic.

\paragraph{}
A class of circulant graphs that is interesting in number theory is
are {\bf quadratic residue graphs} or {\bf Paley graphs} $QR(q)$ defined 
by a prime power $q=p^r$ with $q =1 ({\rm mod} 4)$. Paley graphs are self-complementary. 
The edges are all pairs $(a,b)$ such that $(a-b)$ is a square (a quadratic residue). 
For $q=5$ one has the squares $1,4$ so that we have $C_5$. The next one is
$QR(13)$ which is actually a 2-dimensional torus and a 2-graph. For $2$-graphs, 
the curvature has been been defined already early in graph theory in the context
of graph coloring and is $K(v)=1-{\rm deg}(v)/6$. The Paley graph $QR(13)$ is remarkable
as it is a discrete analogue of the {\bf Clifford torus}. It is {\bf flat}. The 
curvature is constant $0$ everywhere. For the next ones $QR(17),QR(29),QR(37),QR(41),
QR(53)$ we always see constant curvature $1$ and $\chi(QR(q))=q$. But then 
$QR(61)$ has constant curvature $-14$, is not geometric because the unit spheres
have $\chi(S(x))=70$ and constant curvature $7/3$ and Betti vectors $(1,0,70,1)$. 

\begin{figure}[!htpb]
\scalebox{0.6}{\includegraphics{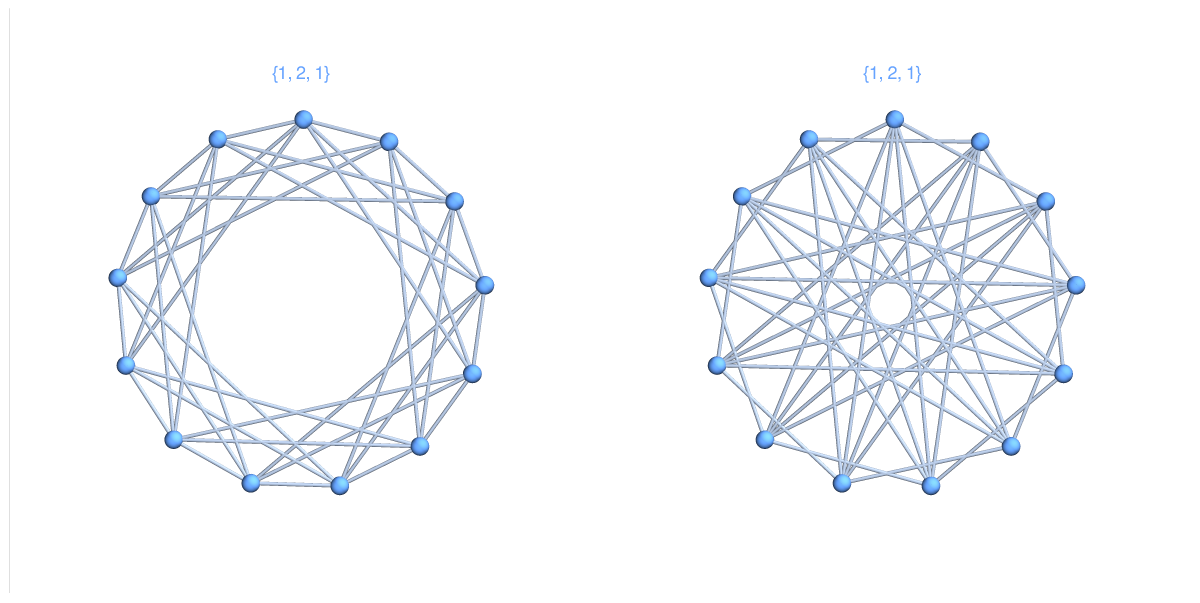}}
\label{Figure 27}
\caption{
The Paley graph $G=QR(13)$ is a self-complementary 2-torus: its dual $\overline{G}$
is isomorphic to $G$. The graph $G$ is a discrete manifold in the sense that every unit sphere is a
circular graph with 6 vertices. It is {\bf flat} because the curvature is zero
everywhere. This graph is a small (with 13 vertices rather than the usual 16 vertices)
and highly symmetric implementation of $\mathbb{T}^2$
and a discrete analogue of the {\bf Clifford torus}. All Paley graphs are 
{\bf constant curvature graphs} due to the symmetry. 
}
\end{figure}

\paragraph{}
A natural question is to look for graphs $QR(q)$ which are d-manifolds 
in the sense that all unit spheres 
\footnote{The term ``neighborhood" in graph theory is misleading as it suggests the
midpoint to be included if the graph is a metric space. We would call the later a unit ball.
Unit balls as pyramid extensions of unit spheres are always contractible.}
are $(d-1)$ spheres, where $d$-spheres
are $d$-manifolds with the property that removing a vertex produces a 
contractible graph. This is a notion which agrees with other notions of 
spheres like in discrete Morse theory \cite{knillreeb}.
So far, only $Q(5),Q(13)$ have turned out to be manifolds. 
Also the Lefschetz story needs more investigations. In the torus case $G=QR(13)$
each of the 13 involutions $T$  in the automorphism group all have Lefschetz number
$\chi(G,T)=4$ while all 13 translations $T$ in ${\rm Aut}(G)$ have $\chi(G,T)=0$. 
The average Lefschetz number $\chi(G,{\rm Aut}(G))$ is therefore $2$ in this case. 
In the cyclic case $G=QR(5)$, we have like in all $C_n$ cases the average Lefschetz
number $1$. This brings us to the still vexing case $G_{11},G_{13}$ corresponding to 
$3$-spheres and $G_{23},G_{25}$ corresponding to $7$ spheres, where all Lefschetz
numbers are zero.

\paragraph{}
{\bf Homotopy aspect}.
Homotopy deformations allow to smooth out a discrete structure can help to make a geometric object
more accessible and more symmetric. The prototype is $G_6$ which is homotopic to a wedge sum of
spheres but which is a regular graph with $\mathbb{Z}_6$ vertex transitive symmetry. This surprises
given that a lemniscate has a singularity.  We have given examples of graphs homotopic to $S^d$
with $n=3d+1$ or $3d-1$ vertices which admit a transitive (ergodic) graph automorphism and
even more remarkably can do with $n=3d$ vertices realize a graph homotopic to $S^d \wedge S^d$
but having also a transitive automorphism. The automorphism groups of $C_n$ and $G_n$ are both the
dihedral group with $2n$ elements.
While the graphs $G_n$ are inhomogeneous in general, they
implement remarkably symmetric and homogeneous simplicial complexes for $S^d$ or $S^d \wedge S^d$.
For $G_6$ for example which is homotopic to $S^1 \wedge S^1$, the figure 8 implementation
costs 7 vertices. We need to glue together two circles along a simplex to reduce this to $6$
vertices which becomes regular by ``padding" the graph with additional vertices.
The homotopies can help to ``round" a graph and make it more smooth. There is computational
advantage if one has a vertex transitive graph.

\paragraph{}
Here is something which has been a motivation throughout these investigation but where
we have not yet gone further yet. We would like to get a tool to compute with a computer as many
homotopy groups as possible.  It is interesting that the wedge sum of arbitrary high dimensional
spheres appears here. It is natural as the sum of two sphere homotopies of $f,g$
is defined using the projection $\Psi$ from $S^d$ to $S^d \wedge S^d$.
We have $f \wedge g$ on $S^d \wedge S^d$ leading to $f+g = f \wedge g \Psi$.
The wedge sum also appears in the {\bf smash product}
of two graphs is $X \smash Y = X \times Y/X \wedge Y$.

\paragraph{}
{\bf Group theme.}
Since the graph $G_{11}$ is homotopic $\mathbb{S}^3$ which carries a Lie group structure, one can ask
whether the graph $G_{11}$ carries an interesting finite group structure. But there
is only one finite group of order $11$, the cyclic group $\mathbb{Z}_{11}$.
The case of $\mathbb{S}^3$ has been historically important as when the manifold
is written as a Lie group $SU(2)$, it acts by rotations on $\mathbb{S}^2$ leading
to the Hopf fibration $f: \mathbb{S}^3 \to \mathbb{S}^2$ which is the generator
of the third homotopy group $\pi_3(\mathbb{S}^2) = \mathbb{Z}$ of the 2-sphere.
Since when performing the addition of the homotopy groups the wedge sum
$\mathbb{S}^d \wedge \mathbb{S}^d$ and its ramified cover $\mathbb{S}^d$ as well
as the group action are given by very symmetric structures, we have had some hope
at first to use these graphs as numerical tools to investigate the homotopy groups
of spheres.

\paragraph{}
{\bf Duality theme.} 
{\bf Duality} plays an important role in general in mathematics. There are a few notions
which are relevant in graph theory: Alexandrof duality, graph complements, 
projective duality, then Poincar\'e Duality and its related Dehn-Sommerville symmetry. 
{\bf Alexandrof duality} is a notion for simplicial complexes $G$ where one looks at the
dual complex generated by the complement sets. It has has the feature that the cohomology 
groups of $G$ determine the cohomology groups of the dual. The combinatorial version is due
to Kalai and Stanley.  The {\bf graph complement} on the other hand change the cohomology a lot. 
We have seen that the edge refinement $C_n \to C_{n+1}$ is dual to a cube root of suspension
when looking at homotopy classes because $G_{n+3}$ is homotopic to a suspension
of $G_n$. One can wonder whether this does generalize but composing edge refinements 
for general does in general not lead to suspensions.
One of the oldest dualities are related to the symmetry in projective spaces replacing
lines with hyperplanes etc. For 2-polyhedra, it resulted in switching facets with points.
It is actually a symmetry among CW-complexes. The dodecahedron for example has 12 cells
attached which do not belong to the Whitney complex generated by the graph. 
In higher dimensions, it can become more ambiguous what we want to call cells 
\cite{Gruenbaum2003} so that the theory of polyhedra often refers to convex 
polytopes in Euclidean space like \cite{Ziegler}.

\paragraph{}
Finally, one has {\bf Poincar\'e Duality} which is a palindromic symmetry of the Betti vector. 
Discrete orientable manifolds have this property like d-spheres $b=(1,0,\dots, 0,1)$ or 
tori which have as Betti vectors the rows $(B(d,0),B(d,1),\dots,B(d,d))$ of the Pascal 
triangle with Binomial coefficients $B(d,k)$. 
An example is the 2-torus with Betti vector $b=(1,2,1)$. Non-orientability destroys this symmetry 
already. Related is the amazing Dehn-Sommerville symmetry which defines a class of complexes
resembling spheres. It is the symmetry $f_G(t)+(-1)^d f_G(-1-t)=0$ and implies $\chi(G) = 1+(-1)^d$
like for spheres. Examples are even dimensional Bouquet of spheres \cite{dehnsommervillegaussbonnet}.
The functional Gauss-Bonnet theorem mentioned above immediately establishes the unit spheres
of Dehn-Sommerville spaces are Dehn-Sommerville. 
The symmetry also can happen for non-manifolds and suggests to look at generalized manifolds
as simplicial complexes for which all unit spheres are Dehn Sommerville complexes of the same 
dimension similarly as discrete manifolds are defined as complexes, where all unit spheres are
spheres. 

\section{Homotopy manifolds}

\paragraph{}
More general notions of ``manifolds" can be obtained by using homotopy to ``soften" the
rigidity that comes from requiring the unit spheres in the graph to be spheres. Of course, this has to be done in a
way that locally in such a universe, an observer still sees ``spheres" for small geodesic spheres. 
One can now define a {\bf homotopy manifold} as a complex (or Whitney complex of a graph to stay
within graph theory) for which all unit spheres are {\bf homotopy spheres} or contractible. This allows the 
manifold to be ``thickened" and still get the right dimension. 
The Moebius strip $G_7$ is an example where all unit spheres are contractible. It is still
not contractible itself. The graphs $G_{3d-1}$ are 
models of $d$-spheres which are homotopy spheres. An example is $G_{11}$ the homotopy 3-sphere
we were considering in the introduction. Its unit spheres are $G_{8}^+$ which are homotopic
to $2$-spheres but which are of the weaker type as their unit spheres are either $1$-spheres
or contractible. 

\paragraph{}
It is informative to experiment with notions of ``manifold" which are more close to 
what takes place if we look at data or computer implementations of Euclidean spaces, 
where rounding errors can happen when storing coordinate data which can come with 
errors or when looking at non-standard models of Euclidean space which defines 
from a compact d-manifold $M$ a graph by taking an $\epsilon>0$ 
(nonstandard small = infinitesimal in the terminology of IST \cite{Nelson77}), 
a finite set $V$ of points in $M$ which have the property that the $\epsilon$ cover
(by geodesic balls) centered at $V$ covers $M$ and where two vertices are connected if
their distance is smaller than $2\epsilon$. We call this graph a discrete model of $M$ if it
has has the same homotopy type than $M$ and if unit spheres are either contractible or 
homotopic to spheres.

\paragraph{}
A {\bf homotopy $d$-manifold model} $G$ of a compact d-manifold $M$ as a 
finite simple graph such that all unit spheres are either contractible or homotopy $d-1$-spheres. 
A homotopy $d$ sphere $G$ is a $d$-manifold which has the property that it is not contractible 
and such that there is a vertex such that $G$ with the unit ball of this vertex removed renders it contractible. 
If all unit spheres are contractible and $G$ is a homotopy $d$-sphere already, then it is a homotopy
$d$-manifold.  This can happen even if all unit spheres are contractible.
The spaces $G_{3d+1}$ illustrates this. The spaces $G_{3d},G_{3d \pm 1}$ are homotopy $d$ manifolds with 
this notion. Discrete $d$-manifolds or discrete $d$-manifolds with boundary are of course 
homotopy $d$-manifolds. For $d$-manifolds, all unit spheres are spheres. For 
$d$-manifolds with boundary, the unit spheres are either spheres or contractible.  
It makes sense to look at a notion of {\bf homotopy dimension} which is $d$ if $G$ is a homotopy sphere
or if all unit spheres are either contractible or $d-1$ homotopy spheres. 
It is different from the cohomological dimension which is the maximal $d$ for which the Betti number $b_d$ is not zero.
For a discrete projective plane $G=\mathbb{P}^2$ for example, the cohomological dimension is $0$ because the Betti
vector is $(1,0,0)$. But $G$ is a discrete $2$-manifold so that the homotopy dimension is $2$. 
There is an implementation of the projective plane with $15$ vertices which in which which all unit spheres have either $5$
or $6$ vertices and the curvature of the 6 vertices of vertex degree $5$ is $1/6$ each. 
A discrete Moebius strip like $G_7$ is an example of a homotopy $1$ manifold as it is already a sphere. 
But it is an example where all unit spheres are contractible. Taking away such a unit sphere or more
generally a contractible part from $G_7$ changes the topology. 

\paragraph{}
The notion homotopy d-manifold notion is a realistic model in applied situations.
For example, when we realize a 2-manifolds in nature in the form of a surface, 
they are not actual 2-manifolds because every surface in nature has a thickness. 
But they are still also are not just 3-dimensional solids.
If we take a surface like a plastic bag or a soap bubble and rip a hole somewhere, 
the boundary of that hole is a unit sphere is a homotopy 1-sphere. 
Despite the fact that the plastic bag actually has a tiny thickness, it is a 2-dimensional
surface and it is the homotopy of the unit spheres in the bag which determines this. 
While there are points on the plastic
bag at the boundary where the unit spheres are contractible but at most points, puncturing the 
bag produces a topology change in that a new one dimensional circular boundary is introduced.
When we deal with objects which are fatter like a cup, this becomes three dimensional
because we can in principle drill a tiny cavity into the cup and have a 2-sphere
boundary there. So, the plastic bag is a two-dimensional manifold but the cup is a three-dimensional
manifold with boundary and the homotopy types of the unit spheres determines that. 
This is realistic because we have fundamental limits of thicknesses of surfaces 
visible already in nano-technology. There are surfaces with the thickness of one atom 
(examples are bucky balls).  These surfaces are not 3-dimensional in nature because a unit sphere
is a 1-sphere and not a 2-sphere. 

\paragraph{}
We should also point out that the just invoked notion of $d$-manifold 
invokes also some sort of dimension which goes beyond
homotopy alone. A wheel graph is contractible and so homotopic to a point. Its cohomological dimension is $0$. 
But it has the property that all unit spheres are either 1-spheres or contractible. It is therefore considered
a homotopy 2-manifold. The strong product of the wheel graph with $K_n$ is still a homotopy $2$-manifold. 
This is not true any more for the strong product of a wheel graph with the linear graph $L_2$ of length $2$. 
Similarly as $K_n$ are $0$ dimensional for $n \geq 1$ and path graphs $L_n$ of 
length $2$ or higher are $1$-dimensional (there are vertices for which the unit sphere is $0$-dimensional), 
the product of the Wheel graph and $L_n$ is then a three dimensional manifold. 
Let us mention our earlier attempts to capture notions of homeomorphisms for graphs
\cite{KnillTopology} which can be used completely within graph theory. There, we have suggested
to look for a ``nerve" subgraph in in a manifold and consider two graphs homeomorphic if their
nerve graphs agree. The dimension notion then is used in the nerve graph. This notion is harder to 
check and it is better to loosen up a bit and allow for more general homotopy deformations, still 
keeping a notion of dimension which is realistic. We still think that \cite{KnillTopology} is valuable
especially in case where the graphs model fractals in the continuum. What we look at there is more
manifold like and so closer to differential geometry. 

\paragraph{}
Coming back to the {\bf floating point arithmetic model} of an interval, this is clearly
a one dimensional homotopy manifold because floating point arithmetic honors the Archimedean property of the real line 
if $x<y$, then any implementation of an operation $f$ which is classically monotone ($f'(x)>0$) 
satisfies also $f(x)<f(y)$ in a floating point arithmetic implementation. This forces the 
existence of {\bf threshold} points, where the unit sphere $S(x)$ of the finite model which are no more contractible
because there are point $y \in S(x)$ which are either strictly smaller or strictly larger than $x$. In other words,
the unit sphere $S(x)$ is in general either a $0$-sphere or contractible. Graphs which are discrete homotopy $d$-manifolds
behave like d-manifolds in many respect. It prevents to have singularities or have different type of 
dimensions together. A star graph with $3$ or more leaves is not in a one dimensional discrete manifold
because there is a unit sphere which is not a sphere. 

\paragraph{}
Data always come with error margins. If we take such a data model of a manifold
like a computer implementation which tells from a point whether it is in the manifold or not, then 
we have points where the unit spheres are contractible. These are {\bf boundary points}. The other
points are interior points where the unit spheres are spheres. In a computer implementation of real arithmetic,
we use floating point implementations of a real number $x$ which when increased arbitrarily changes
how the number is implemented. This can be subtle and has consequences like that floating point arithmetic is 
not distributive: iterating the logistic map $f(x)= 4x(x-1)$ or $g(x) = 4x^2-4x$ gives completely different 
results after 60 iterations with standard floating point arithmetic using 17 digits. We see that $f^n(x)-g^n(x)$
is of macroscopic order already after $n=60$ iterations even so $f^n$ and $g^n$ are formally the same.
If we use machine accuracy with $K$ decimal digits we start to get deviations after 
$\log_2(10^K)$ iterations because Lyapunov exponent of the interval map is $\log(2)$. Because the computer has a
finite memory, the iteration is in reality implemented as a map on a finite set $V$. Every computer implementation
of such a laboratory defines a graph $G=(V,E)$ where $E$ is the set of pairs $(a,b)$ of points in 
$V$ for which the outcome remains the same. This graph $G$ is in $\mathcal{X}_1$. It is homotopic to an 
interval $[0,1]$ and every unit sphere $S(x)$ of a point is either contractible or then a $0$-sphere (consists of 
two contractible sets in $G$). In the allegory of the cave of Plato, the shadows we see of the real 
objects are blurred, already in mundane situations like real numbers dealt with in a computer.
Objects like ``points" are idealized notions which can also be described as ``contractible shapes". 
We hope to have shown that also notions of differential geometry like curvature or notions like 
Lefschetz fixed point theory are interesting for such shapes  and that there are surprises like
universality in the curvature and periodicity in the Lefschetz numbers.

\bibliographystyle{plain}

\end{document}